\documentclass[reqno,a4paper,twoside]{amsart}

\usepackage[a4paper,margin=0.943in,footskip=0.25in]{geometry}

\usepackage{enumitem}

\setenumerate{label=\textnormal{(\arabic*)}}

\allowdisplaybreaks[4]

\usepackage{amsmath,amssymb,dsfont,verbatim,bm,geometry,mathrsfs}

\usepackage{mathtools}
\usepackage[raggedright]{titlesec}

\usepackage[utf8]{inputenc}
\usepackage[T1]{fontenc}

\usepackage{hyperref}
\hypersetup{
	colorlinks=true,
	linkcolor=blue,
	filecolor=magenta,
	urlcolor=cyan,
	unicode=true,
	pdfpagemode=FullScreen,
}

\titleformat{\chapter}[display]
{\normalfont\huge\bfseries}{\chaptertitlename\\thechapter}{20pt}{\Huge}
\titleformat{\section}
{\normalfont\Large\bfseries\center}{\thesection}{1em}{}
\titleformat{\subsection}
{\normalfont\large\bfseries}{\thesubsection}{1em}{}
\titleformat{\subsubsection}
{\normalfont\normalsize\bfseries}{\thesubsubsection}{1em}{}
\titleformat{\paragraph}[runin]
{\normalfont\normalsize\bfseries}{\theparagraph}{1em}{}
\titleformat{\subparagraph}[runin]
{\normalfont\normalsize\bfseries}{\thesubparagraph}{1em}{}

\titlespacing*{\chapter} {0pt}{50pt}{40pt}
\titlespacing*{\section} {0pt}{3.5ex plus 1ex minus .2ex}{2.3ex plus .2ex}
\titlespacing*{\subsection} {0pt}{3.25ex plus 1ex minus .2ex}{1.5ex plus .2ex}
\titlespacing*{\subsubsection}{0pt}{3.25ex plus 1ex minus .2ex}{1.5ex plus .2ex}
\titlespacing*{\paragraph} {0pt}{3.25ex plus 1ex minus .2ex}{1em}
\titlespacing*{\subparagraph} {\parindent}{3.25ex plus 1ex minus .2ex}{1em}

\usepackage{tikz}
\usetikzlibrary{matrix, shapes}

\usepackage{float, subfig}
\usepackage{amsrefs}

\usepackage{titletoc}

\contentsmargin{2.55em}
\dottedcontents{section}[1.5em]{}{2em}{1pc}
\dottedcontents{subsection}[4.35em]{}{2.8em}{1pc}
\dottedcontents{subsubsection}[7.6em]{}{3.2em}{1pc}
\dottedcontents{paragraph}[10.3em]{}{3.2em}{1pc}

\makeindex

\newtheorem{theorem}{Theorem}[section]
\newtheorem{lemma}[theorem]{Lemma}
\newtheorem{proposition}[theorem]{Proposition}
\newtheorem{corollary}[theorem]{Corollary}

\theoremstyle{definition}
\newtheorem{definition}[theorem]{Definition}

\newtheorem{notation}[theorem]{Notation}
\newtheorem{example}[theorem]{Example}
\newtheorem{examples}[theorem]{Examples}

\theoremstyle{remark}
\newtheorem{remark}[theorem]{Remark}

\usepackage{tikz}
\usetikzlibrary{matrix,shapes}
\usetikzlibrary{cd}

\DeclareMathOperator{\Aut}{Aut}
\DeclareMathOperator{\End}{End}

\DeclareMathOperator{\cop}{cop}

\DeclareMathOperator{\ide}{id}

\DeclareMathOperator{\Soc}{Soc}
\DeclareMathOperator{\LSoc}{LSoc}
\DeclareMathOperator{\RSoc}{RSoc}

\DeclareMathOperator{\op}{op}

\DeclareMathOperator{\Hom}{Hom}

\newcommand{\mli}[1]{\mathit{#1}}
\newcommand*{\dobleast}{\raisebox{0ex}{*}\llap{\raisebox{-1.3ex}{*}}}
\newcommand{\dast}{\mathchoice
{\mathbin{\raisebox{.18ex}{\scalebox{.60}{\dobleast}}}}
{\mathbin{\raisebox{.18ex}{\scalebox{.60}{\dobleast}}}}
{\mathbin{\raisebox{.136ex}{\scalebox{.42}{\dobleast}}}}
{\mathbin{\raisebox{.090ex}{\scalebox{.30}{\dobleast}}}}
}

\newcommand*{\doblediamond}{\raisebox{0ex}{$\diamond$}\llap{\raisebox{-1.3ex}{$\diamond$}}}
\newcommand{\ddiam}{\mathchoice
	{\mathbin{\raisebox{.62ex}{\scalebox{.60}{\doblediamond}}}}
	{\mathbin{\raisebox{.62ex}{\scalebox{.60}{\doblediamond}}}}
	{\mathbin{\raisebox{.45ex}{\scalebox{.42}{\doblediamond}}}}
	{\mathbin{\raisebox{.3ex}{\scalebox{.30}{\doblediamond}}}}
}

\newcommand{\xcirc}{\hspace{0.9pt}}

\newcommand{\diam}{\diamond}

\newcommand{\Coalg}{\textbf{Coalg}}

\newcommand\mdoblemas{\ensuremath{\mathbin{+\mkern-10mu+}}}
\newcommand\mdoubletimes{\ensuremath{\mathbin{\times\mkern-10mu\times}}}
\newcommand{\ov}{\overline}
\newcommand{\ot}{\otimes}
\newcommand{\wh}{\widehat}
\newcommand{\wt}{\widetilde}

\newcommand{\hs}{\hspace{-0.7pt}}

\newcommand{\dpu}{\mathbin{:}}

\numberwithin{equation}{section}

\DeclareMathAlphabet{\mathpzc}{OT1}{pzc}{m}{it}

\usepackage{mathtools}
\newtagform{brackets}{[}{]}
\usetagform{brackets}

\newcommand{\pq}{p\hspace{-1.0pt}q}

\begin{document}

\title{Set-theoretic type solutions of the braid equation}

\author{Jorge A. Guccione}
\address{Departamento de Matem\'atica\\ Facultad de Ciencias Exactas y Naturales-UBA, Pabell\'on~1-Ciudad Universitaria\\ Intendente Guiraldes 2160 (C1428EGA) Buenos Aires, Argentina.}
\address{Instituto de Investigaciones Matem\'aticas ``Luis A. Santal\'o''\\ Pabell\'on~1-Ciudad Universitaria\\ Intendente Guiraldes 2160 (C1428EGA) Buenos Aires, Argentina.}
\email{vander@dm.uba.ar}

\author{Juan J. Guccione}
\address{Departamento de Matem\'atica\\ Facultad de Ciencias Exactas y Naturales-UBA\\ Pabell\'on~1-Ciudad Universitaria\\ Intendente Guiraldes 2160 (C1428EGA) Buenos Aires, Argentina.}
\address{Instituto Argentino de Matem\'atica-CONICET\\ Saavedra 15 3er piso\\ (\!C1083ACA\!) Buenos Aires, Argentina.}
\email{jjgucci@dm.uba.ar}

\thanks{Jorge A. Guccione and Juan J. Guccione were supported by CONICET PIP 2021-2023 GI,11220200100423CO and CONCYTEC-FONDECYT within the framework of the contest ``Proyectos de Investigaci\'on B\'asica 2020-01'' [contract number 120-2020-FONDECYT]}

\author{Christian Valqui}
\address{Pontificia Universidad Cat\'olica del Per\'u, Secci\'on Matem\'aticas, PUCP, Av. Universitaria 1801, San Miguel, Lima 32, Per\'u.}

\address{Instituto de Matem\'atica y Ciencias Afines (IMCA) Calle Los Bi\'ologos 245. Urb San C\'esar. La Molina, Lima 12, Per\'u.}
\email{cvalqui@pucp.edu.pe}

\thanks{Christian Valqui was supported by CONCYTEC-FONDECYT within the framework of the contest ``Proyectos de Investigaci\'on B\'asica 2020-01'' [contract number 120-2020-FONDECYT]}



\begin{abstract}
In this paper we begin the study of set-theoretic type solution of the braid equation. Our theory includes set-theoretical solutions as basic examples.  More precisely, the linear solution associated to a set-theoretic solution on a set $X$ can be regarded as coming from the coalgebra $kX$, where $k$ is a field and the elements of $X$ are group-like. We introduce and study a broader class of linear solutions associated in a similar way to more general coalgebras. We show that the relationships between set-theoretical solutions, $q$-cycle sets, $q$-braces, skew-braces, matched pairs of groups and invertible $1$-cocycles remain valid in our setting.
\end{abstract}	
	
\maketitle

\tableofcontents

\section*{Introduction}
The Yang–Baxter equation first appeared in theoretical physics in~1967 in the paper \cite{Y} by Yang, and then in~1972 in the paper \cite{B} by Baxter. Since then, many solutions of various forms of the Yang–Baxter equation have been constructed by physicists and mathematicians. In the past three decades this equation has been widely studied from different perspectives and attracted the attention of a wide range of mathematicians because of the applications to mathematical physics, non-commutative descent theory, knot theory, representations of braid groups, Hopf algebras and quantum groups, etc.

Let $V$ be a vector space over a field $k$ and let $r\colon V\ot V \to V \ot V$ be a linear map. We say that $r$ satisfies the Yang-Baxter equation on $V$ if
$$
r_{12} \xcirc r_{13}\xcirc r_{23} = r_{23} \xcirc r_{13} \xcirc r_{12}\quad \text{in $\End_k(V \ot V \ot V)$,}
$$
where $r_{ij}$ means $r$ acting in the $i$-th and $j$-th components. It is easy to check that this occurs if and only if $s\coloneqq \tau\xcirc r$,  where $\tau$ denotes the flip, satisfies the braid equation $s_{12} \xcirc s_{23} \xcirc s_{12} = s_{23} \xcirc s_{12} \xcirc s_{23}$. In~\cite{D}, Drinfeld raised the question of finding set-theoretical (or combinatorial) solutions; i.e. pairs $(Y,s)$, where $Y$ is a set and $s\colon Y\times Y\to Y\times Y$ is a map satisfying this equation. This approach was first considered by Etingof, Schedler and Soloviev~\cite{ESS} and Gateva-Ivanova and Van den Bergh~\cite{GIVB}, for involutive solutions, and by Lu, Yan and Zhu~\cite{LYZ}, and Soloviev~\cite{S}, for non-involutive solutions. Now it is known that there are connections between solutions and affine torsors, Artin-Schelter regular rings, Bieberbach groups and groups of $I$-type, Garside structures, Hopf-Galois theory, left symmetric algebras, etc. (\cites{AGV, DG, De1, DDM, Ch, GI1, GI2, GI3, GI4, GIVB, JO})
	
Let $Y$ be a set and let $kY$ be the free vector space with basis $Y$ endowed with the unique coalgebra structure such that $\Delta(y)=y\ot y$ for all $y\in Y$. Each set-theoretical solution $s\colon Y\times Y\to Y\times Y$ of the braid equation induces a solution $\tilde{s}\colon kY\ot kY\to kY\ot kY$, where $\tilde{s}$ is the coalgebra morphism obtained by linearization of $s$. In this paper we begin the study of set-type solutions of the braid equation in the category of coalgebras. The underlying idea is to replace $kY$ by an arbitrary coalgebra $X$. Our main aim is to show that the relationships between set-theoretical solutions, $q$-cycle sets, $q$-braces, skew-braces, matched pairs of groups and invertible $1$-cocycles (\cites{CJO, ESS, GV, LYZ, R1}) remain valid in our setting. Many of the results in this papers were obtained previously in \cites{AGV, GGV} for cocommutative coalgebras. In order to remove the cocommutativity hypothesis, first we consider left non-degenerate coalgebra endomorphisms of $X\ot X$ and $q$-magma coalgebras on $X$, and we prove that there is a one to one correspondence between these structures. By restriction this induces a one to one correspondence between left non-degenerate solutions of the braid equation and $q$-cycle coalgebras ($q$-magma coalgebras satisfying suitable conditions). Under this bijection, invertible solutions correspond to regular $q$-cycle coalgebras, while non-degenerate invertible solutions correspond to non-degenerate $q$-cycle coalgebras. This is a coalgebra version of \cite{R2}*{Proposition~1}. Then, we begin the study of Hopf $q$-braces (an adaptation to the Hopf algebra setting of the notion of $q$-brace due to Rump). Recall that a braiding operator on a group $G$, as defined in~\cite{LYZ}, is a bijective map $s\colon G\times G \to G\times G$, satisfying several conditions. If we relax the definition, not demanding that identity~[7] in~\cite{LYZ} be satisfied, then we arrive at the definition of weak braiding operator (which we extend to the setting of Hopf algebras). Let $H$ be a Hopf algebra. One of our first results about Hopf $q$-braces is that Hopf $q$-braces on $H$ and weak braiding operators on $H$ are essentially the same. Moreover, we adapt \cite{LYZ}*{Theorem 1} to the setting of Hopf algebras and we obtain several results about the relation between the binary operations in a Hopf $q$-brace and the antipode, and we generalize some results of~\cite{R2}*{Section~3}.
	
A particular type of Hopf $q$-braces are the Hopf skew-braces. We prove that having a Hopf skew-brace with bijective antipode is the same as having a linear $q$-cycle coalgebra (an adaptation to the context of coalgebras of the concept of linear $q$-cycle set due to Rump), and is also the same as having a GV-Hopf skew-brace (an adaptation to the context of Hopf algebras of the original definition of skew-brace due to Guarnieri and Vendramin). Furthermore, we prove that Hopf skew-braces are equivalent to braiding operators and to invertible $1$-cocycles (if the antipode of the underlying Hopf algebra is bijective), which generalizes~\cite{LYZ}*{Theorem~2}. Moreover, we prove that the category of Hopf skew braces is isomorphic to the category of Yetter-Drinfeld braces, recently introduced in \cite{FS}.
	
An ideal of a $q$-brace $\mathcal{H}$ is a Hopf ideal $I$ of $H$ (the underlying Hopf algebra of $\mathcal{H}$) such that the quotient $H/I$ has a structure of $q$-brace induced by the one on $\mathcal{H}$. An example is the $q$-commutator $[\mathcal{H},\mathcal{H}]_q$, which is the smallest ideal $I$ of $\mathcal{H}$ such that the quotient $\mathcal{H}/I$ is a skew-brace. Additionally to this notion, we also study the closely related concept of Hopf sub $q$-brace. An example that we analyze in detail is the socle.
	
Finally, we construct the universal Hopf $q$-brace with bijective antipode and the universal Hopf skew-brace with bijective antipode of a very strongly regular $q$-cycle coalgebra (Definition~\ref{muy fuertemente regular}). This generalizes~\cite{LYZ}*{Theorem~4}.

\smallskip

\noindent {\bf Precedence of operations}\enspace The operations precedence in this paper is the following: the operators with the highest precedence are the binary  operations $x^y$, $x_y$, ${}^yx$, ${}_yx$, $x^{\underline{y}_i}$, $x_{\underline{y}_i}$, ${}^{\underline{y}_i}x$ and ${}_{\underline{y}_i}x$; then comes the multiplication $xy$; and then, the operations $\cdot$, $\dpu$, $*$, $\dast$, $\diam$, $\ddiam$, $\cdot_i$, $\dpu_i$, $\diam_i$, $\ddiam_i$, $\times$, $\mdoubletimes$, $\times^n$, $\mdoubletimes^n$, $\bullet^n$ and $\mathbin{\hookleftarrow}$, that have equal precedence. Of course, as usual, this order of precedence can be modified by the use of parenthesis.

\smallskip

We thank the referee, whose suggestions allow us to improve our paper.  
		
\section{Preliminaries}\label{section: Preliminares}
In this paper we work in the category of vector spaces over a fixed field $k$, all the maps are supposed to be $k$-linear maps, we set $\ot\coloneqq \ot_k$ and the tensor product $V\ot\cdots\ot V$, of $n$ copies of a vector space $V$ is denoted by $V^n$. As usual, $\Delta$ and $\epsilon$ denote the comultiplication and counit of a coalgebra $X$, respectively. Moreover, we use the Sweedler notation $x_{(1)}\ot x_{(2)}$, without the summation symbol, for the comultiplication of $x\in X$. As it is also usual, $X^{\cop}$ stands for the opposite coalgebra of a coalgebra $X$, we let $S$ denote the antipode of a Hopf algebra $H$, and $H^{\op}$ stands for the opposite Hopf algebra of $H$.

Let $X,Y_1,\dots,Y_n$ be coalgebras. For each map $f\colon X\to Y_1\ot\cdots\ot Y_n$ and for each $1\le i\le n$, we set $f_i\coloneqq (\epsilon^{i-1}\ot Y_i\ot \epsilon^{n-i})\xcirc f$. Moreover for each map $s\colon X\ot Y\to Y\ot X$ we let $\tilde{s}$ denote~$s$ regarded as a map from $X^{\cop}\ot Y^{\cop}$ to $Y^{\cop}\ot X^{\cop}$, and we set $s_{\tau}\coloneqq\tau\xcirc s\xcirc \tau$, where $\tau\colon Y\ot X\to X\ot Y$ denote the flip. Note that $(s_{\tau})_1 = s_2\xcirc\tau$ and $(s_{\tau})_2 = s_1\xcirc\tau$.

\begin{remark}\label{some facts} We will need the following facts.

\begin{enumerate}[itemsep=0.7ex, topsep=1.0ex, label={\arabic*)}]
	
\item $\tilde{s}$ is a coalgebra morphism if and only if $s$ is.

\item For each coalgebra morphism $f\colon X\to Y_1\ot\cdots\ot Y_n$, the maps $f_i$ are the unique coalgebra morphisms such that $f = (f_1\ot\cdots\ot f_n) \xcirc \Delta_n$, where $\Delta_n(x) = x_{(1)}\ot\cdots\ot x_{(n)}$.

\item  A map $s\colon X\ot Y\to Y\ot X$  is a coalgebra morphism  if and only if $s_1$ and $s_2$ are coalgebra morphisms, $s=(s_1\ot s_2)\xcirc \Delta_{X\ot Y}$, and, for all $x\in X$ and $y\in Y$,
\begin{equation}\label{eq para nenes}
\qquad\quad s_1(x_{(1)}\ot y_{(1)})\ot s_2(x_{(2)}\ot y_{(2)})=s_1(x_{(2)}\ot y_{(2)})\ot s_2(x_{(1)}\ot y_{(1)}).
\end{equation}

\end{enumerate}
\end{remark}

\begin{proposition}\label{UTIL} Let $X$ be a coalgebra and let $\alpha\colon X^2\to X$ and $\beta\colon X\ot X^{\cop}\to X$ be two maps. Set $x^y\coloneqq \alpha(x\ot y)$ and $x\cdot y\coloneqq \beta(x\ot y)$. If
\begin{equation}\label{igualdades para cdot}
(x\cdot y_{(1)})^{y_{(2)}} = x^{y_{(1)}}\cdot y_{(2)}=\epsilon(y)x\quad\text{for all $x,y\in X$,}
\end{equation}
then $\alpha$ is a coalgebra map if and only if $\beta$ is.
\end{proposition}

\begin{proof} We prove one of the implications and leave the other one to the reader. Assume that $\alpha$ is a coalgebra map. By equality~\eqref{igualdades para cdot}, we have $\epsilon(x)\epsilon(y)=\epsilon\bigl((x\cdot y_{(1)})^{y_{(2)}}\bigr) = \epsilon(x\cdot y)$. So, $\beta$ is counitary. Using again equal\-ity~\eqref{igualdades para cdot}, we get
\begin{equation*}
\epsilon(y)x_{(1)}\ot x_{(2)}= {x_{(1)}}^{y_{(1)}} \cdot y_{(4)}\ot {x_{(2)}}^{y_{(2)}} \cdot y_{(3)}=\bigl(x^{y_{(1)}}\bigr)_{(1)}\cdot y_{(3)}\ot \bigl(x^{y_{(1)}}\bigr)_{(2)}\cdot y_{(2)}.
\end{equation*}
Since $(x\cdot y)_{(1)}\ot (x\cdot y)_{(2)} = \epsilon(y_{(2)})(x\cdot y_{(1)})_{(1)}\ot (x\cdot y_{(1)})_{(2)}$, we have
$$
(x\cdot y)_{(1)}\ot (x\cdot y)_{(2)} = \bigl((x\cdot y_{(1)})^{y_{(2)}}\bigr)_{(1)}\cdot y_{(4)} \ot \bigl((x\cdot y_{(1)})^{y_{(2)}}\bigr)_{(2)}\cdot y_{(3)} = x_{(1)}\cdot y_{(2)}\ot x_{(2)}\cdot y_{(1)},
$$
which proves that $\beta$ is compatible with the comultiplications.
\end{proof}

\begin{remark}\label{unicidad de exponencial} Note that the equalities~\eqref{igualdades para cdot} hold if and only if the map $x\ot y\mapsto x\cdot y_{(1)}\ot y_{(2)}$ is bijective with inverse $x\ot y\mapsto x^{y_{(1)}}\ot y_{(2)}$. Hence for each operation $\cdot$, there exists at most one map $x\ot y\mapsto x^y$~satis\-fying~\eqref{igualdades para cdot} 
\end{remark}

\subsection[Left non-degenerate coalgebra endomorphisms of \texorpdfstring{$X^2$}{X\texttwosuperior}]{Left non-degenerate coalgebra endomorphisms of \texorpdfstring{$\mathbf{X^2}$}{X\texttwosuperior}}

Let $X$ be a coalgebra. For each map $s\colon X^2\to X^2$ we set ${}^x y\coloneqq s_1(x\ot y)$ and $x^y\coloneqq s_2(x\ot y)$. Also, we let $G_s\in \End_k(X^2)$ denote the map defined by $G_s(x\ot y)\coloneqq x^{y_{(1)}}\ot y_{(2)}$.

\begin{definition}\label{no degenerado} A coalgebra morphism $s\colon X^2\to X^2$ is {\em left non-degenerate} if $G_s$ is in\-ver\-tible; and it is  {\em right~non-degenerate} if $\tilde{s}_{\tau}$ is left non-degenerate. If $s$ is left and right non-degenerate, then we say that~$s$ is {\em non-de\-generate}.
\end{definition}

\begin{remark}\label{Gesetau} Note that $\tilde s_{\tau}(x\ot y)={y_{(2)}}^{x_{(2)}}\ot {}^{y_{(1)}}x_{(1)}$ and so $G_{\tilde s_{\tau}}(x\ot y) = {}^{y_{(2)}}x\ot y_{(1)}$.
\end{remark}

\begin{remark}\label{s no de equivale a tilde{s}_tau no degenerado} A coalgebra morphism $s\colon X^2\to X^2$ is non-degenerate if and only if $\tilde{s}_{\tau}$ is non-de\-generate.
\end{remark}

\begin{remark}\label{nociones clasicas de no degenerado} Our notions of left non-degenerate, right non-degenerate and non-degenerate coalgebra morphism generalize the corresponding notions in the set-theoretic framework (see \cites{R0,R2,ESS}). In fact, if $X=kY$ and $s$ extends linearly a map $s'\colon Y^2\to Y^2$, then $G_s$ is defined, on the canonical basis, by $G_s(x\ot y)= x^y\ot y$, which is bijective if and only if the map $x\mapsto x^y$, from $Y$ to $Y$, is bijective for all $y\in Y$ (which means that $s'$ is left non-degenerate according to \cite{R2}*{page~143}). Similarly, $G_{\tilde{s}_{\tau}}$ is given by $G_{\tilde{s}_{\tau}}(x\ot y)={}^yx\ot y$, for $x,y\in Y$. So, $s$ is right non-degenerate if and only if the map $x\mapsto {}^yx$ is bijective for all $y\in Y$ (see~\cite{R2}*{Definition~4}).
\end{remark}

\begin{notation}\label{notacion cdot y dpu} For a left non-degenerate coalgebra morphism $s\colon X^2\to X^2$, we set $\mathcal{X}\coloneqq (X,\cdot,\dpu)$, where $x\cdot y\coloneqq (X\ot \epsilon)\xcirc G_s^{-1}(x\ot y)$ and $x\dpu y\coloneqq  {}^{y\cdot x_{(1)}}x_{(2)}$. If necessary, we will write $x\cdot_s y$, $x\dpu_s y$ and $\mathcal{X}_s$ instead of $x\cdot y$, $x\dpu y$ and $\mathcal{X}$, respectively.  Arguing as in Remark~\ref{nociones clasicas de no degenerado} we see that these operations are  generalizations of the notions of \cite{R2}*{Page~143}. However, we change the order of the elements: If we use $\cdot_r$ and $\dpu_r$ to represent the $q$-magma operations defined in~\cite{R2}, we have  $x\cdot y = y\cdot_r x$ and $x\dpu y = y\dpu_r x$. Taking this transposition into account, most of our results, when considered within the set-theoretic context, align with the results presented in~\cite{R2}.

Next we explain our choice of laterality. Since the mapping $x\mapsto x^y$ is a ``right action'' (which, in section~5 is a genuine right action), we will denote its ``inverse action'' as $x\mapsto x\cdot y$. The advantage of our choice of laterality can be seen for example in cases like identities~\eqref{condicion q-braza},  where the actions of $\cdot$ and $\dpu$ on a product are more naturally represented as right actions.
\end{notation}

\begin{remark}\label{Caracterizacion de no degenerado a izquierda} Let $s\in \End_{\Coalg}(X^2)$. If $s$ is left non-degenerate, then $G_s^{-1}$ is right colinear since $G_s$ is, which implies that $G_s^{-1}(x\ot y)= x\cdot y_{(1)}\ot y_{(2)}$ for all $x,y\in X$. Using this, and applying $X\ot \epsilon$ to the identities $G_s\xcirc G_s^{-1}=G_s^{-1}\xcirc G_s=\ide _{X^2}$, we obtain that identities~\eqref{igualdades para cdot} are fulfilled. Conversely, if there exists a map $x\ot y\mapsto x\cdot y$ satisfying identities~\eqref{igualdades para cdot}, then $s$ is left non-degenerate. Note that identities~\eqref{igualdades para cdot} imply that $y_{(2)} \dpu x^{y_{(1)}} = {}^xy$, for all $x,y\in X$.
\end{remark}

\begin{notation}\label{notacion ast y bar{ast}} For a right non-degenerate coalgebra map $s\colon X^2\to X^2$, we set
$$
x\ast y\coloneqq (X\ot \epsilon)\xcirc G_{\tilde{s}_{\tau}}^{-1}(x\ot y)\quad\text{and}\quad x\dast y\coloneqq  {x_{(1)}}^{y\ast x_{(2)}}.
$$
Note that $\dast$ is related with $\tilde{s}_{\tau}$ and $\ast$ in the same way as $\dpu$ is related with $s$ and $\cdot$. That is,~$\mathcal{X}_{\scriptscriptstyle \tilde{s}_{\tau}}=(X^{\cop},\ast,\dast)$. If necessary, we will write $x\ast_s y$ and $x\dast_{\!s} y$ instead of $x\ast y$ and $x\dast y$, respectively. Notice that in the set-theoretic context the operation $\ast$ corresponds to the one introduced in \cite{R2}*{Definition~2}.
\end{notation}

\begin{remark}\label{Caracterizacion de no degenerado a derecha} Let $s\in\End_{\Coalg}(X^2)$. If $s$ is right non-degenerate, then
\begin{equation}\label{igualdades para ast}
{}^{y_{(1)}}(x\ast y_{(2)}) = {}^{y_{(2)}}x\ast y_{(1)}=\epsilon(y)x\quad\text{for all $x,y\in X$.}
\end{equation}
Conversely, if there exists a map $x\ot y\mapsto x \ast y$ satisfying identities~\eqref{igualdades para ast}, then $s$ is right non-degenerate.  Clearly identities~\eqref{igualdades para ast} imply that $y_{(1)}\dast {}^{y_{(2)}}x=y^x$, for all $x,y\in X$.
\end{remark}

\subsection[Left regular\texorpdfstring{$q$}{q}-magma coalgebras]{\texorpdfstring{Left regular $\mathbf{q}$}{q}-magma coalgebras}\label{Left regular coalgebras}
Let $X$ be a coalgebra and let $p,d\colon X^2\to X$ be maps such that $\epsilon\xcirc p=\epsilon\xcirc d=\epsilon\ot \epsilon$. For each $x,y\in X$, we set $x\cdot y\coloneqq p(x\ot y)$ and $x\dpu y\coloneqq d(x\ot y)$. Let $\overline{G}_{\mathcal{X}}\in \End_k(X^2)$ be the map defined by $\overline{G}_{\mathcal{X}}(x\ot y)\coloneqq x\cdot {y_{(1)}}\ot y_{(2)}$.

\begin{definition}\label{q-magma coalgebra} We say that $\mathcal{X}=(X,\cdot,\dpu)$ is a  {\em $q$-magma coalgebra} if the map $h\colon X\ot X^{\cop}\to X^{\cop}\ot X$, defined by $h(x\ot y)\coloneqq y_{(1)}\dpu x_{(2)}\ot x_{(1)}\cdot y_{(2)}$, is a coalgebra morphism. If necessary, we will write $h_{\mathcal{X}}$ instead of $h$. Let $\mathcal{Y}=(Y,\cdot,\dpu)$ be another $q$-magma coalgebra. A {\em morphism} $f\colon \mathcal{X}\to \mathcal{Y}$ is a coalgebra map $f\colon X\to Y$ such that $f(x\cdot x')=f(x)\cdot f(x')$ and $f(x\dpu x')=f(x)\dpu f(x')$, for all $x,x'\in X$.
\end{definition}

\begin{remark}\label{d y p son morfismos de coalgears} By Remark~\ref{some facts} and the fact that $\epsilon\xcirc p=\epsilon\xcirc d = \epsilon\ot \epsilon$, we know that $\mathcal{X}$ is a $q$-magma coalgebra if and only if $p,d\colon X\ot X^{\cop}\to X$ are coalgebra maps~and
\begin{equation}\label{intercambio . :}
y_{(1)}\dpu x_{(2)}\ot x_{(1)}\cdot y_{(2)}=y_{(2)}\dpu x_{(1)}\ot x_{(2)}\cdot y_{(1)}\qquad \text{for all $x,y\in X$.}
\end{equation}
\end{remark}

\begin{remark}\label{opuesto de q magma coalgebra} If $\mathcal{X}=(X,\cdot,\dpu)$ is a $q$-magma coalgebra, then $\mathcal{X}^{\op}\coloneqq (X^{\cop},\dpu, \cdot)$ is also. We call~$\mathcal{X}^{\op}$~the~{\em~op\-po\-site} $q$-magma coalgebra of $\mathcal{X}$. Equality~\eqref{intercambio . :} says that $h_{\mathcal{X}^{\op}}= \tau\xcirc h_{\mathcal{X}}\xcirc \tau$. Note that $\overline{G}_{\mathcal{X}^{\op}}(x\ot y)\coloneqq x\dpu {y_{(2)}}\ot y_{(1)}$.
\end{remark}

\begin{definition}\label{q-magma coalgebra no degenerada a izquierda} A $q$-magma coalgebra $\mathcal{X}$ is {\em left regular} if the map $\overline{G}_{\mathcal{X}}$ is invertible. If~$\mathcal{X}^{\op}$ is left regular, then we call $\mathcal{X}$ {\em right regular}. Finally, we say that $\mathcal{X}$ is {\em regular} if it is left regular and right regular.
\end{definition}

\begin{remark}\label{X regular equivale a X^op regular} Note that $\mathcal{X}$ is regular if and only if $\mathcal{X}^{\op}$ is regular.
\end{remark}

\begin{remark} Let $(Y,\cdot,\dpu)$ be a $q$-magma (see \cite{R2}*{Definition~1}), and let $\mathcal{X}=(kY,\cdot,\dpu)$ be the linearization of $(Y,\cdot,\dpu)$. As in Remark~\ref{nociones clasicas de no degenerado}, we see that $\mathcal{X}$ is left regular if and only if the map $x\mapsto x\cdot y$, from $Y$ to $Y$, is bijective for all $y\in Y$. Similarly, $\mathcal{X}$ is right regular if and only if the map $x\mapsto x \dpu y$, from $Y$ to $Y$, is bijective for all $y\in Y$. Thus, these notions generalize the set-theoretic ones (see \cite{R2}*{Definition~1}).
\end{remark}

\begin{remark}\label{notacion s1 y s2} For a left regular $q$-magma coalgebra $\mathcal{X}=(X,\cdot,\dpu)$, we set $G_{\mathcal{X}}\coloneqq  \overline{G}_{\mathcal{X}}^{-1}$ and we define the map $s\colon X^2\to X^2$ by $s(x\ot y)\coloneqq {}^{x_{(2)}}y_{(2)}\ot {x_{(1)}}^{y_{(1)}}$, where $x^y\coloneqq (X\ot \epsilon)\xcirc G_{\mathcal{X}}(x\ot y)$ and ${}^xy\coloneqq y_{(2)}\dpu x^{y_{(1)}}$ (if necessary we will write $s_{\scriptscriptstyle\mathcal{X}}$ instead of $s$). Arguing as in Remark~\ref{Caracterizacion de no degenerado a izquierda}, we can see that $G_{\mathcal{X}}(x\ot y) = x^{y_{(1)}}\ot y_{(2)}$, for all~$x,y\in X$, and identities~\eqref{igualdades para cdot} are fulfilled. Conversely, if there is a map $x\ot y\mapsto x^y$ satisfying~\eqref{igualdades para cdot}, then~$\mathcal{X}$ is left regular. Note that by~\eqref{igualdades para cdot} and Proposition~\ref{UTIL},
$$
{}^{x\cdot y_{(1)}}y_{(2)} = y\dpu x,\quad\epsilon(x)\epsilon(y)=\epsilon(x^y)\quad\text{and} \quad \epsilon({}^{x}y)= \epsilon(y_{(2)}\dpu x^{y_{(1)}}) = \epsilon(y_{(2)}) \epsilon(x^{y_{(1)}}) = \epsilon(x)\epsilon(y),
$$	
for all $x,y\in X$. Hence, $s_2(x\ot y)=x^y$ and $s_1(x\ot y)= {}^xy$, for all $x,y\in X$.
\end{remark}

\begin{remark}\label{automorfismos de X^2} Let $\mathcal{X}$ be a left regular $q$-magma coalgebra. A direct computation, using identities~\eqref{igualdades para cdot} and that $p\colon X\ot X^{\cop}\to X$, is a coalgebra map (by Remark~\ref{d y p son morfismos de coalgears}), shows that $s\xcirc \overline{G}_{\mathcal{X}}\xcirc \tau = \ov{G}_{\mathcal{X}^{\op}}$. Thus $s$ is invertible if and only if $\mathcal{X}$ is regular.
\end{remark}

\begin{remark}\label{notacion bar s1 y bar s2} For a right regular $q$-magma coalgebra $\mathcal{X}=(X,\cdot,\dpu)$, we define $\bar{s}\colon (X^{\cop})^2 \to (X^{\cop})^2$, by $\bar{s}(x\ot y)\coloneqq \cramped{{}_{x_{(1)}}y_{(1)}\ot {x_{(2)}}_{y_{(2)}}}$, where $x_y\coloneqq (X\ot \epsilon)\xcirc G_{\mathcal{X^{\op}}}(x\ot y)$ and ${}_xy\coloneqq y_{(1)}\cdot x_{y_{(2)}}$. Note that $\bar{s}= s_{\mathcal{X}^{\op}}$. So, by Remark~\ref{notacion s1 y s2}, the equality $G_{\mathcal{X^{\op}}}(x\ot y)= x_{y_{(2)}}\ot y_{(1)}$ holds, for all $x,y\in X$, and
\begin{equation}\label{igualdades para dpu}
(x\dpu y_{(2)})_{y_{(1)}} = x_{y_{(2)}}\dpu y_{(1)}=\epsilon(y)x\quad\text{for all $x,y\in X$.}
\end{equation}
Conversely, if there exists a map $x\ot y\mapsto x_y$ satisfying identities~\eqref{igualdades para dpu}, then $\mathcal{X}$ is right regular.
\end{remark}

\begin{remark} Let $(Y,\cdot,\dpu)$ be a $q$-magma and let $\mathcal{X}=(kY,\cdot,\dpu)$ be the linearization of $(Y,\cdot,\dpu)$. Assume that~$\mathcal{X}$ is right regular. Arguing as in Remark~\ref{nociones clasicas de no degenerado}, one can see that the re\-stric\-tion to $Y\times Y$ of the map $x\ot y \mapsto x_y$ is the map introduced in \cite{R2}*{Definition~1}. Similarly, the re\-stric\-tion to $Y\times Y$ of $x\ot y \mapsto {}_xy$ is the map introduced in Remark~1 of \cite{R2}*{Definition~4}.
\end{remark}

\begin{definition}\label{q-magma coalgebra no degenerada} A $q$-magma coalgebra $\mathcal{X}=(X,\cdot,\dpu)$ is {\em right non-degenerate} if it is left regular and the map $H_{\mathcal{X}}\in \End_k(X^2)$, defined by $H_{\mathcal{X}}(x\ot y)\coloneqq {}^{y_{(2)}}x\ot y_{(1)}$, is invertible; while it is {\em left non-degenerate} if~$\mathcal{X}^{\op}$ is right non-degenerate (that is, if $\mathcal{X}$ is right regular and the map $x\ot y \mapsto {}_{y_{(1)}}x\ot y_{(2)}$ is bijective). A $q$-magma coalgebra is {\em non-degenerate} if it is left and right non-degenerate.
\end{definition}

\begin{remark} In the set-theoretic context, the notion of non-degenerate reduces to that introduced in~\cite{R2}*{De\-finition~1}. In fact, in our context a $q$-magma coalgebra is non-degenerate, if and only if the maps
\begin{alignat*}{2}
&x\ot y\mapsto x\cdot y_{(1)}\ot y_{(2)}&&\qquad x\ot y\mapsto {}^{y_{(2)}}x\ot y_{(1)}=x_{(2)}\dpu {y_{(2)}}^{x_{(1)}}\ot y_{(1)}\\
& x\ot y\mapsto x\dpu y_{(2)}\ot y_{(1)}&&\qquad  x\ot y\mapsto {}_{y_{(1)}}x\ot y_{(2)}=x_{(1)}\cdot {y_{(1)}}_{x_{(2)}}\ot y_{(2)}
\end{alignat*}
are bijective. Since, according to~\cite{R2}*{Definition~1}, in the set-theoretic context a $q$-magma is non-degenerate if and only if the maps $x \mapsto x\cdot y$, $x \mapsto x\dpu y^x$, $x\mapsto x\dpu y$ and $x\mapsto x\cdot y_x$ are bijective, our definition generalizes the set-theoretic one.
\end{remark}		

\section[Coalgebra endomorphisms vs \texorpdfstring{$q$}{q}-magma coalgebras]{Coalgebra endomorphisms vs \texorpdfstring{$\mathbf{q}$}{q}-magma coalgebras}\label{Seccion left regular q-magma coalgebras}

Let $X$ be a coalgebra. In this section we establish a one-to-one correspondence between the left non-degenerate coalgebra endomorphisms of $X^2$ and the left regular $q$-magma coalgebras with underlying coalgebra $X$. From here until Proposition~\ref{h es morfismo de coalgebras} inclusive, we assume that $s$ is a left non-degenerate coalgebra endomorphism of $X^2$ and $\cdot$ and $\dpu$ are as in Notation~\ref{notacion cdot y dpu}. Note that, by Remark~\ref{Caracterizacion de no degenerado a izquierda}, equalities~\eqref{igualdades para cdot} hold. Moreover, by Proposition~\ref{UTIL}, the map $p\colon X\ot X^{\cop}\to X$, defined by $p(x\ot y)\coloneqq x\cdot y$, is a coalgebra map.

\begin{remark}\label{relacion entre s, G y dpu} By the fact that $p$ is a coalgebra morphism and equalities~\eqref{eq para nenes} and~\eqref{igualdades para cdot},
$$
s(x\cdot y_{(1)} \ot y_{(2)}) = {}^{x_{(2)}\cdot y_{(1)}}y_{(4)}\ot (x_{(1)}\cdot y_{(2)})^{y_{(3)}} = y\dpu x_{(2)} \ot x_{(1)} \quad\text{for all $x,y\in X$.}
$$
\end{remark}

\begin{remark}\label{intercambio} By Remark~\ref{Caracterizacion de no degenerado a izquierda} and equality~\eqref{eq para nenes}, for each $x,y\in X$, we have:
$$
y_{(3)}\dpu {x_{(2)}}^{\hspace{-0.5pt}y_{(2)}} \ot y_{(4)}\dpu {x_{(1)}}^{\hspace{-0.5pt}y_{(1)}} = {}^{x_{(2)}}\hspace{-0.5pt}y_{(2)}\ot y_{(3)}\dpu {x_{(1)}}^{\hspace{-0.5pt}y_{(1)}} = {}^{x_{(1)}}\hspace{-0.5pt}y_{(1)}\ot y_{(3)}\dpu {x_{(2)}}^{\hspace{-0.5pt}y_{(2)}} = {}^{x_{(1)}}\hspace{-0.5pt}y_{(1)}\ot {}^{x_{(2)}}\hspace{-0.5pt}y_{(2)}.
$$
\end{remark}

\begin{remark} Since $\epsilon (x\cdot y) = \epsilon(x)\epsilon(y)$, we have $\epsilon (x\dpu y) = \epsilon({}^{y\cdot x_{(1)}}x_{(2)}) = \epsilon(y\cdot x_{(1)})\epsilon(x_{(2)}) = \epsilon(x)\epsilon(y)$.
\end{remark}

\begin{lemma}\label{dpu es un morfismo de coalgebras} The map $d\colon X\ot X^{\cop}\to X$, defined by $d(x\ot y)\coloneqq x\dpu y$, is a coalgebra morphism.
\end{lemma}

\begin{proof} By the previous remark we know that $d$ is compatible with $\epsilon$. Let $\tilde{d}\colon X^{\cop}\ot X\to X$ be the map given by $\tilde{d}(x\ot y)\coloneqq y\dpu x$. Since $d = \tilde{d}\xcirc \tau$, in order to finish the proof it suffices to show that $\tilde{d}$ is compatible with the comultiplications. Since $G_s$ is bijective, in order to check this, it suffices to prove that
$$
\bigl({y_{(2)}}\dpu x^{y_{(1)}}\bigr)_{(1)}\ot \bigl({y_{(2)}}\dpu x^{y_{(1)}}\bigr)_{(2)} = y_{(3)} \dpu {x_{(2)}}^{y_{(2)}}\ot y_{(4)}\dpu {x_{(1)}}^{y_{(1)}}\quad\text{for all $x,y\in X$.}
$$
But, since ${}^{x_{(1)}}y_{(1)}\ot {}^{x_{(2)}}y_{(2)} = \bigl({}^xy\bigr)_{(1)}\ot \bigl({}^xy\bigr)_{(2)}$, from  Remarks~\ref{Caracterizacion de no degenerado a izquierda} and~\ref{intercambio}, it follows that
$$
y_{(3)} \dpu {x_{(2)}}^{y_{(2)}}\ot y_{(4)}\dpu {x_{(1)}}^{y_{(1)}} = \bigl({}^xy\bigr)_{(1)}\ot \bigl({}^xy\bigr)_{(2)} =  \bigl(y_{(2)} \dpu x^{y_{(1)}}\bigr)_{(1)}\ot \bigl(y_{(2)} \dpu x^{y_{(1)}}\bigr)_{(2)},
$$
as desired.
\end{proof}

\begin{proposition}\label{h es morfismo de coalgebras} The triple $\mathcal{X}=(X,\cdot,\dpu)$ is a left regular $q$-magma coalgebra.
\end{proposition}

\begin{proof} Since $\overline{G}_{\mathcal{X}} = G_s^{-1}$ (by Remark~\ref{Caracterizacion de no degenerado a izquierda} and the definition of $\overline{G}_{\mathcal{X}}$ above Definition~\ref{q-magma coalgebra}) we know that $\overline{G}_{\mathcal{X}}$ is invertible. So, by Remark~\ref{d y p son morfismos de coalgears}, the fact that $p$ is a coalgebra morphism and Lemma~\ref{dpu es un morfismo de coalgebras}, we only must prove that
$$
y_{(1)}\dpu x_{(2)}\ot x_{(1)}\cdot y_{(2)}=y_{(2)}\dpu x_{(1)}\ot x_{(2)}\cdot y_{(1)}\quad\text{for all $x,y\in X$.}
$$
But, by Remark~\ref{relacion entre s, G y dpu}, we have
$$
y\dpu x_{(2)} \ot x_{(1)} = s(x\cdot y_{(1)} \ot y_{(2)}) = {}^{x_{(1)}\cdot y_{(2)}}y_{(3)}  \ot (x_{(2)}\cdot y_{(1)})^{y_{(4)}} = y_{(2)}\dpu x_{(1)}\ot (x_{(2)}\cdot y_{(1)})^{y_{(3)}},
$$
which implies
$$
y_{(2)}\dpu x_{(1)}\ot x_{(2)}\cdot y_{(1)}=y_{(2)}\dpu x_{(1)}\ot {(x_{(2)}\cdot y_{(1)})}^{y_{(3)}}\cdot y_{(4)}= y_{(1)}\dpu x_{(2)}\ot x_{(1)}\cdot y_{(2)},
$$
as desired.
\end{proof}

By Proposition~\ref{h es morfismo de coalgebras}, each left non-degenerate coalgebra endomorphism $s$ of $X^2$ has associated a left regular $q$-magma coalgebra. Our next aim is to prove the converse. From here to Proposition~\ref{s es morfismo de coalgebras} inclusive, we assume that $\mathcal{X}= (X,\cdot,\dpu)$ is a left regular $q$-magma coalgebra and that $s$ is as in Remark~\ref{notacion s1 y s2}. Note that, by Remark~\ref{notacion s1 y s2}, equalities~\eqref{igualdades para cdot} hold.

\begin{lemma}\label{s_i es morfismo de coalgebras} The maps $s_1\colon X^2\to X$ and $s_2\colon X^2\to X$ are coalgebra morphisms.
\end{lemma}

\begin{proof} By Proposition~\ref{UTIL} and Remark~\ref{d y p son morfismos de coalgears}, we know that $s_2$ is a coalgebra map. Next we prove that $s_1$ is also. By Remark~\ref{notacion s1 y s2}, we know that $s_1$ is compatible with the counits. Since $\overline{G}_{\mathcal{X}}$ is bijective, in order to check that $s_1$ is also compatible with the comultiplications, it suffices to prove that
$$
\bigl({}^{x\cdot y_{(1)}}y_{(2)}\bigl)_{(1)}\ot \bigl({}^{x\cdot y_{(1)}}y_{(2)}\bigl)_{(2)}= {}^{x_{(1)}\cdot y_{(2)}}y_{(3)}\ot {}^{x_{(2)}\cdot y_{(1)}}y_{(4)} \quad\text{for all $x,y\in X$.}
$$
But, by equality~\eqref{intercambio . :} and Remark~\ref{notacion s1 y s2}, we have
$$
\bigl({}^{x\cdot y_{(1)}}y_{(2)}\bigl)_{(1)}\ot \bigl({}^{x\cdot y_{(1)}}y_{(2)}\bigl)_{(2)}= y_{(1)}\!\dpu \! x_{(2)} \ot y_{(2)}\! \dpu\! x_{(1)}= y_{(2)}\!\dpu\! x_{(1)} \ot {}^{x_{(2)}\cdot y_{(1)}}y_{(3)} = {}^{x_{(1)}\cdot y_{(2)}}y_{(3)}\ot {}^{x_{(2)}\cdot y_{(1)}}y_{(4)},
$$
as desired.
\end{proof}

\begin{proposition}\label{s es morfismo de coalgebras} The map $s$ is a left non-degenerate coalgebra endomorphism of $X^2$.
\end{proposition}

\begin{proof} By Remark~\ref{notacion s1 y s2} and the definition of $G_s$ above Definition~\ref{no degenerado}, we know that $G_s=\overline{G}_{\mathcal{X}}^{-1}$ is invertible. So, we only must prove that  $s$ is a coalgebra morphism. By Remark~\ref{some facts} and Lemma~\ref{s_i es morfismo de coalgebras} in order to check this it suffices to prove that ${{}^{x_{(1)}}}{y_{(1)}\ot {x_{(2)}}}^{y_{(2)}} = {}^{x_{(2)}}y_{(2)}\ot {x_{(1)}}^{y_{(1)}}$, since this implies $s=(s_1\ot s_2) \Delta_{X^2}$. But,
\begin{align*}
{}^{x_{(1)}}y\ot x_{(2)} &= y_{(4)}\dpu {x_{(1)}}^{y_{(1)}}\ot {x_{(2)}}^{y_{(2)}}\cdot y_{(3)} && \text{by equality~\eqref{igualdades para cdot}}\\
& = y_{(3)}\dpu {x_{(2)}}^{y_{(2)}}\ot  {x_{(1)}}^{y_{(1)}}\cdot y_{(4)} && \text{by Lemma~\ref{s_i es morfismo de coalgebras} and equality~\eqref{intercambio . :}}\\
& = {}^{x_{(2)}}y_{(2)} \ot  {x_{(1)}}^{y_{(1)}}\cdot y_{(3)},
\end{align*}
and so ${}^{x_{(1)}}y_{(1)}\ot {x_{(2)}}^{y_{(2)}} = {}^{x_{(2)}}y_{(2)}\ot \bigl({x_{(1)}}^{y_{(1)}}\cdot y_{(3)}\bigr)^{y_{(4)}}= {}^{x_{(2)}}y_{(2)}\ot {x_{(1)}}^{y_{(1)}}$, as desired.
\end{proof}

\begin{remark}\label{correspondencia endomorfismos de X^2 no degenerados a izquierday q magma coalgebras regulares a izquierda} A direct computation shows that the correspondences $s\mapsto \mathcal{X}_{s}$ and $\mathcal{X}\mapsto s_{\scriptscriptstyle \mathcal{X}}$ are inverse one of each other. So, we have a one-to-one correspondence between left non-degenerate coalgebra endomorphisms of $X^2$ and left regular $q$-magma coalgebras with underlying coalgebra~$X$.
\end{remark}

\begin{proposition}\label{q-magma coalgebra asociada a s_tau} Let $s\colon X^2\to X^2$ be a left non-degenerate coalgebra morphism. The following assertions are equiv\-alent:

\begin{enumerate}[itemsep=0.7ex, topsep=1.0ex, label={\emph{\arabic*)}}]

\item $s$ is non-degenerate.

\item $\tilde{s}_{\tau}$ is left non-degenerate.

\item $\tilde{s}_{\tau}$ is non-degenerate.

\item $\mathcal{X}_s$ is right non-degenerate.

\item $\mathcal{X}_{\tilde{s}_{\tau}}$ is right non-degenerate.

\item $\mathcal{X}_{\tilde{s}_{\tau}}$ is left regular.

\end{enumerate}

\end{proposition}

\begin{proof} Since $\widetilde{(\tilde{s}_{\tau})}_{\tau} = s$, from Definition~\ref{no degenerado}, it follows that items~1), 2) and~3) are equivalent; while, by Definition~\ref{q-magma coalgebra no degenerada}, items~2) and~4) are equivalent, since $\mathcal{X}_s$ is left regular (by Remark~\ref{correspondencia endomorfismos de X^2 no degenerados a izquierday q magma coalgebras regulares a izquierda}) and $H_{\mathcal{X}_s} = G_{\tilde{s}_{\tau}}$ (by Remark~\ref{Gesetau}). Clearly item~5) implies item~6); while item~6) implies item~2), by Remark~\ref{correspondencia endomorfismos de X^2 no degenerados a izquierday q magma coalgebras regulares a izquierda}. Finally, using that item~1) implies item~4), applied to $s_{\tau}$, we obtain that item~5) follows from item~3).
\end{proof}

\begin{remark}\label{analogo a la derecha de s(x cdot y_(1) ot y_(2))} Assume that $s\colon X^2\to X^2$ is a right non-degenerate coalgebra map and let $\ast$ and $\dast$ be as in Notation~\ref{notacion ast y bar{ast}}. Applying Remark~\ref{relacion entre s, G y dpu} and Proposition~\ref{h es morfismo de coalgebras} to $\tilde{s}_{\tau}$, we get that $\mathcal{X}_{\scriptscriptstyle \tilde{s}_{\tau}}=(X^{\cop},\ast,\dast)$ is a left regular $q$-magma coalgebra and $s(x_{(1)}\ot y\ast x_{(2)}) = \tau \tilde{s}_{\tau}\tau (x_{(1)}\ot y\ast x_{(2)})= y_{(2)}\ot x\dast y_{(1)}$, for all $x,y\in X$.
\end{remark}

\begin{proposition}\label{q magma coalgebra de tilde s^-1} A map $s\in \Aut_{\Coalg}(X^2)$ is left non-degenerate if and only if $\tilde{s}^{-1}$ is. Moreover,~in~this case, $\ov{G}_{\mathcal{X}_s^{\op}}$ is invertible, $G_{\tilde{s}^{-1}} = \ov{G}_{\mathcal{X}_s^{\op}}^{-1}$ and $\mathcal{X}_{\tilde{s}^{-1}}=\mathcal{X}_s^{\op}$.
\end{proposition}

\begin{proof} By symmetry it suffices to prove that if $s$ is left non-degenerate, then $\tilde{s}^{-1}$ is also, and that, in this case $\ov{G}_{\mathcal{X}_s^{\op}}$ is invertible, $G_{\tilde{s}^{-1}} = \ov{G}_{\mathcal{X}_s^{\op}}^{-1}$ and $\mathcal{X}_{\tilde{s}^{-1}}=\mathcal{X}_s^{\op}$. In order to prove that $\tilde{s}^{-1}$ is left non-degenerate we must check that $G_{\tilde{s}^{-1}}$ is invertible. Since~$\ov{G}_{\mathcal{X}_s^{\op}}$ is invertible (by Remarks~\ref{automorfismos de X^2} and~\ref{correspondencia endomorfismos de X^2 no degenerados a izquierday q magma coalgebras regulares a izquierda}), for this is suffices to check that $G_{\tilde{s}^{-1}}\xcirc \ov{G}_{\mathcal{X}_s^{\op}}=\ide$. But,
$$
G_{\tilde{s}^{-1}}\xcirc \ov{G}_{\mathcal{X}_s^{\op}}(y\ot x) = G_{\tilde{s}^{-1}}(y\dpu x_{(2)}\ot x_{(1)})= (\tilde{s}^{-1})_{2}(y\dpu x_{(3)}\ot x_{(2)})\ot x_{(1)}= y\ot x,
$$
because $(\tilde{s}^{-1})_{2}(y\dpu x_{(2)}\ot x_{(1)})=\epsilon(x)y$, by Remark~\ref{relacion entre s, G y dpu}. It remains to prove that $\mathcal{X}_{\tilde{s}^{-1}}=\mathcal{X}_s^{\op}$. But,~by~No\-tation~\ref{notacion cdot y dpu} and the definition of $\ov{G}_{\mathcal{X}_s^{\op}}$, we have
$$
x\cdot_{\tilde{s}^{-1}} y = (X\ot \epsilon)\xcirc G_{\tilde{s}^{-1}}^{-1}(x\ot y) = (X\ot \epsilon)\xcirc \ov{G}_{\mathcal{X}^{\op}}(x\ot y) = x \dpu y,
$$
which implies that $\dpu_{\tilde{s}^{-1}} = \cdot$, because
$$
x\cdot y_{(1)} \ot y_{(2)} = \tilde{s}^{-1}(y \dpu x_{(2)} \ot x_{(1)}) = \tilde{s}^{-1}(y\cdot_{\tilde{s}^{-1}} x_{(2)} \ot x_{(1)}) = x \dpu_{\tilde{s}^{-1}} y_{(1)} \ot y_{(2)},
$$
where the first equality follows from Remark~\ref{relacion entre s, G y dpu}; and the last one, from Remark~\ref{relacion entre s, G y dpu} applied to $\tilde{s}^{-1}$.
\end{proof}

\begin{remark}\label{s^-1=bar s} By Remark~\ref{correspondencia endomorfismos de X^2 no degenerados a izquierday q magma coalgebras regulares a izquierda} and Proposition~\ref{q magma coalgebra de tilde s^-1}, if $s$ is a left non-degenerate coalgebra automorphism, then $\mathcal{X}_s^{\op} = \mathcal{X}_{\tilde{s}^{-1}}$ is left regular, and so $\mathcal{X}_s$ is right regular. Hence, by Remark~\ref{notacion bar s1 y bar s2}, we have $\tilde{s}^{-1}=\bar{s}$.
\end{remark}

\begin{corollary}\label{involutivo} If $s$ and $\tilde{s}$ are left non-degenerate coalgebra automorphisms, then $s$ is involutive if and only if $\cdot_{\tilde{s}}=\dpu_s$ and $\dpu_{\tilde{s}}=\cdot_s$.
\end{corollary}

\begin{proof} By Remark~\ref{correspondencia endomorfismos de X^2 no degenerados a izquierday q magma coalgebras regulares a izquierda} and Proposition~\ref{q magma coalgebra de tilde s^-1},
$$
s\text{ involutive } \Leftrightarrow \tilde{s} \text{ involutive } \Leftrightarrow \mathcal{X}_{\tilde{s}} = \mathcal{X}_{\tilde{s}^{-1}} \Leftrightarrow \mathcal{X}_{\tilde{s}} = \mathcal{X}_s^{\op} \Leftrightarrow \cdot_{\tilde{s}}=\dpu_s\text{ and } \dpu_{\tilde{s}}=\cdot_s,
$$
as desired.
\end{proof}

\begin{corollary} If $X$ is cocommutative, then a left non-degenerate coalgebra automorphism $s\colon X\to X$ is involutive if and only if, in its associated $q$-magma coalgebra, $\cdot=\dpu$.
\end{corollary}

\begin{proof} By Corollary~\ref{involutivo}, since $s= \tilde{s}$.
\end{proof}	

\begin{example}\label{ej p pq q} The previous result is not true, if $X$ is non-cocommutative. Take for example $X$ the~coal\-gebra with basis $\{p,\pq,q\}$ and comultiplication given by
$$
\Delta(p)\coloneqq p\ot p,\quad \Delta(q)\coloneqq q\ot q\quad\text{and}\quad \Delta(\pq)\coloneqq p\ot \pq+\pq\ot q.
$$
Note that $X$ is counitary with $\epsilon(p)=\epsilon(q)=1$ and $\epsilon(\pq)=0$. Consider the $k$-linear map $s:X\ot X\to X\ot X$ defined by
\begin{alignat*}{3}
&s(p\ot p)\coloneqq p\ot p, &&\quad s(p\ot \pq)\coloneqq \pq\ot p, &&\quad s(p\ot q)\coloneqq q\ot p,\\
&s(\pq\ot p)\coloneqq p\ot \pq,&&\quad s(\pq\ot \pq)\coloneqq p\ot p-\pq\ot \pq-q\ot q, &&\quad s(\pq\ot q)\coloneqq -q\ot \pq,\\
&s(q\ot p)\coloneqq p\ot q, &&\quad s(q\ot \pq)\coloneqq -\pq\ot q, &&\quad s(q\ot q)\coloneqq q\ot q.
\end{alignat*}
This example corresponds to case~2) of Theorem~5.4 of~\cite{GGV1} with $\Gamma_1\coloneqq 1$, $\alpha_1\coloneqq 1$ and $\alpha_3\coloneqq -1$. A direct computation proves that $s$ and $\tilde s$ are left non-degenerate involutive coalgebra automorphisms. In this example $\cdot$  and $\dpu$ are given by
\begin{alignat*}{3}
&p\cdot p=p, &&\quad p\cdot \pq=0, &&\quad p\cdot q=p,\\
&\pq\cdot p=\pq,&&\quad \pq\cdot \pq=-p+q, &&\quad \pq\cdot q=-\pq,\\
&q\cdot p=q, &&\quad q\cdot \pq=0, &&\quad q\cdot q=q
\end{alignat*}
and
\begin{alignat*}{3}
&p\dpu p=p, &&\quad p\dpu \pq=0, &&\quad p\dpu q=p,\\
&\pq\dpu p=\pq,&&\quad \pq\dpu \pq=p-q, &&\quad \pq\dpu q=-\pq,\\
&q\dpu p=q, &&\quad q\dpu \pq=0, &&\quad q\dpu q=q.
\end{alignat*}
Note that $\pq\cdot \pq\ne \pq\dpu\pq$, and so $\dpu\ne \cdot$.
\end{example}

\begin{corollary}\label{si s es no degenerada, s^-1 es no degenerada} If $s\in \Aut_{\Coalg}(X^2)$ is non-degenerate, then $\tilde{s}^{-1}$ is non-degenerate.
\end{corollary}

\begin{proof} By Definition~\ref{no degenerado} and Proposition~\ref{q magma coalgebra de tilde s^-1}, the morphisms $\tilde{s}^{-1}$ and $\widetilde{(\tilde{s}_{\tau})}^{-1}$ are left non-degenerate. But, we know that $\tilde{s}^{-1}$ is right non-degenerate if and only if $\widetilde{(\tilde{s}^{-1})}_{\tau}=s_{\tau}^{-1}= \widetilde{(\tilde{s}_{\tau})}^{-1}$ is left non-degenerate, which finishes the proof.
\end{proof}

\begin{proposition}\label{right no deg = left no deg} Let $\mathcal{X}$ be a regular q-magma coalgebra. Then $\mathcal{X}$ is right non-degenerate if and only if it is left non-degenerate.
\end{proposition}

\begin{proof} By Remarks~\ref{automorfismos de X^2} and~\ref{correspondencia endomorfismos de X^2 no degenerados a izquierday q magma coalgebras regulares a izquierda} the endomorphism $s=s_{\scriptscriptstyle\mathcal{X}}$ is invertible and left non-degenerate. By~Propo\-sition~\ref{q-magma coalgebra asociada a s_tau}, we know that $\mathcal{X}$ is right non-degenerate if and only if $\tilde{s}_{\tau}$ is left non-degenerate. By Proposition~\ref{q magma coalgebra de tilde s^-1} this happens if and only if $s_{\tau}^{-1}$ is left non-degenerate. Since, again by Proposition~\ref{q magma coalgebra de tilde s^-1}, the map $\tilde{s}^{-1}$ is left non-degenerate, we can apply Proposition~\ref{q-magma coalgebra asociada a s_tau} to $\tilde{s}^{-1}$, in order to get that $s_{\tau}^{-1}$ is left non-degenerate if and only if $\mathcal{X}_{\tilde s^{-1}}$ is right non-degenerate, which, by Proposition~\ref{q magma coalgebra de tilde s^-1}, means that $\mathcal{X}$ is left non-degenerate (see Definition~\ref{q-magma coalgebra no degenerada}).
\end{proof}

\begin{remark}\label{formulas para * y doble *} Assume that $\mathcal{X}$ is a non-degenerate $q$-magma coalgebra. Since $\mathcal{X}$ is regular, by Remark~\ref{automorfismos de X^2} the map $s\coloneqq s_{\scriptscriptstyle\mathcal{X}}$ is invertible. Furthermore, by Remark~\ref{correspondencia endomorfismos de X^2 no degenerados a izquierday q magma coalgebras regulares a izquierda} and Proposition~\ref{q-magma coalgebra asociada a s_tau}, the maps $s$ and~$\tilde{s}_{\tau}$ are non-degenerate. Thus, by Remark~\ref{Caracterizacion de no degenerado a derecha}, equality~\eqref{igualdades para ast} is satisfied. Let $\mathcal{Y}\coloneqq \mathcal{X}_{\tilde{s}_{\tau}}=(X^{\cop},\ast,\dast)$ be as in Notation~\ref{notacion ast y bar{ast}}. Since $\tilde{s}_{\tau}$ is bijective, by Remark~\ref{automorfismos de X^2}, we know that $\mathcal{Y}$ is regular (and so $\mathcal{Y}^{\op}$ is also). In particular, the map $\ov{G}_{\mathcal{Y}^{\op}}\colon X^2\to X^2$, given by $\ov{G}_{\mathcal{Y}^{\op}}(x,y) = x\dast y_{(1)}\ot y_{(2)}$, is bijective (see Definition~\ref{q-magma coalgebra no degenerada a izquierda}). Moreover, by Proposition~\ref{q magma coalgebra de tilde s^-1}, we know that $G_{s_{\tau}^{-1}}=\overline{G}^{-1}_{\mathcal{Y}^{\op}}$; while, by Remark~\ref{notacion bar s1 y bar s2} and~\ref{s^-1=bar s}, we also have 
$$
s^{-1}_{\tau}(x\ot y) =\tilde{s}^{-1}_{\tau}(x\ot y) = \tau\xcirc \ov{s}\xcirc\tau(x\ot y) = {y_{(2)}}_{x_{(2)}}\ot {}_{y_{(1)}}x_{(1)}, 
$$
which implies that $G_{s_{\tau}^{-1}}(x\ot y)= {}_{y_{(1)}}x\ot y_{(2)}$. Thus, arguing as in Remark~\ref{Caracterizacion de no degenerado a izquierda}, we obtain that
\begin{equation}\label{relacion dast subindice invertido}
{}_{y_{(2)}}(x\dast y_{(1)}) = {}_{y_{(1)}}x\dast y_{(2)} = \epsilon(y)x\quad\text{for all $x,y\in X$.}
\end{equation}
\end{remark} 

\begin{remark}\label{X,dast,ast caso conjuntista} When $X=kY$, our $(X,\dast,\ast)$ induces by restriction a $q$-magma on $Y$, which is canonically isomorphic to the dual $q$-magma $Y^*$ (see~\cite{R2}*{Definition~2}).
\end{remark}

\begin{proposition}\label{equivalencia no degenerado} A left non-degenerate coalgebra endomorphism $s$ is invertible and non-degenerate if~and only if $\mathcal{X}_s$ is non-degenerate.
\end{proposition}

\begin{proof} Clearly, if $\mathcal{X}_s$ is non-degenerate, then it is regular (see Definition~\ref{q-magma coalgebra no degenerada}). Moreover, by Remark~\ref{automorfismos de X^2} we know that $\mathcal{X}_s$ is regular if and only if $s$ is invertible. On the other hand, by Propositions~\ref{q-magma coalgebra asociada a s_tau} and~\ref{right no deg = left no deg}, the map $s$ is non-degenerate if and only if $\mathcal{X}_s$ is.
\end{proof}

\begin{proposition}\label{no degenerado a derecha equiv h es inversible} Let $\mathcal{X}$ be a left regular $q$-magma coalgebra and let $s\coloneqq s_{\mathcal{X}}$. Then $\mathcal{X}_{\tilde s}$ is right non-de\-generate if and only if it is left regular and the map $h$, introduced in Definition~\ref{q-magma coalgebra},~is~invertible.
\end{proposition}

\begin{proof} By Remark~\ref{correspondencia endomorfismos de X^2 no degenerados a izquierday q magma coalgebras regulares a izquierda} we know that $s$ is left non-degenerate. Thus, we have
$$
h\xcirc G_s(x\ot y)= y_{(3)}\dpu {x_{(2)}}^{y_{(2)}}\ot {x_{(1)}}^{y_{(1)}}\cdot y_{(4)}= y_{(4)}\dpu {x_{(1)}}^{y_{(1)}}\ot {x_{(2)}}^{y_{(2)}}\cdot y_{(3)} = {}^{x_{(1)}}y\ot x_{(2)}= H_{\mathcal{X}_{\tilde{s}}}(y\ot x);
$$
where the second equality holds by identity~\eqref{intercambio . :}; third one, by equalities~\eqref{igualdades para cdot} and Remark~\ref{Caracterizacion de no degenerado a izquierda}; and the last one, by Definition~\ref{q-magma coalgebra no degenerada}. Thus, since $G_s$ is invertible, $h$ is invertible if and only if $H_{\mathcal{X}_{\tilde{s}}}$ is. Hence, $\mathcal{X}_{\tilde s}$ is right non-degenerate if and only if it is left regular and $h$ is invertible (see Definition~\ref{q-magma coalgebra no degenerada}). 
\end{proof}

\begin{proposition}\label{no degenerado equiv h es inversible} Let $\mathcal{X}$ be a regular $q$-magma coalgebra and let $s\coloneqq s_{\mathcal{X}}$. Then $\mathcal{X}_{\tilde s}$ is non-de\-generate if and only if it is left regular and the map $h$, introduced in Definition~\ref{q-magma coalgebra}, is invertible.
\end{proposition}

\begin{proof} By Proposition~\ref{no degenerado a derecha equiv h es inversible}, if $\mathcal{X}_{\tilde s}$ is non-degenerate, then it is left regular and $h$ is invertible. By Propositions~\ref{right no deg = left no deg} and~\ref{no degenerado a derecha equiv h es inversible}, in order prove the converse, we must only check that $\mathcal{X}_{\tilde s}$ is regular. But, by Remark~\ref{automorfismos de X^2}, this happens if and only if $\tilde{s}$ is invertible, which is true since $s$ is invertible (again by Remark~\ref{automorfismos de X^2}).
\end{proof}

\begin{lemma}\label{cocos} Let $\mathcal{X}$, $s$, $\tilde{s}$ and $h$ be as in Proposition~\ref{no degenerado equiv h es inversible}. Assume that $\mathcal{X}_{\tilde s}$ is non-degenerate. Then,~the maps $\tilde{s}$ and $s_{\tau}$ are invertible and non-degenerate. Moreover, $\tilde{s}^{-1}_{\tau}$ is left non-degenerate~and~$\mathcal{X}_{\tilde{s}_{\tau}^{-1}} = \mathcal{X}_{s_{\tau}}^{\op}$. Finally, if we set $(X,\diam,\ddiam\hspace{0.3pt})\coloneqq \mathcal{X}_{s_{\tau}}$, then we have
\begin{equation}\label{eq:agregada}
{}_{y_{(2)}}x\ddiam y_{(1)}={}_{y_{(1)}}(x\ddiam y_{(2)})=\epsilon(y)x\quad\text{and}\quad {}^ {y_{(1)}}x\diam y_{(2)}={}^{y_{(2)}}(x\diam y_{(1)})=\epsilon(y)x,
\end{equation}
for all $x,y\in X$.
\end{lemma}

\begin{proof} By  Proposition~\ref{equivalencia no degenerado}, the map $\tilde{s}$ is invertible and non-degenerate. By Definition~\ref{no degenerado}, this implies that $s_{\tau}$ is invertible and left non-degenerate. By Proposition~\ref{q magma coalgebra de tilde s^-1}, we also know that the map $\tilde{s}^{-1}_{\tau}$ is left non-degenerate and $(X^{\cop},\ddiam\hspace{0.3pt},\diam)=\mathcal{X}_{\tilde{s}_{\tau}^{-1}}$. Note that, by Remarks~\ref{notacion bar s1 y bar s2} and~\ref{s^-1=bar s},
\begin{equation}\label{s_{tau} y tilde{s}_{tau}^{-1}}
s_{\tau}(x\ot y)= {y_{(2)}}^{x_{(2)}}\ot {}^{y_{(1)}}x_{(1)} \quad\text{and}\quad \tilde{s}_{\tau}^{-1}(x\ot y)= {y_{(2)}}_{x_{(2)}}{\ot_{\phantom{)}}}_{y_{(1)}}x_{(1)}.
\end{equation}
Hence, iden\-ti\-ties~\eqref{igualdades para cdot}  applied to $s_{\tau}$ and $\tilde{s}_{\tau}^{-1}$ yield
\begin{equation*}
{}^ {y_{(1)}}x\diam y_{(2)}={}^{y_{(2)}}(x\diam y_{(1)})=\epsilon(y)x\quad\text{and}\quad {}_{y_{(2)}}x\ddiam y_{(1)}={}_{y_{(1)}}(x\ddiam y_{(2)})=\epsilon(y)x,
\end{equation*}
for all $x,y\in X$.
\end{proof}

\begin{proposition}\label{inversa de h} Under the conditions of Lemma~\ref{cocos}, the map $\ov{h}\colon X^{\cop}\ot X \to X\ot X^{\cop}$, defined by  $\bar{h}(x\ot y)\coloneqq y_{(1)}\ddiam x_{(2)}\ot x_{(1)}\diam y_{(2)}$, is the inverse of the coalgebra morphism $h_{\mathcal{X}}$ (see Definition~\ref{q-magma coalgebra}).
\end{proposition}

\begin{proof} By Proposition~\ref{no degenerado equiv h es inversible}, we know that $h_{\mathcal{X}}$ is invertible. So, it suffices to prove that $\ov{h}\xcirc h_{\mathcal{X}} = \ide_{X\ot X^{\cop}}$. Since $(X,\diam,\ddiam\hspace{0.3pt})=\mathcal{X}_{s_{\tau}}$ is a $q$-magma coalgebra, we know that $\ov{h} = \widetilde{h_{\mathcal{X}_{s_{\tau}}}}$ is a coalgebra morphism (see~Defini\-tion~\ref{q-magma coalgebra}). Thus $\ov{h}\xcirc h_{\mathcal{X}}$ is a coalgebra morphism and consequently, by item~2) of Remark~\ref{some facts}, to check that $\ov{h}\xcirc h_{\mathcal{X}} = \ide_{X\ot X^{\cop}}$ it suffices to show that 
$$
(\ide_X\ot \epsilon)\xcirc \ov{h}\xcirc h_{\mathcal{X}}(x\ot y)= \epsilon(y)x\quad\text{and}\quad (\epsilon\ot \ide_{X^{\cop}})\xcirc \ov{h}\xcirc h_{\mathcal{X}}(x\ot y) = \epsilon(x)y. 
$$
A direct computation shows that this is equivalent to see that
\begin{equation}\label{equ1}
(x_{(1)}\cdot y_{(2)})\ddiam (y_{(1)}\dpu x_{(2)})=\epsilon(y)x\quad\text{and}\quad (y_{(1)}\dpu x_{(2)})\diam (x_{(1)}\cdot y_{(2)})=\epsilon(x)y.
\end{equation}
By Remark~\ref{notacion bar s1 y bar s2}, we have ${}_{y\dpu x_{(2)}}x_{(1)} = x_{(1)} \cdot (y\dpu x_{(3)})_{x_{(2)}} = x\cdot y$. Hence, by~\eqref{eq:agregada},
$$
(x_{(1)}\cdot y_{(2)})\ddiam (y_{(1)}\dpu x_{(2)})={}_{y_{(2)}\dpu x_{(2)}}x_{(1)}\ddiam (y_{(1)}\dpu x_{(3)})=\epsilon(y)x,
$$
as desired. To prove the second equality in~\eqref{equ1}, we note that $y\dpu x={}^{x\cdot y_{(1)}}y_{(2)}$, and so, by~\eqref{intercambio . :} and~\eqref{eq:agregada},
$$
(y_{(1)}\dpu x_{(2)})\diam (x_{(1)}\cdot y_{(2)})=(y_{(2)}\dpu x_{(1)})\diam (x_{(2)}\cdot y_{(1)})={}^{x_{(1)}\cdot y_{(2)}}y_{(3)}\diam (x_{(2)}\cdot y_{(1)})=\epsilon(x)y,
$$
which finishes the proof.
\end{proof}

\section[Very strongly regular \texorpdfstring{$q$}{q}-magma coalgebras]{Very strongly regular \texorpdfstring{$\mathbf{q}$}{q}-magma coalgebras}
In this section we begin the study of a special type of $q$-magma coalgebras that satisfy a very strong~con\-di\-tion of regularity. Examples are all the cocommutative $q$-magma coalgebras and also the Hopf $q$-braces introduced in Section~\ref{section: weak braiding operators and Hopf q-braces}.

\begin{definition}\label{strongly regular} Let $\mathcal{X}=(X,\cdot,\dpu)$ be a $q$-magma coalgebra and let $p$ and $d$ be as above Definition~\ref{q-magma coalgebra}. Assume that $\mathcal{X}$ is regular (see Definition~\ref{q-magma coalgebra no degenerada a izquierda}) and let $G_{\mathcal{X}}$ and $G_{\mathcal{X}^{\op}}$ be as in Remarks~\ref{notacion s1 y s2} and~\ref{notacion bar s1 y bar s2}, respectively. We say that $\mathcal{X}$ is {\em strongly regular} if, for each $i\in \mathds{Z}$, there exist maps $p_i\colon X^2\to X$, $d_i\colon X^2\to X$, $\mli{gp}_i\colon X^2\to X$ and $\mli{gd}_i\colon X^2\to X$, such that
\begin{equation}\label{condiciones para cero}
p_0= p,\quad d_0=d,\quad \mli{gp}_0= (X\ot \epsilon)\hspace{0.3pt} G_{\mathcal{X}},\quad \mli{gd}_0= (X\ot \epsilon)\hspace{0.3pt} G_{\mathcal{X}^{\op}}
\end{equation}
and
\begin{align}
& x^{{\underline{y_{(2)}}}_i}\cdot_{\scriptscriptstyle{i-1}} y_{(1)} = (x\cdot_{\scriptscriptstyle{i-1}} y_{(2)})^{{\underline{y_{(1)}}}_i}=\epsilon(y)x, \label{def de gp_i simplificado}\\
& x_{{\underline{y_{(1)}}}_{i-1}}\dpu_{\scriptscriptstyle i} y_{(2)} = (x\dpu_{\scriptscriptstyle i} y_{(1)})_{{\underline{y_{(2)}}}_{i-1}}=\epsilon(y)x,\label{def de gd_i simplificado}\\
& x^{{\underline{y_{(1)}}}_i}\cdot_{\scriptscriptstyle i} y_{(2)} = (x\cdot_{\scriptscriptstyle i} y_{(1)})^{{\underline{y_{(2)}}}_i}=\epsilon(y)x, \label{def de p_i simplificado}\\
& x_{{\underline{y_{(2)}}}_i}\dpu_{\scriptscriptstyle i} y_{(1)} = (x\dpu_{\scriptscriptstyle i} y_{(2)})_{{\underline{y_{(1)}}}_i}=\epsilon(y)x, \label{def de d_i simplificado}
\end{align}
where $x\cdot_{\scriptscriptstyle i} y\coloneqq p_i(x\ot y)$, $x\dpu_{\scriptscriptstyle i} y\coloneqq d_i(x\ot y)$, $x^{\underline{y}_i}\coloneqq \mli{gp}_i(x\ot y)$ and $x_{\underline{y}_i}\coloneqq \mli{gd}_i(x\ot y)$. If necessary, we will write $p^{\scriptscriptstyle \mathcal{X}}_i$, $d^{\scriptscriptstyle \mathcal{X}}_i$, $\mli{gp}^{\scriptscriptstyle \mathcal{X}}_i$ and $\mli{gd}^{\scriptscriptstyle \mathcal{X}}_i$ instead of $p_i$, $d_i$, $\mli{gp}_i$ and $\mli{gd}_i$, respectively.
\end{definition}

\begin{remark}\label{unicidad} The maps $p_i$, $d_i$, $\mli{gp}_i$ and $\mli{gd}_i$ are unique. In fact, since, by their very definitions, this is true for $p_0$, $d_0$, $\mli{gp}_0$ and $\mli{gd}_0$, the result follows from the fact that, for all $i\in \mathds{Z}$,
	
\begin{enumerate}[itemsep=0.7ex, topsep=1.0ex, label={\arabic*)}]
		
\item $(\mli{gp}_i\ot X)\hspace{0.3pt} (X\ot \Delta_{X^{\cop}})$ is the inverse of $(p_{i-1}\ot X)\hspace{0.3pt} (X\ot \Delta_{X^{\cop}})$,
		
\item $(d_i\ot X)\hspace{0.3pt} (X\ot \Delta_X)$ is the inverse of $(\mli{gd}_{i-1}\ot X)\hspace{0.3pt} (X\ot \Delta_X)$,
		
\item $(p_i\ot X)\hspace{0.3pt} (X\ot \Delta_X)$  is the inverse of $(\mli{gp}_i\ot X)\hspace{0.3pt} (X\ot \Delta_X)$,
		
\item $(\mli{gd}_i\ot X)\hspace{0.3pt} (X\ot \Delta_{X^{\cop}})$ is the inverse of $(d_i\ot X)\hspace{0.3pt} (X\ot \Delta_{X^{\cop}})$.
\end{enumerate}
\end{remark}

Note that $x\cdot_{\scriptscriptstyle 0} y = x\cdot y$, $x\dpu_{\scriptscriptstyle 0} y = x\dpu y$, $x^{\underline{y}_0}=x^y$ and $x_{\underline{y}_0}=x_y$, and that if $X$ is cocommutative, then $p_i=p_0$, $d_i=d_0$, $\mli{gp}_i= \mli{gp}_0$ and $\mli{gd}_i=\mli{gd}_0$ for all~$i$.

\begin{remark}\label{X^{op} es fr sii X es fr} If $\mathcal{X}=(X,\cdot,\dpu)$ is strongly regular, then $\mathcal{X}^{\op} =(X^{\cop},\dpu,\cdot)$ is also. In fact, in order to see this, it suffices to note that the maps $p^{\op}_i$, $d^{\op}_i$, $\mli{gp}^{\op}_i$ and $\mli{gd}^{\op}_i$, from $X^{\cop}\ot X^{\cop}$ to $X^{\cop}$, defined by 
$$
p^{\op}_i\coloneqq d_{-i},\quad d^{\op}_i\coloneqq p_{-i},\quad \mli{gp}^{\op}_i\coloneqq \mli{gd}_{-i}\quad\text{and}\quad \mli{gd}^{\op}_i\coloneqq \mli{gp}_{-i},\qquad\text{for $i\in \mathds{Z}$,}
$$
satisfy items~1)--4) of Remark~\ref{unicidad}, with $X$ replaced by $X^{\cop}$.
\end{remark}

\begin{lemma}\label{los pi, di, gpi y gdi son morfismos de coalgebra} For all $i\in \mathds{Z}$, the maps $p_i$ and $d_i$ are coalgebra morphisms from $X\ot X^{\cop}$ to $X$ and the maps $\mli{gp}_i$ and $\mli{gd}_i$ are coalgebra morphisms from $X^2$ to $X$.
\end{lemma}

\begin{proof} Since $p_0$, $d_0$, $\mli{gp}_0$, $\mli{gd}_0$ are coalgebra morphisms, to prove the result it suffices to check that
\begin{align}
&\text{$p_i$ is a coalgebra map} \Leftrightarrow \text{$\mli{gp}_{i+1}$ is a coalgebra map} \Leftrightarrow \text{$p_{i+1}$ is a coalgebra map}\label{EQ1}\\
\shortintertext{and}
&\text{$\mli{gd}_i$ is a coalgebra map} \Leftrightarrow \text{$d_{i+1}$ is a coalgebra map} \Leftrightarrow \text{$\mli{gd}_{i+1}$ is a coalgebra map}.\label{EQ2}
\end{align}
But the first equivalence in~\eqref{EQ1} follows using~\eqref{def de gp_i simplificado}, item~1) of Remark~\ref{some facts} and Proposition~\ref{UTIL} applied to~$X^{\cop}$; and the second one, using~\eqref{def de p_i simplificado} and Proposition~\ref{UTIL} applied to $X$; while the first equivalence in~\eqref{EQ2} follows using~\eqref{def de gd_i simplificado} and Proposition~\ref{UTIL} applied to $X$; and the second one, using~\eqref{def de d_i simplificado}, item~1) of Remark~\ref{some facts} and Proposition~\ref{UTIL} applied to $X^{\cop}$.
\end{proof}	

\begin{proposition}\label{indecente} Let $\mathcal{X}$ and $\mathcal{Y}$ be two strongly regular $q$-magma coalgebras. If $f\colon \mathcal{X}\to \mathcal{Y}$ is a morphism of $q$-magma coalgebras, then, for all $i\in \mathds{Z}$,
$$
f(x\cdot_{\scriptscriptstyle i} y) = f(x)\cdot_{\scriptscriptstyle i} f(y),\quad f(x\dpu_{\scriptscriptstyle i} y) = f(x)\dpu_{\scriptscriptstyle i} f(y),\quad f(x^{\underline{y}_i}) = f(x)^{\underline{f(y)}_i}\quad\text{and}\quad f(x_{\underline{y}_i}) = f(x)_{\underline{f(y)}_i}).
$$
\end{proposition}

\begin{proof} Since the above equalities hold for $i=0$, the result will be follow if we show that

\begin{enumerate}[itemsep=0.7ex, topsep=1.0ex, label={\arabic*)}]

\item $f(x\cdot_{\scriptscriptstyle i} y) = f(x)\cdot_{\scriptscriptstyle i} f(y)$ implies $f(x^{\underline{y}_i}) = f(x)^{\underline{f(y)}_i}$ and $f(x^{\underline{y}_{i+1}}) = f(x)^{\underline{f(y)}_{i+1}}$,

\item $f(x^{\underline{y}_i}) = f(x)^{\underline{f(y)}_i}$ implies $f(x\cdot_{\scriptscriptstyle i} y) = f(x)\cdot_{\scriptscriptstyle i} f(y)$ and $f(x\cdot_{\scriptscriptstyle i-1} y) = f(x)\cdot_{\scriptscriptstyle i-1} f(y)$,

\item $f(x\dpu_{\scriptscriptstyle i} y) = f(x)\dpu_{\scriptscriptstyle i} f(y)$ implies $f(x_{\underline{y}_i}) = f(x)_{\underline{f(y)}_i}$ and $f(x_{\underline{y}_{i-1}}) = f(x)_{\underline{f(y)}_{i-1}}$,

\item $f(x_{\underline{y}_i}) = f(x)_{\underline{f(y)}_i}$ implies $f(x\dpu_{\scriptscriptstyle i} y) = f(x)\dpu_{\scriptscriptstyle i} f(y)$ and $f(x\dpu_{\scriptscriptstyle i+1} y) = f(x)\dpu_{\scriptscriptstyle i+1} f(y)$.

\end{enumerate}
We prove the first two items and leave the other ones to the reader. Assume that $f(x\cdot_{\scriptscriptstyle i} y) = f(x)\cdot_{\scriptscriptstyle i} f(y)$, for all $x,y\in X$. Then, by con\-dition~\eqref{def de p_i simplificado}, we have
$$
f(x^{{\underline{y_{(1)}}}_i})\cdot_{\scriptscriptstyle i} f(y_{(2)}) = f\bigl(x^{{\underline{y_{(1)}}}_i}\cdot_{\scriptscriptstyle i} y_{(2)}\bigr) = \epsilon(y)f(x) = f(x)^{{\underline{f(y_{(1)}})}_i}\cdot_{\scriptscriptstyle i} f(y_{(2)}),
$$
which implies that
$$
f(x^{\underline{y}_i}) = \Bigl(f\bigl(x^{{\underline{y_{(1)}}}_i}\bigr)\cdot_{\scriptscriptstyle i} f(y_{(2)})\Bigr)^{{\underline{f(y_{(3)}})}_i} = \Bigl(f(x)^{{\underline{f(y_{(1)}})}_i}\cdot_{\scriptscriptstyle i} f(y_{(2)})\Bigr)^{{\underline{f(y_{(3)}})}_i} = f(x)^{\underline{f(y)}_i}.
$$
Similarly, by condition~\eqref{def de gp_i simplificado}, we have
$$
f(x^{{\underline{y_{(2)}}}_{i+1}})\cdot_{\scriptscriptstyle i} f(y_{(1)}) = f\bigl(x^{{\underline{y_{(2)}}}_{i+1}}\cdot_{\scriptscriptstyle i} y_{(1)}\bigr) = \epsilon(y)f(x) = f(x)^{{\underline{f(y_{(2)}})}_{i+1}}\cdot_{\scriptscriptstyle i} f(y_{(1)}),
$$
and hence
$$
f(x^{\underline{y}_{i+1}}) = \Bigl(f\bigl(x^{{\underline{y_{(3)}}}_{i+1}}\bigr)\cdot_{\scriptscriptstyle i} f(y_{(2)})\Bigr)^{{\underline{f(y_{(1)}})}_{i+1}} = \Bigl(f(x)^{{\underline{f(y_{(3)}})}_{i+1}}\cdot_{\scriptscriptstyle i} f(y_{(2)})\Bigr)^{{\underline{f(y_{(1)}})}_{i+1}} = f(x)^{\underline{f(y)}_{i+1}}.
$$
Thus item~1) is true. We next prove item~2). Assume that $f(x^{\underline{y}_i}) = f(x)^{\underline{f(y)}_i}$, for all $x,y\in X$. Then, by condition~\eqref{def de p_i simplificado}, we have
$$
f(x\cdot_{\scriptscriptstyle i} y_{(1)})^{{\underline{f(y_{(2)}})}_i}= f\bigl((x\cdot_{\scriptscriptstyle i} y_{(1)})^{{\underline{y_{(2)}}}_i}\bigr) =\epsilon(y)f(x) = (f(x)\cdot_{\scriptscriptstyle i} f(y_{(1)}))^{{\underline{f(y_{(2)})}}_i},
$$
and consequently,
$$
f(x\cdot_{\scriptscriptstyle i} y) = \Bigl(f(x\cdot_{\scriptscriptstyle i} y_{(1)})^{{\underline{f(y_{(2)}})}_i}\Bigr)\cdot_{\scriptscriptstyle i} f(y_{(3)}) = \Bigl(\bigl(f(x)\cdot_{\scriptscriptstyle i} f(y_{(1)})\bigr)^{{\underline{f(y_{(2)})}}_i}\Bigr)\cdot_{\scriptscriptstyle i} f(y_{(3)}) = f(x)\cdot_{\scriptscriptstyle i} f(y).
$$
Finally, by condition~\eqref{def de gp_i simplificado}, we have
$$
f(x\cdot_{\scriptscriptstyle{i-1}} y_{(2)})^{{\underline{f(y_{(1)})}}_i}= f\bigl((x\cdot_{\scriptscriptstyle{i-1}} y_{(2)})^{{\underline{y_{(1)}}}_i}\bigr) = \epsilon(y)f(x) = f(x)\cdot_{\scriptscriptstyle{i-1}} f(y_{(2)})^{{\underline{f(y_{(1)})}}_i},
$$
and so
$$
f(x\cdot_{\scriptscriptstyle{i-1}} y) = \Bigl(f(x\cdot_{\scriptscriptstyle{i-1}} y_{(3)})^{{\underline{f(y_{(2)})}}_i}\Bigr) \cdot_{\scriptscriptstyle{i-1}} f(y_{(1)}) =  \Bigl(f(x)\cdot_{\scriptscriptstyle{i-1}} f(y_{(3)})^{{\underline{f(y_{(2)})}}_i} \Bigr)\cdot_{\scriptscriptstyle{i-1}} f(y_{(1)}) = f(x)\cdot_{\scriptscriptstyle{i-1}} f(y),
$$
as desired.
\end{proof}

Let $\mathcal{X}$ be a strongly regular $q$-magma coalgebra and let $s\coloneqq s_{\mathcal{X}}$. By Remark~\ref{automorfismos de X^2}, we know that $s$ is bijective. Assume that $\tilde{s}^{-1}_{\tau}$ is left non-degenerate. By Proposition~\ref{q magma coalgebra de tilde s^-1}, we know that $s_{\tau}$ is left non-degen\-erate and $\mathcal{X}_{\tilde{s}_{\tau}^{-1}} = \mathcal{X}_{s_{\tau}}^{\op} = (X^{\cop},\ddiam\hspace{0.3pt},\diam)$, where $\diam$ and $\ddiam\hspace{0.3pt}$ are as in Lemma~\ref{cocos}. Assume that $\widehat{\mathcal{X}} \coloneqq \mathcal{X}_{\tilde{s}_{\tau}^{-1}}$ is strongly regular. For $i\in \mathds{Z}$ and~$x,y\in X$, write
\begin{equation}\label{formulas para muy fuertemente}
x\ddiam_{\!i} y\coloneqq p^{\scriptscriptstyle \widehat{\mathcal{X}}}_i(x\ot y),\quad x\diam_i y\coloneqq d^{\scriptscriptstyle \widehat{\mathcal{X}}}_i(x\ot y),\quad {}_{\underline{y}_i} x \coloneqq \mli{gp}^{\scriptscriptstyle \widehat{\mathcal{X}}}_i(x\ot y)\quad\text{and}\quad {}^{\underline{y}_i} x \coloneqq \mli{gd}^{\scriptscriptstyle \widehat{\mathcal{X}}}_i(x\ot y)
\end{equation}
Note that $x\diam_{\scriptscriptstyle 0} y = x\diam y$, $x\ddiam_{\!\scriptscriptstyle 0} y = x\ddiam y$, ${}_{\underline{y}_0} x={}_yx$ and ${}^{\underline{y}_0} x={}^yx$ (for the last two equalities use~\eqref{eq:agregada}).

\begin{definition}\label{muy fuertemente regular} Under the conditions above we say that $\mathcal{X}$ is {\em very strongly regular} if
\begin{align}
&y_{(1)}\dpu_{-i} x_{(2)}\ot x_{(1)}\cdot_i y_{(2)}=y_{(2)}\dpu_{-i} x_{(1)}\ot x_{(2)}\cdot_i y_{(1)}&&\text{for all $i\in \mathds{Z}$ and $x,y\in X$,}\label{very strongly 1}\\
&{}^{{\underline{x_{(2)}}}_{-i}}y_{(2)}\ot {x_{(1)}}^{{\underline{y_{(1)}}}_i}={}^{{\underline{x_{(1)}}}_{-i}}y_{(1)}\ot {x_{(2)}}^{{\underline{y_{(2)}}}_i} &&\text{for all $i\in \mathds{Z}$ and $x,y\in X$,}\label{very strongly 2}\\
&y_{(2)}\diam_{-i} x_{(1)}\ot x_{(2)}\ddiam_{\!i} y_{(1)}=y_{(1)}\diam_{-i} x_{(2)}\ot x_{(1)}\ddiam_{\!i} y_{(2)}&&\text{for all $i\in \mathds{Z}$ and $x,y\in X$,}\label{very strongly 3}\\
&{y_{(1)}}_{\underline{x_{(1)}}_{-i}}\ot {}_{\underline{y_{(2)}}_i}{x_{(2)}}={y_{(2)}}_{\underline{x_{(2)}}_{-i}}\ot {}_{\underline{y_{(1)}}_i}{x_{(1)}} &&\text{for all $i\in \mathds{Z}$ and $x,y\in X$.}\label{very strongly 4}
\end{align}
\end{definition}

\begin{proposition} \label{X^{op} es mfr sii X es mfr} If $\mathcal{X}=(X,\cdot,\dpu)$ is very strongly regular, then $\mathcal{X}^{\op} =(X^{\cop},\dpu,\cdot)$ is also.
\end{proposition}

\begin{proof} Let $s\coloneqq s_{\mathcal{X}}$ be the map associated with $\mathcal{X}$ according to Remark~\ref{correspondencia endomorfismos de X^2 no degenerados a izquierday q magma coalgebras regulares a izquierda}. By Remark~\ref{automorfismos de X^2}, we know~that~$s$ is bijective. By Proposition~\ref{q magma coalgebra de tilde s^-1}, the map associated with $\mathcal{X}^{\op}$, according to Remark~\ref{correspondencia endomorfismos de X^2 no degenerados a izquierday q magma coalgebras regulares a izquierda}, is $\tilde{s}^{-1}$. As we saw above $\wt{\bigr(\tilde{s}^{-1}\bigr)_{\tau}}^{-1} = s_{\tau} $ is left non-degenerate. In order to verify that $\mathcal{X}^{\op}$ is very strongly regular we must prove that

\begin{itemize}

\item[-] $\mathcal{X}^{\op}$ is strongly regular,

\item[-] $\widehat{\mathcal{X^{\op}}} = \widehat{\mathcal{X}}^{\op} = \mathcal{X}_{s_{\tau}} = (X,\diam,\ddiam\hspace{0.3pt})$ is strongly regular,

\item[-] Identities~\eqref{very strongly 1}--~\eqref{very strongly 4} are true for $\mathcal{X}^{\op}$.

\end{itemize}
The first item is true by Remark~\ref{X^{op} es fr sii X es fr}; while the second one, follows from the same remark applied to $\widehat{\mathcal{X}}$. Next we check the third condition. Identities~\eqref{very strongly 1} for $\mathcal{X}^{\op}$, reads
$$
y_{(2)}\cdot_i x_{(1)}\ot x_{(2)}\dpu_{-i} y_{(1)}=y_{(1)}\cdot_i x_{(2)}\ot x_{(1)}\dpu_{-i} y_{(2)}.
$$
But these are true by identities~\eqref{very strongly 1} for $\mathcal{X}$. On the other hand, identities~\eqref{very strongly 2} for $\mathcal{X}^{\op}$, reads
$$
{}_{\underline{x_{(1)}}_i}{y_{(1)}}\ot {x_{(2)}}_{\underline{y_{(2)}}_{-i}} = {}_{\underline{x_{(2)}}_i}{y_{(2)}}\ot {x_{(1)}}_{\underline{y_{(1)}}_{-i}}
$$
But these are true by identities~\eqref{very strongly 4} for $\mathcal{X}$. The other identities can be proved in a similar way.
\end{proof}

\begin{example} The  $q$-magma coalgebra $\mathcal{X}_s$ associated with the map $s$ in Example~\ref{ej p pq q} is a very strongly regular $q$-magma coalgebra. In fact, in this case we have
$$
p_{i}=d_{i+1}=gp_{i+1}=gd_{i+1}=p_{i+1}^{\hat{\mathcal{X}}}=d_i^{\hat{\mathcal{X}}}=gp_{i+1}^{\hat{\mathcal{X}}}=gd_{i+1}^{\hat{\mathcal{X}}} = \begin{cases} p_0 & \text{if $i$ is even,}\\ d_0 & \text{if $i$ is odd,}\end{cases}
$$
where $p_0$ and $d_0$ are the maps $\cdot$ and $\dpu$ given in Example~\ref{ej p pq q}.
\end{example}

\section[Solutions of the braid equation and \texorpdfstring{$q$}{q}-cycle coalgebras]{Solutions of the braid equation and \texorpdfstring{$\mathbf{q}$}{q}-cycle coalgebras}
In this section we introduce the notions of set-theoretic type solution of the braid equation and $q$-cycle coalgebra, and we prove that, for each coalgebra $X$, the correspondence between left non-degenerate~endo\-morphisms of $X^2$ and left regular $q$-magma coalgebra structures on $X$, established in~Section~\ref{Seccion left regular q-magma coalgebras}, induces a bijective correspondence between set-theoretic type solutions of the braid equation on $X^2$ and $q$-cycle coalgebra structures on $X$.

\begin{definition}\label{YBE} Let $X$ be a vector space. We say that a map $s\in \End_k(X^2)$ is a {\em solution of the braid equation} if $s_{12}\xcirc s_{23}\xcirc s_{12}=s_{23}\xcirc s_{12}\xcirc s_{23}$, where $s_{12}\coloneqq s\ot X$ and $s_{23}\coloneqq X\ot s$. If, moreover, $X$ is a coalgebra and $s$ is a coalgebra endomorphism  of $X^2$, then we say that $s$ is a {\em set-theoretic type} solution.
\end{definition}

\begin{remark}\label{s, tilde s, s_tau, tilde s_tau son simultaneamente} Let $X$ be a vector space. A map $s\in \End_k(X^2)$ is a solution of the braid equation if and only if $s_{\tau}$ is.
\end{remark}

\begin{proposition}\label{condiciones para YBE} A left non-degenerate coalgebra morphism $s\colon X^2\to X^2$, is a set-theoretic type solution of the braid equation, if and only if, for all $x,y,z\in X$,
\begin{enumerate}[itemsep=0.7ex, topsep=1.0ex, label={\emph{\arabic*)}}]
		
\item $\left(x^{y\cdot z_{(1)}}\right)^{z_{(2)}}=\left(x^{z\dpu y_{(2)}}\right)^{y_{(1)}}$,
		
\item $\bigl({}^{x_{(2)}}(y_{(2)}\cdot z_{(1)})\bigr)^{{}^{\bigl({x_{(1)}}^{y_{(1)}\cdot z_{(2)}}\bigr)}z_{(3)}}= {}^{x^{z\dpu y_{(2)}}}y_{(1)}$,
		
\item ${}^{x\cdot y_{(1)}}({}^{y_{(2)}}z)={}^{y\dpu x_{(2)}}({}^{x_{(1)}}z)$.
		
\end{enumerate}
\end{proposition}

\begin{proof} Let $U\coloneqq s_{12}\xcirc s_{23}\xcirc s_{12}$ and $V\coloneqq s_{23}\xcirc s_{12}\xcirc s_{23}$. Since $U$ and $V$ are coalgebra maps we know that $U=V$ if and only if $U_i=V_i$, for $i\in \{1,2,3\}$. A direct computation, using that
\begin{alignat*}{3}
& U_1=s_1\xcirc(X\ot s_1)\xcirc s_{12},&&\quad  U_2=s_2\xcirc(X\ot s_1)\xcirc s_{12},&& \quad  U_3=s_2\xcirc(s_2\ot X),\\
& V_1=s_1\xcirc(X\ot s_1),&&\quad  V_2=s_1\xcirc(s_2\ot X)\xcirc s_{23},&& \quad  V_3=s_2\xcirc(s_2\ot X)\xcirc s_{23},
\end{alignat*}
shows that $U_i(x\ot y\ot z)=V_i(x\ot y\ot z)$ for $i\in \{1,2,3\}$, if and only if
\begin{itemize}[itemsep=0.7ex, topsep=1.0ex, label={\arabic*.}]
		
\item[a)] $(x^y)^z = \bigl(x^{{}^{y_{(1)}}z_{(1)}}\bigr)^{{y_{(2)}}^{z_{(2)}}}$,
		
\item[b)] ${\bigl({}^{x_{(1)}}y_{(1)}\bigr)}^{\bigl({}^{{x_{(2)}}^{y_{(2)}}}z\bigr)} ={}^{\bigl(x^{{}^{y_{(1)}}z_{(1)}}\bigr)}{\bigl(y_{(2)}}^{z_{(2)}}\bigr)$,
		
\item[c)] ${}^{{}^{x_{(1)}}y_{(1)}}\bigl({}^{{x_{(2)}}^{y_{(2)}}}z\bigr) = {}^x({}^y z)$.

\end{itemize}
Since $s$ is a left non-degenerate map, in order to finish the proof it suffices to note that, by Remark~\ref{relacion entre s, G y dpu}, items~1) and~2) are items~a) and~b) with $x\ot y\ot z$ replaced by $x\ot y\cdot z_{(1)}\ot z_{(2)}$, while item~3) is item~c) with $x\ot y\ot z$ replaced by $x\cdot y_{(1)}\ot y_{(2)}\ot z$.
\end{proof}

In the set-theoretic context, the following definition and theorem correspond to Definition~3 and~Propo\-sition~1 of~\cite{R2}, respectively.

\begin{definition}\label{def: q cycle coalgebra} A {\em $q$-cycle coalgebra} is a left regular $q$-magma coalgebra $\mathcal{X}=(X,\cdot,\dpu)$ such that:
	
\begin{enumerate}[itemsep=0.7ex, topsep=1.0ex, label={\arabic*)}]
		
\item $(x \cdot y_{(1)})\cdot (z\dpu y_{(2)})=(x\cdot z_{(2)})\cdot (y\cdot z_{(1)})$,
		
\item $(x \cdot y_{(1)})\dpu (z\cdot y_{(2)})=(x\dpu z_{(2)})\cdot (y\dpu z_{(1)})$,
		
\item $(x \dpu  y_{(1)})\dpu (z\dpu y_{(2)})=(x\dpu z_{(2)})\dpu (y\cdot z_{(1)})$,
		
\end{enumerate}
for all $x,y,z\in X$.
\end{definition}

\begin{remark}\label{X q-cycle coalgebra equivale a X^op q-cycle coalgebra} By Remark~\ref{X regular equivale a X^op regular}, a regular $q$-magma coalgebra $\mathcal{X}$ is a $q$-cycle coalgebra if and only if $\mathcal{X}^{\op}$ is.
\end{remark}

\begin{example} The $q$-magma coalgebra $\mathcal{X}_s$ associated with the map $s$ in Example~\ref{ej p pq q} is a very strongly regular $q$-cycle coalgebra.
\end{example}

\begin{example}  Let $X$ be the coalgebra with basis $\{p,\pq,q\}$, considered in Example~\ref{ej p pq q}. The $q$-magma coalgebra structure on $X$, defined by
$$
x\cdot p\coloneqq x,\quad x \cdot q\coloneqq x,\quad x\cdot \pq\coloneqq 0\quad\text{and}\quad
x\dpu y\coloneqq \begin{cases}
                     x & \text{if $\pq\notin \{x,y\}$,}\\
                     0 & \text{if $\pq\in \{x,y\}$,}
                   \end{cases}
$$
for all $x\in\{p,\pq,q\}$, yields a $q$-cycle coalgebra, which is not right regular.
\end{example}

Our next aim is to prove the following result:

\begin{theorem}\label{soluciones de tipo conjuntista=q-cycle coalgebras} Let $s$ be a left non-degenerate coalgebra endomorphism of $X^2$ and let $\mathcal{X}$ be the left regular $q$-magma coalgebra associated with $s$. The map $s$ is a set-theoretic type solution of the braid equation if~and only if $\mathcal{X}$ is a $q$-cycle coalgebra.
\end{theorem}

\begin{lemma}\label{condicion 1 del remark} Condition~1) of Proposition~\ref{condiciones para YBE} is equivalent to condition~1) of Definition~\ref{def: q cycle coalgebra}.
\end{lemma}

\begin{proof} A direct computation, shows that the map $x\ot y \ot z \mapsto x\cdot (z_{(1)}\dpu y_{(2)})\ot y_{(1)}\ot z_{(2)}$ is invertible (with inverse $x\ot y \ot z \mapsto x^{z_{(1)}\dpu y_{(2)}}\ot y_{(1)}\ot z_{(2)}$). Hence, the first condition in Proposition~\ref{condiciones para YBE} is fulfilled if and only if $\bigl((x\cdot (z_{(1)}\dpu y_{(2)}))^{y_{(1)}\cdot z_{(2)}}\bigr)^{z_{(3)}}=\epsilon(z)x^y$, which happens if and only if $x^y\cdot z=(x\cdot (z_{(1)}\dpu y_{(2)}))^{y_{(1)}\cdot z_{(2)}}$. Replacing $x\ot y\ot z$ by $x\cdot y_{(1)}\ot y_{(2)}\ot z$ and using identity~\eqref{intercambio . :} we obtain that this~happens if and only if
$$
\epsilon(y)x\cdot z=((x\cdot y_{(1)})\cdot (z_{(1)}\dpu y_{(3)}))^{y_{(2)}\cdot z_{(2)}}=((x\cdot y_{(1)})\cdot (z_{(2)}\dpu y_{(2)}))^{y_{(3)}\cdot z_{(1)}}.
$$
Therefore, if condition~1) in Proposition~\ref{condiciones para YBE} is satisfied, then
$$
(x\cdot y_{(1)})\cdot (z\dpu y_{(2)})=((x\cdot y_{(1)})\cdot (z_{(3)}\dpu y_{(2)}))^{y_{(3)}\cdot z_{(2)}}\cdot (y_{(4)}\cdot z_{(1)})=(x\cdot z_{(2)})\cdot (y\cdot z_{(1)}),
$$
as desired. On the other hand, if condition~1) of Definition~\ref{def: q cycle coalgebra} is fulfilled, then a computation reversing the last step shows that the first condition in Proposition~\ref{condiciones para YBE} is satisfied.
\end{proof}

\begin{lemma}\label{condicion 2 del remark} Assume that condition~1) of Proposition~\ref{condiciones para YBE} is fulfilled. Then, condition~2) is~equiva\-lent to condition~2) of Definition~\ref{def: q cycle coalgebra}.
\end{lemma}

\begin{proof} Since $x\ot y \ot z \mapsto x\cdot (z_{(1)}\dpu y_{(2)})\ot y_{(1)}\ot z_{(2)}$ is invertible, the second condition in Proposition~\ref{condiciones para YBE} is equivalent to
$$
\bigl({}^{x_{(2)}\cdot (z_{(1)}\dpu y_{(4)})}(y_{(2)}\cdot z_{(3)})\bigr)^{{}^{\bigl((x_{(1)}\cdot (z_{(2)}\dpu y_{(3)}))^{y_{(1)}\cdot z_{(4)}}\bigr)}z_{(5)}}= \epsilon(z){}^xy.
$$
Evaluating this in $x\cdot y_{(1)}\ot y_{(2)}\ot z$ and using that $  {}^{x\cdot y_{(1)}}y_{(2)}=y\dpu x$, we obtain that this happens if and only if
$$
\epsilon(z)y\dpu x = \bigl({}^{(x_{(2)}\cdot y_{(1)})\cdot (z_{(1)}\dpu y_{(6)})}(y_{(4)}\cdot z_{(3)})\bigr)^{{}^{\bigl(((x_{(1)}\cdot y_{(2)})\cdot (z_{(2)}\dpu y_{(5)}))^{y_{(3)}\cdot z_{(4)}}\bigr)}z_{(5)}}.
$$
Hence, if condition~2) in Proposition~\ref{condiciones para YBE} is satisfied, then
\begin{align*}
\epsilon(z)y\dpu x & =\bigl({}^{(x_{(2)}\cdot y_{(1)})\cdot (z_{(1)}\dpu y_{(6)})}(y_{(5)}\cdot z_{(2)})\bigr)^{{}^{\bigl(((x_{(1)}\cdot y_{(2)})\cdot (z_{(4)}\dpu y_{(3)}))^{y_{(4)}\cdot z_{(3)}}\bigr)}z_{(5)}} &&\text{by equality~\eqref{intercambio . :}} \\
&=\bigl({}^{(x_{(2)}\cdot y_{(1)})\cdot (z_{(1)}\dpu y_{(5)})}(y_{(4)}\cdot z_{(2)})\bigr)^{{}^{\bigl(((x_{(1)}\cdot z_{(5)})\cdot (y_{(2)}\cdot z_{(4)}))^{y_{(3)}\cdot z_{(3)}}\bigr)}z_{(6)}} &&\text{by Definition~\ref{def: q cycle coalgebra}(1)}\\
&=\bigl({}^{(x_{(2)}\cdot y_{(1)})\cdot (z_{(1)}\dpu y_{(3)})}(y_{(2)}\cdot z_{(2)})\bigr)^{{}^{(x_{(1)}\cdot z_{(3)})}z_{(4)}} &&\text{by equality~\eqref{igualdades para cdot}}\\
&=\bigl({}^{(x_{(2)}\cdot y_{(1)})\cdot (z_{(2)}\dpu y_{(2)})}(y_{(3)}\cdot z_{(1)})\bigr)^{{}^{z_{(3)}\dpu x_{(1)}}} &&\text{by~\eqref{intercambio . :} and def. of $\dpu$}\\
&=\bigl({}^{(x_{(2)}\cdot z_{(3)})\cdot (y_{(1)}\cdot z_{(2)})}(y_{(2)}\cdot z_{(1)})\bigr)^{{}^{z_{(4)}\dpu x_{(1)}}} &&\text{by Definition~\ref{def: q cycle coalgebra}(1)}\\
&= ((y\cdot z_{(1)})\dpu (x_{(2)}\cdot z_{(2)}))^{{}^{z_{(3)}\dpu x_{(1)}}}, &&\text{by the definition of $\dpu$}
\end{align*}
and so, $(y\dpu x_{(2)})\cdot (z \dpu x_{(1)})=(y\cdot z_{(1)})\dpu (x \cdot z_{(2)})$, as desired. The converse follows by a direct computation reversing the previous steps.
\end{proof}

\begin{lemma}\label{condicion 3 del remark} If conditions~1) and~2) of Proposition~\ref{condiciones para YBE} are fulfilled, then condition~3) is equivalent to condition~3) of Def\-i\-ni\-tion~\ref{def: q cycle coalgebra}.
\end{lemma}

\begin{proof} A direct computation using that the map $x\ot y \ot z \mapsto x\cdot z_{(2)} \ot y\cdot z_{(1)} \ot z_{(3)}$ is invertible, shows that the third condition of Proposition~\ref{condiciones para YBE} is satisfied if and only if
\begin{equation}\label{a probar}
{}^{(x\cdot z_{(3)})\cdot (y_{(1)}\cdot z_{(2)})}\bigl({}^{y_{(2)}\cdot z_{(1)}} z_{(4)}\bigr) = {}^{(y\cdot z_{(1)})\dpu (x_{(2)}\cdot z_{(2)})} ({}^{x_{(1)}\cdot z_{(3)}} z_{(4)}). 
\end{equation}
Since, by Definition~\ref{def: q cycle coalgebra}(1), equality~\eqref{intercambio . :} and the definition of $\dpu$,
$$
{}^{(x\cdot z_{(3)})\cdot (y_{(1)}\cdot z_{(2)})}\bigl({}^{y_{(2)}\cdot z_{(1)}}z_{(4)}\bigr)\! =\! {}^{(x\cdot y_{(1)})\cdot (z_{(2)}\dpu y_{(2)})}\bigl({}^{y_{(3)}\cdot z_{(1)}}z_{(3)}\bigr)\!=\!{}^{(x\cdot y_{(1)})\cdot (z_{(1)}\dpu y_{(3)})}\bigl({}^{y_{(2)}\cdot z_{(2)}}z_{(3)}\bigr)\!=\!(z\dpu y_{(2)})\dpu (x\cdot y_{(1)})
$$
and, by Definition~\ref{def: q cycle coalgebra}(2) and the definition of $\dpu$,
$$
{}^{(y\cdot z_{(1)})\dpu (x_{(2)}\cdot z_{(2)})} ({}^{x_{(1)}\cdot z_{(3)}} z_{(4)}) = {}^{(y\dpu x_{(3)})\cdot (z_{(1)}\dpu x_{(2)})}(z_{(2)}\dpu x_{(1)}) =(z\dpu x_{(1)})\dpu (y\dpu x_{(2)}),
$$
this finishes the proof.
\end{proof}

\begin{proof}[Proof of Theorem~\ref{soluciones de tipo conjuntista=q-cycle coalgebras}] By Proposition~\ref{condiciones para YBE} and Lemmas~\ref{condicion 1 del remark}, \ref{condicion 2 del remark} and~\ref{condicion 3 del remark}.
\end{proof}

\begin{corollary}\label{correspondencia entr q-cycle coalgebras y soluciones} There is a bijective correspondence between left non-degenerate set-theoretic type solutions of the braid equation and $q$-cycle coalgebras. Under this bijection, invertible solutions correspond to regular $q$-cycle coalgebras, while non-degenerate invertible solutions correspond to non-degenerate $q$-cycle coalgebras.
\end{corollary}

\begin{proof} This follows immediately from Theorem~\ref{soluciones de tipo conjuntista=q-cycle coalgebras}, Proposition~\ref{equivalencia no degenerado}
and Remarks~\ref{automorfismos de X^2} and~\ref{correspondencia endomorfismos de X^2 no degenerados a izquierday q magma coalgebras regulares a izquierda}.
\end{proof}

\begin{remark} In the set-theoretic context, Corollary~\ref{correspondencia entr q-cycle coalgebras y soluciones} yields~\cite{R2}*{Proposition~1}.
\end{remark}

\begin{corollary}\label{tilde X es q cycle coalgebra} Let $s\colon X^2\to X^2$ be a left non-degenerate coalgebra morphism. If $\mathcal{X}_s$ is a non-degenerate $q$-cycle coalgebra, then $\mathcal{X}_{\tilde{s}_{\tau}}= (X^{\cop},\ast, \dast)$ is also (see Notation~\ref{notacion ast y bar{ast}}).
\end{corollary}

\begin{proof} By Corollary~\ref{correspondencia entr q-cycle coalgebras y soluciones} the coalgebra map $s\colon X^2\to X^2$ is an invertible and non-degenerate set-theoretic type solution of the braid equation. Hence, by Remark~\ref{s, tilde s, s_tau, tilde s_tau son simultaneamente} and Proposition~\ref{q-magma coalgebra asociada a s_tau}, the map $\tilde{s}_{\tau}$ is also. Thus, again by Corollary~\ref{correspondencia entr q-cycle coalgebras y soluciones}, we conclude that $\mathcal{X}_{\tilde{s}_{\tau}}$ is non-degenerate.
\end{proof}

\begin{corollary}\label{X^op es q cycle coalgebra} If $\mathcal{X}$ is a non-degenerate $q$-cycle coalgebra, then $\mathcal{X}^{\op}$ is also.
\end{corollary}

\begin{proof} This follows immediately from Proposition~\ref{q magma coalgebra de tilde s^-1} and Corollaries~\ref{si s es no degenerada, s^-1 es no degenerada} and~\ref{correspondencia entr q-cycle coalgebras y soluciones}.
\end{proof}

\begin{definition}\label{right racks coalgebras} A {\em rack coalgebra} is a coalgebra $X$ endowed with a coalgebra map $\triangleleft\colon X\ot X\to X$ such that:
	
\begin{enumerate}[itemsep=0.7ex, topsep=1.0ex, label={\arabic*)}]
		
\item The map $x\ot y\mapsto x \triangleleft y_{(1)}\ot y_{(2)}$ is invertible,
		
\item $y_{(2)}\ot x\triangleleft y_{(1)}=y_{(1)}\ot x\triangleleft y_{(2)}$ for all $x,y\in X$,
		
\item $(x\triangleleft y) \triangleleft z= (x\triangleleft z_{(2)}) \triangleleft (y \triangleleft z_{(1)})$ for all $x,y,z\in X$.
		
\end{enumerate}
\end{definition}

\begin{remark}\label{rack coalgebras como soluciones} Let  $X$ be a coalgebra and let $\triangleleft\colon X\ot X\to X$ be a map. The pair $(X,\triangleleft)$ is a rack coalgebra if and only if the map $s\colon X^2\to X^2$, defined by $s(x\ot y)\coloneqq y_{(2)}\ot x\triangleleft y_{(1)}$ is a left non-degenerate set-theoretic type solution of the braid equation. In fact, by Remark~\ref{some facts}, we know that $s$ is a coalgebra morphism if and only if $\triangleleft$ is a coalgebra morphism and $y_{(2)}\ot x\triangleleft y_{(1)}=y_{(1)}\ot x\triangleleft y_{(2)}$, for all $x,y\in X$; $s$ is left non-degenerate if and only the map $x\ot y\mapsto x \triangleleft y_{(1)}\ot y_{(2)}$ is invertible; and finally, a straightforward computation shows that $s$ satisfies $s_{12}s_{23}s_{12}=s_{23}s_{12}s_{23}$, if and only if $(x\triangleleft y)\triangleleft z=(x\triangleleft z_{(1)})\triangleleft (x\triangleleft z_{(2)})$, for all $x,y,z\in X$.
\end{remark}

\begin{example} If in a $q$-cycle coalgebra $\mathcal{X}=(X,\cdot,\dpu)$ the operator $\dpu$ is trivial, which means that \hbox{$x\dpu y=\epsilon(y)x$}, then  $(X,\triangleleft)$ is a rack coalgebra, via $x\triangleleft y\coloneqq x^y$. In fact, in this case, by Remark~\ref{notacion s1 y s2}, the solution associated with $\mathcal{X}$ via Corollary~\ref{correspondencia entr q-cycle coalgebras y soluciones}, is the map given by $s(x\ot y)\coloneqq y_{(2)}\ot x\triangleleft y_{(1)}$, which by Remark~\ref{rack coalgebras como soluciones} is a rack coalgebra. In the set-theoretic context this example is~\cite{R2}*{Example~2}.
\end{example}

\section[Weak braiding operators and Hopf \texorpdfstring{$q$}{q}-braces]{Weak braiding operators and Hopf \texorpdfstring{$\mathbf{q}$}{q}-braces}\label{section: weak braiding operators and Hopf q-braces}
In this section we introduce the notions of Hopf $q$-brace and weak braiding operator, and we prove that they are equivalent (see Theorem~\ref{pares apareados vs q-brazas}). These are adaptations of the concept of $q$-brace introduced in \cite{R2} and of a weak version of the concept of braiding operator introduced in~\cite{LYZ} to the setting of Hopf algebras. Then, we study the structure of Hopf $q$-brace in detail. In particular we obtain Hopf algebra versions of the results in \cite{R2}*{Section~3} and~\cite{LYZ}*{Theorem~1}. We also obtain some results that only make sense in the context of Hopf algebras.

\begin{definition}\label{weak braiding operator} A {\em weak braiding operator} is a pair $(H,s)$, consisting of a Hopf algebra $H$ and a set-theoretic type solution of the braid equation $s\colon H^2\to H^2$, that satisfies the following~iden\-tities:
\begin{align}
&s\xcirc(\mu\ot H)=(H\ot \mu)\xcirc (s\ot H)\xcirc (H\ot s),\label{eq:bo1}\\
&s\xcirc (H\ot \mu)=(\mu\ot H)\xcirc (H\ot s)\xcirc (s\ot H),\label{eq:bo2}\\
&s\xcirc (\eta\ot H)=H\ot\eta,\label{eq:bo3}\\
&s\xcirc (H\ot\eta)=\eta\ot H\label{eq:bo4},
\end{align}
where $\mu$ is the multiplication map of $H$ and $\eta$ is the unit of $H$. A weak braiding operator $(H,s)$ is a~{\em braid\-ing operator} if
\begin{equation}\label{eq:bo5}
\mu=\mu\xcirc s.
\end{equation}
A {\em morphism} from a weak braiding operator $(H,s)$ to a weak braiding operator $(K,s')$ is a Hopf algebra morphism $f\colon H\to K$ such that $(f\ot f) \xcirc s=s'\xcirc (f\ot f)$.
\end{definition}

\begin{remark} If $H$ is the group algebra $kG$, then our definition corresponds to the definition of braiding operator given in~\cite{LYZ}*{page~3}. In fact, our conditions~\eqref{eq:bo1}--\eqref{eq:bo5} translate directly into the conditions~4)--7) of~\cite{LYZ}, but the authors in~\cite{LYZ} require additionally that the map $\sigma\colon G\times G\to G\times G$, induced by $s$, to be bijective.
\end{remark}

\begin{remark}\label{Hopf s tau} If $(H,s)$ is a (weak) braiding operator, then $(H^{\op\cop},\tilde{s}_{\tau})$ is also. Moreover, if the antipode of $H$ is bijective, then $(H^{\cop},\tilde{s})$ and $(H^{\op},s_{\tau})$ are also (weak) braiding operators.
\end{remark}

Let $H$ and $L$ be bialgebras and let $\alpha\colon L\ot H\to L$ and $\beta\colon L\ot H\to H$ be maps. For each $h\in H$ and $l\in L$ set $l^h\coloneqq \alpha(l\ot h)$ and ${}^lh\coloneqq \beta(l\ot h)$. Recall from~\cite{Ka}*{Definition~IX.2.2}, that $(L,H,\alpha,\beta)$ is a matched pair of bialgebras if

\begin{enumerate}[itemsep=0.7ex, topsep=1.0ex, label={\arabic*)}]
	
\item $L$ is a right $H$-module coalgebra via $\alpha$,
	
\item $H$ is a left $L$-module coalgebra via $\beta$,
	
\item $\cramped{{}^l(h h')=\bigl({}^{l_{(1)}} h_{(1)}\bigr)\bigl({}^{{l_{(2)}}^{h_{(2)}}}h'\bigr)}$, for all $h,h'\in H$ and $l\in L$,
	
\item $\cramped{(l' l)^h=\bigl({l'}^{{}^{l_{(1)}}h_{(1)}}\bigr)\bigl({l_{(2)}}^{h_{(2)}}\bigr)}$, for all $h\in H$ and $l,l'\in L$,
	
\item $1^h=\epsilon(h)1$ and ${}^l 1=\epsilon(l)1$, for all $h\in H$ and $l\in L$,
	
\item ${}^{l_{(1)}}h_{(1)}\ot {l_{(2)}}^{h_{(2)}}={}^{l_{(2)}}h_{(2)}\ot {l_{(1)}}^{h_{(1)}}$, for all $h\in H$ and $l\in L$.
\end{enumerate}
In what follows we will write $(H,\alpha,\beta)$ instead of $(H,H,\alpha,\beta)$. A {\em morphism}, $f\colon (H,\alpha,\beta)\to (K,\alpha',\beta')$, is a morphism of bialgebras, $f\colon H\to K$, such that $f({}^hl) = {}^{f(h)}f(l)$ and $f(h^l) = f(h)^{f(l)}$.

\begin{theorem}\label{pares apareados y estructuras de bialgebra en el producto} For each matched pair of bialgebras $(L,H,\alpha,\beta)$, there exists a unique bialgebras structure $H\bowtie L$, called the {\em bicrossed product} of $H$ and $L$, with underlying coalgebra $H\ot L$, such that
$$
(h\ot l)(h'\ot l')=h  ({}^{l_{(1)}} h'_{(1)})\ot  ({l_{(2)}}^{h'_{(2)}}) l'\qquad\text{for all $h,h'\in H$ and $l,l'\in L$.}
$$
The bialgebras $H$ and $L$ are identified with the subbialgebras $H\ot k$ and $k\ot L$ of $H\bowtie L$ and every element $h\ot l$ of~$H\bowtie L$ can be uniquely written as a product $(h\ot 1)(1\ot l)$ of an element $h\in H$ and $l\in L$. Moreover, if $H$ and $L$ are Hopf algebras with antipode $S_H$ and $S_L$, then $H\bowtie L$ is a Hopf algebra with antipode $S$ given by $S(h\ot l)\coloneqq {}^{S_L(l_{(2)})}S_H(h_{(2)}) \ot {S_L(l_{(1)})}^{S_H(h_{(1)})}$.
\end{theorem}

\begin{proof} See \cite{Ka}*{Theorem~IX.2.3}.
\end{proof}

\begin{proposition}\label{operadores debiles vs pares apareados} Let $H$ be a Hopf algebra and let $s\colon H^2\to H^2$ be a solution of the braid equation. Then the pair $(H,s)$ is a weak braiding operator if and only if $(H,s_2,s_1)$ is a matched pair.
\end{proposition}

\begin{proof} We already know that $s$ is a coalgebra map if and only if $s_1$ and $s_2$ are coalgebra maps and~condi\-tion~6) above Theorem~\ref{pares apareados y estructuras de bialgebra en el producto} is satisfied. We will freely use this fact. A direct computation shows that
\begin{align*}
&(\ide\ot \mu)(s\ot \ide)(\ide\ot s)(h\ot k\ot l) = (\ide\ot \mu)(s\ot \ide)\bigl(h\ot {}^{k_{(1)}} l_{(1)}\ot {k_{(2)}}^{l_{(2)}}\bigr)\\
&\phantom{(\ide\ot \mu)(s\ot \ide)(\ide\ot s)(h\ot k\ot l)} = (\ide\ot \mu)\bigl({}^{h_{(1)}} ({}^{k_{(1)}} l_{(1)})\ot {h_{(2)}}^{{}^{k_{(2)}} l_{(2)}}\ot {k_{(3)}}^{l_{(3)}}\bigr)\\
& \phantom{(\ide\ot \mu)(s\ot \ide)(\ide\ot s)(h\ot k\ot l)} = {}^{h_{(1)}} ({}^{k_{(1)}} l_{(1)})\ot \bigl({h_{(2)}}^{{}^{k_{(2)}} l_{(2)}}\bigr)\bigl({k_{(3)}}^{l_{(3)}}\bigr),\\[3pt]
& s(\mu\ot \ide)(h\ot k\ot l) = s(hk\ot l) = {}^{h_{(1)}k_{(1)}} l_{(1)}\ot (h_{(2)}k_{(2)})^{l_{(2)}},\\[3pt]
& (\mu\ot \ide)(\ide\ot s)(s\ot \ide)(h\ot k\ot l) = (\ide\ot \mu)(\ide\ot s)\bigl({}^{h_{(1)}} k_{(1)}\ot {h_{(2)}}^{k_{(2)}}\ot l\bigr)\\
&\phantom{(\ide\ot \mu)(s\ot \ide)(\ide\ot s)(h\ot k\ot l)} = (\mu\ot \ide)\bigl({}^{h_{(1)}} k_{(1)}\ot {}^{{h_{(2)}}^{k_{(2)}}} l_{(1)}\ot
({h_{(3)}}^{k_{(3)}})^{l_{(2)}}\\
&\phantom{(\ide\ot \mu)(s\ot \ide)(\ide\ot s)(h\ot k\ot l)} = \bigl({}^{h_{(1)}} k_{(1)}\bigr)\bigl({}^{{h_{(2)}}^{k_{(2)}}} l_{(1)}\bigr)\ot
({h_{(3)}}^{k_{(3)}})^{l_{(2)}},
\shortintertext{and}
& s(\ide \ot \mu)(h\ot k\ot l) = s(h\ot kl) = {}^{h_{(1)}}(k_{(1)}l_{(1)})\ot {h_{(2)}}^{k_{(2)}l_{(2)}}.
\end{align*}
Hence,
\begin{equation}\label{compatibilidad}
(\ide\ot\mu)(s\ot\ide)(\ide\ot s) = s(\mu\ot\ide)\quad\text{and}\quad (\mu\ot\ide)(\ide\ot s)(s\ot\ide) = s(\ide\ot\mu)
\end{equation}
if and only if the left action $s_1$ is associative, the right action $s_2$ is associative and conditions~3) and~4) above Theorem~\ref{pares apareados y estructuras de bialgebra en el producto} are satisfied. Note now that
$$
s(h\ot 1) = {}^{h_{(1)}} 1\ot {h_{(2)}}^1 \qquad \text{and} \qquad s(1\ot h) = {}^1 h_{(1)}\ot 1^{h_{(2)}},
$$
Hence $s(h\ot 1) = 1\ot h$ and $s(1\ot h) = h\ot 1$ if and only if the left action $s_1$ is unitary, the right action $s_2$ is unitary and condition~5) above Theorem~\ref{pares apareados y estructuras de bialgebra en el producto} is satisfied.
\end{proof}

\begin{corollary}\label{weak braiding operator is lef non degenerate} If $(H,s)$ is a weak braiding operator, then $s$ is left non-degenerate.
\end{corollary}

\begin{proof} Let $h\cdot l\coloneqq h^{S(l)}$. By the previous proposition, we have
$$
(h\cdot l_{(1)})^{l_{(2)}} = h^{S(l_{(1)})l_{(2)}} = \epsilon(l)h\quad\text{and}\quad h^{l_{(1)}}\cdot l_{(2)} = h^{l_{(1)}S(l_{(2)})} = \epsilon(l)h.
$$
Hence, Remark~\ref{Caracterizacion de no degenerado a izquierda} shows that $s$ is left non-degenerate, as desired.
\end{proof}

\begin{remark}\label{comentario a operadores debiles vs pares apareados} Let $H$ be a Hopf algebra and let $s\colon H^2\to H^2$ be a $k$-linear map. The proof of Proposition~\ref{operadores debiles vs pares apareados} shows that $s$ is a coalgebra map and equalities~\eqref{eq:bo1}--\eqref{eq:bo4} hold if and only if $(H,s_2,s_1)$ is a matched~pair. Moreover, arguing as in the proof of Corollary~\ref{weak braiding operator is lef non degenerate}, we obtain that $s$ is left non-degenerate.
\end{remark}

\begin{definition}\label{q-brace} Let $H$ be a Hopf algebra and let $\mathcal{H}=(H,\cdot,\dpu)$ be a left regular $q$-cycle coalgebra. We say that~$\mathcal{H}$ is a {\em Hopf $q$-brace} if $(H,\cdot)$ and $(H,\dpu)$ are right $H^{\op}$-modules and the following equalities~hold:
\begin{equation}\label{condicion q-braza}
hk\cdot l = (h\cdot (l_{(1)}\dpu k_{(2)}))(k_{(1)}\cdot l_{(2)})\quad\text{and}\quad hk\dpu l = (h\dpu (l_{(1)}\cdot k_{(2)}))(k_{(1)}\dpu l_{(2)})
\end{equation}
for all $h,k,l\in H$. Let $\mathcal{K}=(K,\cdot,\dpu)$ be another Hopf $q$-brace. A {\em morphism} $f\colon \mathcal{H}\to \mathcal{K}$ is a Hopf algebra morphism $f\colon H\to K$ such that $f(h\cdot h')=f(h)\cdot f(h')$ and $f(h\dpu h')=f(h)\dpu f(h')$, for all $h,h'\in H$.
\end{definition}

\begin{lemma}\label{basta menos que regular} Let $\mathcal{H}=(H,\cdot,\dpu)$ be a $q$-magma coalgebra such that $H$ is a Hopf algebra. The following~asser\-tions hold:

\begin{enumerate}[itemsep=0.7ex, topsep=1.0ex, label=\emph{\arabic*)}]

\item If $(H,\dpu)$ is a right $H^{\op}$-module, then $\mathcal{H}$ is right regular. Moreover $h_k=h\dpu S(k)$, for all $h,k\in H$.

\item If $(H,\cdot)$ is a right $H^{\op}$-module and the antipode $S$, of $H$, is bijective, then $\mathcal{H}$ is left regular.~More\-over $h^k=h\cdot S^{-1}(k)$, for all $h,k\in H$.

\end{enumerate}
\end{lemma}

\begin{proof} 1)\enspace This follows from Remark~\ref{notacion bar s1 y bar s2} and the identities
$$
\epsilon(k)h= \epsilon(k)h\dpu 1= h\dpu k_{(1)}S(k_{(2)})= (h\dpu S(k_{(2)}))\dpu k_{(1)}\quad\text{and}\quad \epsilon(k)h=(h\dpu k_{(2)})\dpu S(k_{(1)}).
$$

\smallskip
	
\noindent 2)\enspace Apply item~1) to $\mathcal{H}^{\op}=(H^{\cop},\dpu,\cdot)$, taking into account that the antipode of $H^{\cop}$ is~$S^{-1}$.
\end{proof}

\begin{remark}\label{hopf q brace es regular} By Lemma~\ref{basta menos que regular}(1) every Hopf $q$-brace is regular. 
\end{remark}

\begin{proposition}\label{acciones a derecha sobre 1} Let $\mathcal{H}=(H,\cdot,\dpu)$ be a Hopf $q$-brace. We have $1\cdot h = \epsilon(h)1 = 1\dpu h$, for each $h\in H$.
\end{proposition}

\begin{proof} By the first equality in~\eqref{condicion q-braza} and identity~\eqref{intercambio . :},
$$
1\cdot h=\left(1\cdot (h_{(1)}\dpu 1)\right)(1\cdot h_{(2)})=\left(1\cdot (h_{(2)}\dpu 1)\right)(1\cdot h_{(1)})=(1\cdot h_{(2)})(1\cdot h_{(1)}).
$$
So,
$$
\epsilon(h)1=(1\cdot h_{(2)})S(1\cdot h_{(1)})=(1\cdot h_{(3)})(1\cdot h_{(2)})S(1\cdot h_{(1)})= 1\cdot h.
$$
A similar argument proves that  $1\dpu h=\epsilon(h)1$.
\end{proof}

\begin{remark}\label{remark: acciones de h sobre 1} Let $\mathcal{H}=(H,\cdot,\dpu)$ be a Hopf $q$-brace. Then, from the fact that $\cdot$ and $\dpu$ are coalgebra morphism from $H\ot H^{\cop}$ to $H$, it follows that $(H,\cdot)$ and $(H,\dpu)$ are right $H^{\op\cop}$-module coalgebras.
\end{remark}

\begin{definition}\label{Def: Hopf q-brace con antiooda bijectiva} Let $\mathcal{H}=(H,\cdot,\dpu)$ be a Hopf $q$-brace. If $H$ has bijective antipode, then we say that $\mathcal{H}$ is a {\em Hopf $q$-brace with bijective antipode}. (By Lemma~\ref{basta menos que regular}, in this case, the requirement that $\mathcal{H}$ be left regular is superfluous).
\end{definition}

\begin{remark}\label{q-braza opuesta} If $\mathcal{H}=(H,\cdot,\dpu)$ is a Hopf $q$-brace with bijective antipode, then $\mathcal{H}^{\op}= (H^{\cop},\dpu,\cdot)$ is also. In fact,~$H^{\cop}$ is a Hopf algebra with antipode $S^{-1}$; by Remark~\ref{X q-cycle coalgebra equivale a X^op q-cycle coalgebra} and~\ref{hopf q brace es regular}, we know that $\mathcal{H}^{\op}$ is a regular $q$-cycle coalgebra; and a direct computation using identity~\eqref{intercambio . :} shows that $\mathcal{H}^{\op}$ also satisfies conditions~\eqref{condicion q-braza}.
\end{remark}

\begin{example}\label{estructurasobre la algebras de Taft} Let $n\ge 2$ and let $\xi\in \mathbb{C}$ be a root of unity of order $n$. Recall that the {\em Taft algebra} $T_n$ is the Hopf algebra generated as an algebra by elements $x$ and $g$ subject to the relations  $g^n=1$, $x^n=0$ and $xg=\xi gx$. The comultiplication, counit and antipode are determined by
$$
\Delta(g) = g\ot g,\quad \Delta(x)= 1\ot x+x\ot g,\quad \epsilon(g) = 1,\quad \epsilon(x)=0,\quad S(g) = g^{-1},\quad S(x) = -xg^{-1}.
$$
A direct computation shows that $T_n$ is a Hopf $q$-brace with bijective antipode via
$$
(g^{\imath} x^{\jmath})\cdot (g^{\imath'}x^{\jmath'})\coloneqq \begin{cases} \xi^{-\imath'\jmath}{g^{\imath} x^{\jmath} } &\text{if $j'=0$,}\\0 & \text{otherwise,}\end{cases} \qquad\text{and}\qquad (g^{\imath} x^{\jmath})\dpu (g^{\imath'}x^{\jmath'})\coloneqq \begin{cases} \xi^{\imath'\jmath}{g^{\imath} x^{\jmath}} &\text{if $j'=0$,}\\ 0 & \text{otherwise.} \end{cases}
$$
Of all these algebras, only the Sweedler Hopf algebra $T_2$ is a Hopf skew-brace (see Definition~\ref{def: Hopf skew-brace}).
\end{example}

\begin{lemma}\label{algunas cuentas} Let $\mathcal{H}=(H,\cdot,\dpu)$ be a left regular $q$-magma co\-algebra. Assume that $H$ is a Hopf algebra~and let $s\colon H^2\to H^2$ be the coalgebra map associated with $\mathcal{H}$. The following assertions hold:
	
\begin{enumerate}[itemsep=0.7ex, topsep=1.0ex, label=\emph{\arabic*)}]
		
\item $H$ is a right $H^{\op}$-module via $\cdot$ if and only if $H$ is a right $H$-module via $k\ot h\mapsto k^h$.
		
\item If $h^1 = h$, then ${}^h 1 = 1\dpu h$.
		
\item $1\cdot h = \epsilon(h) 1$, for all $h\in H$ if and only if $1^h = \epsilon(h)1$, for all $h\in H$.
		
\item If $1^h = \epsilon(h) 1$, for all $h\in H$, then ${}^1 h = h\dpu 1$.
				
\item If $H$ is a right $H$-module via $k\ot h\mapsto k^h$, then $h\cdot l = h^{S(l)}$, for all $h,l\in H$.

\end{enumerate}
	
\end{lemma}

\begin{proof} 1)\enspace Suppose that $H$ is a right $H$-module via $k\ot h\mapsto k^h$ and take $h,k,l\in H$. By equalities~\eqref{igualdades para cdot},
$$
h\cdot kl = \bigl((((h\cdot l_{(1)})\cdot k_{(1)})^{k_{(2)}})^{l_{(2)}}\bigr)\cdot k_{(3)}l_{(3)} =  \bigl(((h\cdot l_{(1)})\cdot k_{(1)})^{k_{(2)}{l_{(2)}}}\bigr)\cdot k_{(3)}l_{(3)} = (h\cdot l)\cdot k.
$$
Moreover $h\cdot 1 = h^1\cdot 1 = h$. Thus $H$ is a right $H^{\op}$-module via $\cdot$. A similar argument proves the converse.
	
\smallskip
	
\noindent 2)\enspace This follows from ${}^h 1 = 1\dpu h^1$ (see Remark~\ref{notacion s1 y s2}).
	
\smallskip
	
\noindent 3)\enspace This follows from the fact that $(1\cdot h_{(1)})^{h_{(2)}} = 1^{h_{(1)}}\cdot h_{(2)} = \epsilon (h)1$.
	
\smallskip
	
\noindent 4)\enspace This follows from ${}^1 h = h_{(2)}\dpu 1^{h_{(1)}}$ (see Remark~\ref{notacion s1 y s2}).

\smallskip
	
\noindent 5)\enspace In fact, we have $h\cdot l = \bigl(h^{S(l_{(1)})l_{(2)}}\bigr)\cdot l_{(3)} =\bigl(\bigl(h^{S(l_{(1)})}\bigr)^{l_{(2)}}\bigr)\cdot l_{(3)}= h^{S(l)}$, as desired.
\end{proof}

Let $H$ be a Hopf algebra and let $\mathcal{H} = (H,\cdot,\dpu)$ be a left regular $q$-magma coal\-gebra. Let $s\colon H^2\to H^2$~be as in Remark~\ref{notacion s1 y s2}. By Remark~\ref{correspondencia endomorfismos de X^2 no degenerados a izquierday q magma coalgebras regulares a izquierda}, we know that $s$ is a left non-degenerate coalgebra endomorphisms~of~$H^2$.

\begin{lemma}\label{lema para 3 implica 2} If the first identity in~\eqref{condicion q-braza} holds, then $(l'l)^h = {l'}^{{}^{l_{(1)}}h_{(1)}}{l_{(2)}}^{h_{(2)}}$, for all $l,l',h\in H$.
\end{lemma}

\begin{proof} By the first identity in~\eqref{condicion q-braza}, Remark~\ref{notacion s1 y s2} and conditions~\eqref{eq para nenes} and~\eqref{igualdades para cdot},
\begin{align*}
\bigl({l'}^{{}^{l_{(1)}}h_{(1)}}\bigr)\bigl({l_{(2)}}^{h_{(2)}}\bigr)\cdot h_{(3)} & = \bigl({l'}^{{}^{l_{(1)}}h_{(1)}}\cdot \bigl(h_{(4)}\dpu {l_{(3)}}^{h_{(3)}}\bigr)\bigr)\bigl({l_{(2)}}^{h_{(2)}}\cdot h_{(5)}\bigr)\\
& = \bigl({l'}^{{}^{l_{(1)}}h_{(1)}}\cdot {}^{l_{(3)}} h_{(3)}\bigr)\bigl({l_{(2)}}^{h_{(2)}}\cdot h_{(4)}\bigr)\\
& = \bigl({l'}^{{}^{l_{(1)}}h_{(1)}}\cdot {}^{l_{(2)}} h_{(2)}\bigr)\bigl({l_{(3)}}^{h_{(3)}}\cdot h_{(4)}\bigr)\\
& = l'l\epsilon(h),
\end{align*}
which, again by~\eqref{igualdades para cdot}, implies the statement.
\end{proof}

Let $H$ be a Hopf algebra, let $s\colon H^2\to H^2$ be a left non-degenerate set-theoretic type solution of the braid equation and let $\cdot$ and $\dpu$ be as in Notation~\ref{notacion cdot y dpu}. By Corollary~\ref{correspondencia entr q-cycle coalgebras y soluciones}, we know that $(H,\cdot,\dpu)$ is a $q$-cycle coalgebra. We have the following result, which we will sometimes use without mentioning it.

\begin{theorem}\label{pares apareados vs q-brazas} The following assertions are equivalent:

\begin{enumerate}[itemsep=0.7ex, topsep=1.0ex, label=\emph{\arabic*)}]

\item $(H,s)$ is a weak braiding operator.

\item $(H,s_2,s_1)$ is a matched pair.

\item $(H,\cdot,\dpu)$ is a Hopf $q$-brace.

\end{enumerate}
\end{theorem}

\begin{proof} By Proposition~\ref{operadores debiles vs pares apareados} items~1) and~2) are equivalent. We next prove that items~2) and~3) are also. Assume that $(H,s_2,s_1)$ is a matched pair. By item~4) above Theorem~\ref{pares apareados y estructuras de bialgebra en el producto} and Lemma~\ref{algunas cuentas}(1), the Hopf algebra $H$ is a right $H^{\op}$-module via $\cdot$. By item~4) above Theorem~\ref{pares apareados y estructuras de bialgebra en el producto}, and~equal\-ity~\eqref{intercambio . :}, we have
\begin{align*}
\bigl((h\cdot (l_{(1)}\dpu k_{(2)}))(k_{(1)}\cdot l_{(2)})\bigr)^{l_{(3)}} &= \bigl((h\cdot (l_{(1)}\dpu k_{(3)}))^{{}^{{k_{(1)}\cdot l_{(3)}}}{l_{(4)}}}\bigr)\bigl(k_{(2)}\cdot l_{(2)}\bigr)^{l_{(5)}}\\
&= \bigl((h\cdot (l_{(1)}\dpu k_{(3)}))^{l_{(3)} \dpu k_{(1)}}\bigr)\bigl(k_{(2)}\cdot l_{(2)}\bigr)^{l_{(4)}}\\
&= \bigl((h\cdot (l_{(1)}\dpu k_{(3)}))^{l_{(2)} \dpu k_{(2)}}\bigr)\bigl(k_{(1)}\cdot l_{(3)}\bigr)^{l_{(4)}}\\
&=hk\epsilon(l),
\end{align*}
which by~\eqref{igualdades para cdot} implies that the first equality in~\eqref{condicion q-braza} is fulfilled. We next prove that $H$ is a right $H^{\op}$-module via $\dpu$. By Lemma~\ref{algunas cuentas}(4) and items~2) and~5) above Theorem~\ref{pares apareados y estructuras de bialgebra en el producto}, we have $h\dpu 1 = h$, for all $h\in H$.~Moreover,
$$
h\dpu kl = {}^{kl\cdot h_{(1)}}h_{(2)}={}^{(k\cdot(h_{(1)}\dpu l_{(2)}))(l_{(1)}\cdot h_{(2)})}h_{(3)}={}^{k\cdot(h_{(1)}\dpu l_{(2)})}(h_{(2)}\dpu l_{(1)})=(h\dpu l)\dpu k,
$$
as desired. By Corollaries~\ref{correspondencia entr q-cycle coalgebras y soluciones} and~\ref{weak braiding operator is lef non degenerate}, we know that $(H,\cdot,\dpu)$ is a $q$-cycle coalgebra. In order to finish the proof that $(H,\cdot,\dpu)$ is a Hopf $q$-brace it remains to check the second equality in~\eqref{condicion q-braza}. But, by the fact that $H$ is a right $H^{\op}$-module via $\cdot$, and items~3) and~6) above Theorem~\ref{pares apareados y estructuras de bialgebra en el producto}, we have
$$
hk\dpu l={}^{(l\cdot k_{(1)})\cdot h_{(1)}}(h_{(2)}k_{(2)})=\bigl({}^{(l_{(2)}\cdot k_{(1)})\cdot h_{(1)}}h_{(4)}\bigr)\bigl({}^{((l_{(1)}\cdot k_{(2)})\cdot h_{(2)})^{h_{(3)}}}k_{(3)}\bigr)= (h\dpu (l_{(2)}\cdot k_{(1)}))(k_{(2)}\dpu l_{(1)}),
$$
which by equality~\eqref{intercambio . :} gives the desired formula. Next we prove the converse implication. By Corollary~\ref{correspondencia entr q-cycle coalgebras y soluciones}, the map $s$, associated with $(H,\cdot,\dpu)$, is a left non-degenerate set-theoretic type solution of the braid equation. Hence, by Remark~\ref{some facts}, the maps $s_1$ and $s_2$ are coalgebra maps and~condi\-tion~6) above Theorem~\ref{pares apareados y estructuras de bialgebra en el producto} is satisfied. By Lemma~\ref{algunas cuentas}(1), the Hopf algebra $H$ is a right~\hbox{$H$-mo\-dule} via $s_2$, and so item~1) above Theorem~\ref{pares apareados y estructuras de bialgebra en el producto} also holds. By Lemma~\ref{lema para 3 implica 2}, item~4) above Theorem~\ref{pares apareados y estructuras de bialgebra en el producto} holds. Next we end the proof of item~2) above Theorem~\ref{pares apareados y estructuras de bialgebra en el producto}. By Proposition~\ref{acciones a derecha sobre 1} and items~2), 3) and~4) of Lemma~\ref{algunas cuentas}, we know that ${}^h 1 = 1\dpu h = \epsilon(h)1$, $1^h = \epsilon(h)1$ and ${}^1h=h\dpu 1=h$. Moreover, by Remark~\ref{notacion s1 y s2}, item~4) above Theorem~\ref{pares apareados y estructuras de bialgebra en el producto} and the fact that~$H$ is an $H^{\op}$-module via $\dpu$, we have 
$$
{}^{kl}h = h_{(2)}\dpu (kl)^{h_{(1)}} = h_{(3)}\dpu  \bigl(k^{{}^{l_{(1)}}h_{(1)}}\bigr)\bigl({l_{(2)}}^{h_{(2)}}\bigr) = \bigl(h_{(3)}\dpu {l_{(2)}}^{h_{(2)}}\bigr) \dpu k^{{}^{l_{(1)}}h_{(1)}} = {}^{l_{(2)}} h_{(2)} \dpu k^{{}^{l_{(1)}}h_{(1)}} = {}^k({}^lh).
$$
Consequently items~2) and~5) above Theorem~\ref{pares apareados y estructuras de bialgebra en el producto} hold. It remains to prove item~3). By Remark~\ref{notacion s1 y s2}, the second equality in~\eqref{condicion q-braza}, item~6) above Theorem~\ref{pares apareados y estructuras de bialgebra en el producto} and condition~\eqref{igualdades para cdot},
\begin{align*}
{}^l(h h') & = (h_{(2)} h'_{(2)})\dpu l^{h_{(1)} h'_{(1)}}\\ 
& = \bigl(h_{(3)}\dpu\bigl({l_{(1)}}^{h_{(1)} h'_{(1)}}\cdot h'_{(4)} \bigr)\bigr)\bigl(h'_{(3)}\dpu {l_{(2)}}^{h_{(2)} h'_{(2)}}\bigr)\\
& = \bigl(h_{(3)}\dpu\bigl(({l_{(1)}}^{h_{(1)}})^{h'_{(1)}}\cdot h'_{(4)} \bigr)\bigr)\bigl(h'_{(3)}\dpu ({l_{(2)}}^{h_{(2)}})^{h'_{(2)}}\bigr)\\
& = \bigl(h_{(3)}\dpu\bigl(({l_{(1)}}^{h_{(1)}})^{h'_{(1)}}\cdot h'_{(3)} \bigr)\bigr)\bigl({}^{{l_{(2)}}^{h_{(2)}}}h'_{(2)}\bigr)\\
& = \bigl(h_{(3)}\dpu\bigl(({l_{(2)}}^{h_{(2)}})^{h'_{(2)}}\cdot h'_{(3)} \bigr)\bigr)\bigl({}^{{l_{(1)}}^{h_{(1)}}}h'_{(1)}\bigr)\\
& = \bigl(h_{(3)}\dpu {l_{(2)}}^{h_{(2)}}\bigr)\bigl({}^{{l_{(1)}}^{h_{(1)}}}h'\bigr)\\
& = \bigl({}^{l_{(2)}} h_{(2)}\bigr)\bigl({}^{{l_{(1)}}^{h_{(1)}}}h'\bigr),
\end{align*}
which by item~6) gives item~3) above Theorem~\ref{pares apareados y estructuras de bialgebra en el producto}.
\end{proof}

\begin{remark}\label{cat iso} By Corollary~\ref{correspondencia entr q-cycle coalgebras y soluciones} and Theorem~\ref{pares apareados vs q-brazas}, the categories of weak braiding operators; Hopf $q$-braces; and matched pairs $(H,s_1,s_2)$, such that $s$ is a solution of the braid equation, are isomorphic.
\end{remark}

\begin{remark} In the set-theoretic context, the equivalence between items~2) and~3) of Theorem~\ref{pares apareados vs q-brazas} is mentioned in the remark below Definition~7 in~\cite{R2}.
\end{remark}

\begin{corollary}\label{las q-brazas son no degeneradas} Each Hopf~$q$-brace $\mathcal{H}$ is non-degenerate.
\end{corollary}

\begin{proof} By Remark~\ref{hopf q brace es regular} and Proposition~\ref{right no deg = left no deg} it suffices to check that $\mathcal{H}$ is right non-degenerate. That is, that the map $H_{\mathcal{H}}$, introduced in Definition~\ref{q-magma coalgebra no degenerada}, is bijective. But by Theorem~\ref{pares apareados vs q-brazas} and item~2) above Theorem~\ref{pares apareados y estructuras de bialgebra en el producto}, the map $H_{\mathcal{H}}$ is invertible with inverse $h\ot k \mapsto {}^{S(k_{(2)})} h\ot k_{(1)}$.
\end{proof}

\begin{remark} In the set-theoretic context, Corollary~\ref{las q-brazas son no degeneradas} corresponds to the last assertion in~\cite{R2}*{Proposi\-tion~6}.
\end{remark}

\begin{remark}\label{s en no degenarado y biyectivo} Let $(H,s)$ be a weak braiding operator. By Corollary~\ref{weak braiding operator is lef non degenerate}, we know that $s$ is left non-de\-generate. Let $\mathcal{H} = (H,\cdot,\dpu)$ be the Hopf $q$-brace associated with $(H,s)$, according to Theorem~\ref{pares apareados vs q-brazas}. By Corollaries~\ref{correspondencia entr q-cycle coalgebras y soluciones} and~\ref{las q-brazas son no degeneradas}, the map $s$ is bijective~and~non-de\-generate. By Remarks~\ref{notacion s1 y s2} and~\ref{automorfismos de X^2}, we know that $s^{-1} = \overline{G}_{\mathcal{H}}\xcirc \tau\xcirc G_{\mathcal{H}^{\op}}$. Combining this~with~Re\-mark~\ref{notacion bar s1 y bar s2}, we obtain
$$
s^{-1}(h\ot l) = \overline{G}_{\mathcal{H}}(l_{(1)}\ot h_{l_{(2)}}) = l_{(1)}\cdot {h_{(1)}}_{l_{(2)}}\ot {h_{(2)}}_{l_{(3)}} = l_{(1)}\cdot \bigl(h_{(1)}\dpu S(l_{(2)})\bigr)\ot h_{(2)}\dpu S(l_{(3)}),
$$
where the last equality follows from Lemma~\ref{basta menos que regular}(1). Moreover, it is easy to see that $(H,s^{-1})$ also is a weak braiding operator. Consequently, $(H^{\op\cop},\tilde{s}^{-1}_{\tau})$ is a weak braiding operator too (see Remark~\ref{Hopf s tau}). Note that, by Remark~\ref{s^-1=bar s} we have $\tilde{s}^{-1} = \bar{s}$, where $\bar{s}$ is as in Remark~\ref{notacion bar s1 y bar s2}. Thus $\tilde{s}_{\tau}^{-1}(h\ot l) = {l_{(2)}}_{h_{(2)}}{\ot_{\phantom{)}}}_{l_{(1)}}h_{(1)}$, and so, analogous properties to the ones given above Theorem~\ref{pares apareados y estructuras de bialgebra en el producto} are satisfied for the maps $h\ot l\mapsto l_h$ and $h\ot l\mapsto {}_lh$.
\end{remark}

\begin{corollary}\label{Hopf q-brace asociado a tilde{s}} Let $\mathcal{H}=(H,\cdot,\dpu)$ be a Hopf $q$-brace and let $(H,s)$ be the weak braiding operator associated with $\mathcal{H}$ according to Theorem~\ref{pares apareados vs q-brazas}. If $H$ has bijective antipode, then the solution $\tilde{s}$ is left non-degenerate and $\mathcal{H}_{\tilde{s}}\coloneqq (H^{\cop},\cdot_{\tilde{s}},\dpu_{\tilde{s}})$ is a Hopf $q$-brace with bijective antipode.
\end{corollary}

\begin{proof} Since $(h^{k_{(2)}})^{S^{-1}(k_{(1)})}=h^{k_{(2)}S^{-1}(k_{(1)})}=\epsilon(k)h$, for all $h,k\in H$, the solution $\tilde{s}$ is left non-degenerate (see Remark~\ref{Caracterizacion de no degenerado a izquierda}). Moreover, clearly $(H^{\cop},\tilde{s})$ is a weak braiding operator, since $(H,s)$ is. The result follows now immediately from Theorem~\ref{pares apareados vs q-brazas}.
\end{proof}

\begin{remark} Let $H$ be a Hopf algebra, let $s\colon H^2\to H^2$ be a left non-degenerate coalgebra endomorphism and let $\cdot$ and $\dpu$ be as in Notation~\ref{notacion cdot y dpu}. By Remark~\ref{correspondencia endomorfismos de X^2 no degenerados a izquierday q magma coalgebras regulares a izquierda}, we know that $(H,\cdot,\dpu)$ is a left regular $q$-magma~co\-algebra. The proof of Theorem~\ref{pares apareados vs q-brazas} shows that the following facts are equivalent:

\begin{enumerate}[itemsep=0.7ex, topsep=1.0ex, label= \arabic*)]

\item $s$ satisfies conditions~\eqref{eq:bo1}, \eqref{eq:bo2}, \eqref{eq:bo3} and~\eqref{eq:bo4}.

\item $(H,s_2,s_1)$ is a matched pair.

\item $(H,\cdot)$ and $(H,\dpu)$ are right $H^{\op}$-modules and conditions~\eqref{condicion q-braza} are satisfied.

\end{enumerate}
Moreover, arguing as in the proof of Corollary~\ref{las q-brazas son no degeneradas}, we obtain that if the conditions in items~3) are satisfied, then $(H,\cdot,\dpu)$ is non-degenerate. Thus, by Proposition~\ref{equivalencia no degenerado}, the map $s$ is invertible and non-degenerate. The rest of Remark~\ref{s en no degenarado y biyectivo} and Corollary~\ref{Hopf q-brace asociado a tilde{s}} can also be generalized to this context.
\end{remark}

The following proposition is taken from~\cite{FS}*{Lemma~3.6}.

\begin{proposition}\label{2 de 3'} Let $H$ be a Hopf algebra and let $s_1,s_2\colon H^2\to H$ be two maps. Write ${}^hk\coloneqq s_1(h\ot k)$ and $h^k\coloneqq s_2(h\ot k)$. Assume that condition~\eqref{eq:bo5} is fulfilled. Then, we have:
	
\begin{enumerate}[itemsep=0.7ex, topsep=1.0ex, label=\emph{\arabic*)}]
		
\item If $s_1$ is a coalgebra map, then $s_2$ is a coalgebra map if and only if condition~\eqref{eq para nenes} is satisfied.
		
\item If $s_2$ is a coalgebra map, then $s_1$ is a coalgebra map if and only if condition~\eqref{eq para nenes} is satisfied.
		
\end{enumerate}
	
\end{proposition}

\begin{proof} 1) Since ${}^{h_{(1)}}k_{(1)}{h_{(2)}}^{k_{(2)}} = hk$, we have $h^k = S\bigl({}^{h_{(1)}}k_{(1)} \bigr)h_{(2)} k_{(2)}$. Thus, 
$$
\epsilon(h^k) = \epsilon(h)\epsilon(k)\quad\text{and}\quad (h^k)_{(1)}\ot (h^k)_{(2)} = {h_{(1)}}^{k_{(1)}}\ot {h_{(2)}}^{k_{(2)}}  
$$ 
if and only if
$$
S\bigl({}^{h_{(2)}}k_{(2)} \bigr)h_{(3)} k_{(3)}\ot S\bigl({}^{h_{(1)}}k_{(1)} \bigr)h_{(4)} k_{(4)} = S\bigl({}^{h_{(1)}}k_{(1)} \bigr)h_{(2)} k_{(2)}\ot S\bigl({}^{h_{(3)}}k_{(3)} \bigr)h_{(4)} k_{(4)}.
$$
But the last equality is fulfilled if and only if ${h_{(2)}}^{k_{(2)}}\ot S\bigl({}^{h_{(1)}}k_{(1)} \bigr) = {h_{(1)}}^{k_{(1)}}\ot S\bigl({}^{h_{(2)}}k_{(2)} \bigr)$. Hence from condition~\eqref{eq para nenes} it follows that $s_2$ is a coalgebra map. Suppose now that this happens. Then,
\begin{align*}
{}^{h_{(1)}}k_{(1)}{h_{(2)}}^{k_{(2)}}\ot  {}^{h_{(3)}}k_{(3)}{h_{(4)}}^{k_{(4)}} & = h_{(1)}k_{(1)}\ot h_{(2)}k_{(2)}\\
&= (hk)_{(1)} \ot (hk)_{(2)}\\
&=  \left({}^{h_{(1)}}k_{(1)}{h_{(2)}}^{k_{(2)}}\right)_{(1)} \ot \left({}^{h_{(1)}}k_{(1)}{h_{(2)}}^{k_{(2)}}\right)_{(2)}\\
& = {}^{h_{(1)}}k_{(1)}{h_{(3)}}^{k_{(3)}}\ot {}^{h_{(2)}}k_{(2)}{h_{(4)}}^{k_{(4)}}.
\end{align*}
Hence ${h_{(1)}}^{k_{(1)}}\ot  {}^{h_{(2)}}k_{(2)}={h_{(2)}}^{k_{(2)}}\ot  {}^{h_{(1)}}k_{(1)}$,
as desired.
\end{proof}

The following is the version for Hopf algebras of~\cite{LYZ}*{Theorem~1}.

\begin{theorem}\label{Condicion suficiente para braided operator} Let $H$ be a Hopf algebra, let $s_1,s_2\colon H\ot H\to H$ two maps and let $s\coloneqq (s_1\ot s_2)\xcirc \Delta_{H^2}$. Assume that $H$ is a left $H$-module coalgebra via $s_1$ and a right $H$-module coalgebra via $s_2$. If 
\begin{equation}\label{smuesmu}
^{h_{(1)}}l_{(1)} {h_{(2)}}^{l_{(2)}} = hl\quad\text{for all $h,l\in H$,} 
\end{equation}
where $^hl\coloneqq s_1(h\ot l)$ and $h^l\coloneqq s_2(h\ot l)$, then $(H,s)$ is a braiding operator. In particular $s$ is a bijective~solu\-tion of the braid equation.
\end{theorem}

\begin{proof} To begin note that, since $s_1$ and $s_2$ are coalgebra morphisms, $(\ide\ot\epsilon)\xcirc s = s_1$ and $(\epsilon\ot \ide)\xcirc s = s_2$. Moreover, by Proposition~\ref{2 de 3'}, identity~\eqref{eq para nenes} is fulfilled. Hence, $s$ is a coalgebra morphism, by Remark~\ref{some facts}(3). We now prove that $s$ is a solution of braid equation. Let $A\coloneqq s_{12}\xcirc s_{23}\xcirc s_{12}$ and $B\coloneqq s_{23}\xcirc s_{12}\xcirc s_{23}$. Sin\-ce~$A$~and~$B$ are coalgebra maps we know that $A=B$ if and only if $A_i=B_i$, for $i\in \{1,2,3\}$. A direct computation using that $A_1=s_1\xcirc(H\ot s_1)\xcirc s_{12}$, $B_1=s_1\xcirc(H\ot s_1)$ and~\hbox{$\mu\xcirc s = \mu$}, shows that
$$
A_1(h\ot l\ot k) = {}^{{}^{h_{(1)}}l_{(1)}}\bigl({}^{{h_{(2)}}^{l_{(2)}}}k\bigr) = {}^{{}^{h_{(1)}}l_{(1)} {h_{(2)}}^{l_{(2)}}}k = {}^{hl} k = {}^h({}^l k) = B_1(h\ot l\ot k),
$$
and a similar computation proves that $A_3 = B_3$. Using again that $\mu\xcirc s = \mu$, we obtain that
\begin{equation}\label{equa1}
\begin{aligned}
A_1(h_{(1)}\ot l_{(1)}\ot k_{(1)})&A_2(h_{(2)}\ot l_{(2)}\ot k_{(2)})A_3(h_{(3)}\ot l_{(3)}\ot k_{(3)}) =  hlk\\
& = B_1(h_{(1)}\ot l_{(1)}\ot k_{(1)})B_2(h_{(2)}\ot l_{(2)}\ot k_{(2)})B_3(h_{(3)}\ot l_{(3)}\ot k_{(3)}).
\end{aligned}
\end{equation}
Since $A_1$ and $A_3$ are convolution invertible (because they are coalgebra maps) from equality~\eqref{equa1} it follows that $A_2 = B_2$. So $s$ is a solution, as desired. We next check identity~\eqref{eq:bo1}. Let $C\coloneqq (H\ot \mu)\xcirc (s\ot H)\xcirc (H\ot s)$ and $D\coloneqq s\xcirc(\mu\ot H)$. In order to prove that $C=D$ it suffices to verify that $C_1=D_1$ and $C_2=D_2$. Using that $C_1 = s_1\xcirc (H\ot s_1)$ and $D_1 = s_1\xcirc(\mu\ot H)$, we obtain
$$
C_1(h\ot l\ot k) = {}^h({}^lk) = {}^{hl}k = D_1(h\ot l\ot k).
$$   
Using now that $\mu = \mu\xcirc s$, we obtain
$$
C_1(h_{(1)}\ot l_{(1)}\ot k_{(1)})C_2(h_{(2)}\ot l_{(2)}\ot k_{(2)}) = hkl = D_1(h_{(1)}\ot l_{(1)}\ot k_{(1)})D_2(h_{(2)}\ot l_{(2)}\ot k_{(2)}),
$$
which, by the same argument as above, implies that $C_2 = D_2$, as desired. Thus, identity~\eqref{eq:bo1} holds.~Sim\-ilar arguments prove~iden\-tities~\eqref{eq:bo2}--\eqref{eq:bo4}. Finally, $s$ is bijective, by Remark~\ref{s en no degenarado y biyectivo}.
\end{proof}

\begin{proposition}\label{Las Hopf q-brazas son fuertemente regulares} Let $\mathcal{H}=(H,\cdot,\dpu)$ be a Hopf $q$-brace with bijective antipode. For each $i\in \mathds{Z}$, let~$p_i$,~$d_i$, $\mli{gp}_i$, $\mli{gd}_i$ be the maps, from $H^2$ to $H$, defined by:
$$
p_i(h\ot k)\!\coloneqq\! h\!\cdot \!S^{2i}(k),\!\quad d_i(h\ot k)\! \coloneqq\! h\!\dpu\! S^{2i}(k),\!\quad \mli{gp}_i(h\ot k)\! \coloneqq\! h\!\cdot\! S^{2i-1}(k)\!\quad\text{and}\quad\! \mli{gd}_i(h\ot k)\!\coloneqq\! h\!\dpu\! S^{2i+1}(k).
$$
These satisfy identities~\eqref{condiciones para cero}--\eqref{def de d_i simplificado}. Consequently, $\mathcal{H}$ is strongly regular.
\end{proposition}

\begin{proof} By definition and Lemma~\ref{basta menos que regular}, equalities~\eqref{condiciones para cero} are satisfied. It remains to check that equalities~\eqref{def de gp_i simplificado}--\eqref{def de d_i simplificado} are also. But a direct computation shows that
\begin{align*}
& h^{{\underline{k_{(2)}}}_i}\cdot_{\scriptscriptstyle{i-1}} k_{(1)} = (h\cdot S^{2i-1}(k_{(2)}))\cdot S^{2i-2}(k_{(1)})= h\cdot S^{2i-2}(k_{(1)})S^{2i-1}(k_{(2)})=\epsilon(k)h,\\
&(h\cdot_{\scriptscriptstyle{i-1}} k_{(2)})^{{\underline{k_{(1)}}}_i} = (h\cdot S^{2i-2}(k_{(2)}))\cdot S^{2i-1}(k_{(1)})= h\cdot S^{2i-1}(k_{(1)})S^{2i-2}(k_{(2)})=\epsilon(k)h,\\
& \left(h\dpu_{\scriptscriptstyle i} k_{(1)}\right)_{{\underline{k_{(2)}}}_{i-1}} = (h\dpu S^{2i}(k_{(1)}))\dpu S^{2i-1}(k_{(2)})= h\dpu S^{2i-1}(k_{(2)})S^{2i}(k_{(1)})=\epsilon(k)h,\\
& h_{{\underline{k_{(1)}}}_{i-1}}\dpu_{\scriptscriptstyle i} k_{(2)} = (h\dpu S^{2i-1}(k_{(1)}))\dpu S^{2i}(k_{(2)})= h\dpu S^{2i}(k_{(2)})S^{2i-1}(k_{(1)})= \epsilon(k)h,\\
&{h^{{\underline{k_{(1)}}}_i}} \cdot_{\scriptscriptstyle{i}} k_{(2)}=(h\cdot S^{2i-1}(k_{(1)}))\cdot S^{2i}(k_{(2)})= h\cdot S^{2i}(k_{(2)})S^{2i-1}(k_{(1)})= \epsilon(k)h,\\
& (h\cdot_{\scriptscriptstyle{i}} k_{(1)})^{{\underline{k_{(2)}}}_i}= (h\cdot S^{2i}(k_{(1)}))\cdot S^{2i-1}(k_{(2)})= h\cdot S^{2i-1}(k_{(2)})S^{2i}(k_{(1)})=\epsilon(k)h,\\
&h_{{\underline{k_{(2)}}}_i}\dpu_{\scriptscriptstyle i} k_{(1)} = (h\dpu S^{2i+1}(k_{(2)}))\dpu S^{2i}(k_{(1)})=h\dpu S^{2i}(k_{(1)})S^{2i+1}(k_{(2)})=\epsilon(k)h
\shortintertext{and}
&(h\dpu_{\scriptscriptstyle i} k_{(2)})_{{\underline{k_{(1)}}}_i} = (h\dpu S^{2i}(k_{(2)}))\dpu S^{2i+1}(k_{(1)})= h\dpu S^{2i+1}(k_{(1)})S^{2i}(k_{(2)})=\epsilon(k)h,
\end{align*}
as desired.
\end{proof}

\begin{corollary}\label{sub, supra y S} Let $\mathcal{H}= (H,\cdot,\dpu)$ be a Hopf $q$-brace with bijective antipode. For all $i\in \mathds{Z}$ and $h,k\in H$,
\begin{equation}
h^{S^{2i}(k)}= h^{{\underline{k}}_i},\quad h^{S^{2i+1}(k)}= h\cdot_{\scriptscriptstyle i} k,\quad h_{S^{2i}(k)} = h_{{\underline k}_i} \quad\text{and}\quad h_{S^{2i+1}(k)}= h\dpu_{\scriptscriptstyle i+1} k.\label{ecua2}
\end{equation}
\end{corollary}

\begin{proof} By Proposition~\ref{Las Hopf q-brazas son fuertemente regulares} and identity~\eqref{igualdades para cdot}
$$
\epsilon(k)h=h^{{S^{2i}(k)}_{(1)}}\cdot {S^{2i}(k)}_{(2)} =h^{S^{2i}(k_{(1)})}\cdot S^{2i}(k_{(2)})=h^{S^{2i}(k_{(1)})}\cdot_{\scriptscriptstyle i} k_{(2)}.
$$
Hence the map $h\ot k\mapsto h^{S^{2i}(k_{(1)})}\ot k_{(2)}$ is right inverse of $(p_i\ot H)\hspace{0.3pt} (H\ot\Delta_H)$. Thus, by Remark~\ref{unicidad}(3), we have $h^{S^{2i}(k)}= h^{{\underline{k}}_i}$. The other claims follow similarly.
\end{proof}

\begin{remark}\label{comentarios'} Let $\mathcal{H} = (H,\cdot,\dpu)$ be a Hopf $q$-brace and let $(H,s)$ be the weak braiding operator associated with~$\mathcal{H}$ (see Theorem~\ref{pares apareados vs q-brazas}). By Remark~\ref{s en no degenarado y biyectivo}, the pair $(H^{\op\cop},\tilde{s}^{-1}_{\tau})$ is a weak braiding operator. Hence, by Theorem~\ref{pares apareados vs q-brazas} and Lemma~\ref{cocos}, we know that $\mathcal{H}_{\tilde{s}^{-1}_{\tau}} = (H^{\op\cop},\ddiam\hspace{0.3pt},\diam)$ is a Hopf~$q$-brace. Assume now that $S$ is bijective. Then, the third equality in~\eqref{ecua2}, applied to $\mathcal{H}_{\tilde{s}^{-1}_{\tau}}$, says that~${}^{S^{2i}(k)}h = {}^{{{\underline{k}}_i}}h$, for all~\hbox{$i\in \mathds{Z}$}.
\end{remark}

\begin{proposition}\label{acciones sobre S^j(h)} Let $\mathcal{H} = (H,\cdot,\dpu)$ be a Hopf $q$-brace. For all $j\in \mathds{N}$, we have
\begin{align*}
& S^j(h)\cdot k = S^j\bigl(h_{(i_1)}\cdot \bigl(\bigl(\cdots \bigl(\bigl(k\dpu S^j(h_{(i_2)})\bigr)\dpu S^{j-1}(h_{(i_3)})\bigr)\cdots \bigr)\dpu S^2(h_{(i_j)})\bigr)\dpu S(h_{(i_{j+1})})\bigr)
\shortintertext{and}
& S^j(h)\dpu k = S^j\bigl(h_{(i_1)}\dpu \bigl(\bigl(\cdots \bigl(\bigl(k\cdot {{S^j(h_{(i_2)})}}\bigr)\cdot {S^{j-1}(h_{(i_3)})}\bigr)\cdots \bigr)\cdot {S^2(h_{(i_j)})}\bigr)\cdot S(h_{(i_{j+1})})\bigr),
\end{align*}
where
$$
(i_1,\dots,i_{j+1})\coloneqq \begin{cases} (n+1, 1, 2n+1, 2, 2n, 3,\dots, n-1,n+3, n, n+2)\quad &\text{if $j=2n$,}\\ (n+1, 2n+2, 1, 2n+1, 2, 2n, \dots, n-1,n+3, n, n+2)\quad &\text{if $j=2n+1$.}
\end{cases}
$$
For example, the tuples obtained from this definition for $j = 1,2,3,4,5$ and $6$ are
$$
(1,2),\quad (2,1,3),\quad (2,4,1,3),\quad (3,1,5,2,4),\quad (3,6,1,5,2,4)\quad\text{and}\quad (4,1,7,2,6,3,5),
$$
respectively.
\end{proposition}
	
\begin{proof} By Proposition~\ref{acciones a derecha sobre 1}, the first equality in~\eqref{condicion q-braza} and equality~\eqref{intercambio . :}, we have,
$$
\epsilon(hk)1 = \epsilon(h)1\cdot k = S(h_{(1)})h_{(2)}\cdot k=(S(h_{(1)})\cdot (k_{(2)}\dpu h_{(2)}))(h_{(3)}\cdot k_{(1)}).
$$
Therefore
\begin{equation}\label{caso 1}
S(h\cdot k)= \epsilon(h_{(1)}k_{(2)}) S(h_{(2)}\cdot k_{(1)}) = (S(h_{(1)})\cdot (k_{(3)}\dpu h_{(2)}))(h_{(3)}\cdot k_{(2)})S(h_{(4)}\cdot k_{(1)})=S(h_{(1)})\cdot (k \dpu h_{(2)}),
\end{equation}
and so,
\begin{equation}\label{caso j=1.}
S(h)\cdot k=S(h_{(1)})\cdot \bigl((k\dpu S(h_{(3)}))\dpu h_{(2)}\bigr)= S\left(h_{(1)}\cdot (k\dpu S(h_{(2)}))\right),
\end{equation}
which is the first formula for $j=1$. Applying now this equality twice we obtain that
$$
S^2(h)\cdot k = S\bigl(S(h)_{(1)}\cdot \bigl(k\dpu S\bigl(S(h)_{(2)}\bigr)\bigr)\bigr) = S\bigl(S(h_{(2)})\cdot \bigl(k\dpu S^2(h_{(1)})\bigr)\bigr) = S^2\bigl(h_{(2)}\cdot \bigl(\bigl(k\dpu S^2(h_{(1)})\bigr)\dpu S(h_{(3)})\bigr)\bigr),
$$
which is the first formula for $j=2$. An inductive argument following these lines proves the general case for the first formula. The formula for $S^j(h)\dpu k$ can be proved in the similar way. For illustration we give the proof for $j=1$. Arguing as above, from Proposition~\ref{acciones a derecha sobre 1}, the second equality in~\eqref{condicion q-braza} and equality~\eqref{intercambio . :}, we obtain $S(h\dpu k)=S(h_{(1)})\dpu (k\cdot h_{(2)})$. Hence, from equality~\eqref{def de gp_i simplificado} with $i=1$, we obtain
\begin{equation}\label{caso j=1:}
S(h)\dpu k=S(h_{(1)})\dpu \bigl((k\cdot S(h_{(3)})) \cdot h_{(2)}\bigr)=  S\left(h_{(1)}\dpu (k\cdot S(h_{(2)}))\right),
\end{equation}
as desired.
\end{proof}

\begin{proposition}\label{acciones sobre S^{-j}(h)} Let $\mathcal{H}$ be a Hopf $q$-brace with bijective antipode. For all $j\in \mathds{N}$, we have
\begin{align*}
& S^{-j}(h)\cdot k = S^{-j}\bigl(h_{(i_1)}\cdot \bigl(\bigl(\cdots \bigl(\bigl(k\dpu S^{-j}(h_{(i_2)})\bigr)\dpu S^{-j+1}(h_{(i_3)})\bigr)\cdots \bigr)\dpu S^{-2}(h_{(i_j)})\bigr)\dpu S^{-1}(h_{(i_{j+1})})\bigr)
\shortintertext{and}
& S^{-j}(h)\dpu k = S^{-j}\bigl(h_{(i_1)}\dpu \bigl(\bigl(\cdots \bigl(\bigl(k\cdot {{S^{-j}(h_{(i_2)})}}\bigr)\cdot{S^{-j+1}(h_{(i_3)})}\bigr)\cdots \bigr)\cdot {S^{-2}(h_{(i_j)})}\bigr)\cdot S^{-1}(h_{(i_{j+1})})\bigr),
\end{align*}
where
$$
(i_1,\dots,i_{j+1})\coloneqq\begin{cases} (n+1,2n+1,1,2n,2,2n-1,3,\dots,n+3,n-1,n+2,n)\quad &\text{if $j=2n$,}\\ (n+2,1,2n+2,2,2n+1,3,\dots,n+4,n,n+3, n+1)\quad &\text{if $j=2n+1$.}
\end{cases}
$$
For example, the tuples obtained from this definition for $j = 1,2,3,4,5$ and $6$ are
$$
(2,1),\quad (2,3,1),\quad (3,1,4,2),\quad (3,5,1,4,2),\quad (4,1,6,2,5,3)\quad\text{and}\quad (4,7,1,6,2,5,3),
$$
respectively.
\end{proposition}

\begin{proof} Apply~Proposition~\ref{acciones sobre S^j(h)} to $\mathcal{H}^{\op}$, which is a Hopf $q$-brace by Remark~\ref{q-braza opuesta}.
\end{proof}

\begin{proposition}\label{compatibilidad de s con S} Let $\mathcal{H}$ be a Hopf $q$-brace and let $(H,s)$ be the weak braiding operator associated with~$\mathcal{H}$. Then $s$ is bijective and 
$$
\begin{tikzcd}
H^{\op\cop}\ot H^{\op\cop} \arrow[r, "\tilde{s}_{\tau}^{-1}"] \arrow[d, "S\ot S"] & H^{\op\cop}\ot H^{\op\cop} \arrow[d, "S\ot S"] \\
H\ot H \arrow[r,  "s"] & H\ot H
\end{tikzcd}
$$
is a commutative diagram of coalgebra morphisms.
\end{proposition}

\begin{proof} We already know that all the maps are coalgebra morphisms. Moreover, by Remark~\ref{s en no degenarado y biyectivo}, the map~$s$ is bijective. Since, by Remark~\ref{s en no degenarado y biyectivo}, we know that $\tilde{s}^{-1}_{\tau}(h\ot l) = {l_{(2)}}_{h_{(2)}}\ot {}_{l_{(1)}}h_{(1)}$, to prove that the diagram commutes, we must check that
$$
{{}^{S(h)_{(1)}}}{S(l)_{(1)}}\ot {S(h)_{(2)}}^{S(l)_{(2)}}= S\bigl({l_{(2)}}_{h_{(2)}}\bigr)\ot S\bigl({}_{l_{(1)}}h_{(1)}\bigr)\quad\text{for all $h,l\in H$.}
$$
By Remark~\ref{some facts}, for this it  suffices to check that ${{}^{S(h)}}S(l)=S(l_h)$ and $S(h)^{S(l)}=S({}_lh)$. But,
$$
{}^{S(h)}S(l) = S(l)_{(2)}\dpu S(h)^{S(l)_{(1)}} = S(l_{(1)})\dpu S(h)^{S(l_{(2)})} = S(l_{(1)})\dpu (S(h)\cdot l_{(2)}) = S(l\dpu S(h)) = S(l_h),
$$
where the first equality holds by Remark~\ref{notacion s1 y s2}; the third one, by Lemma~\ref{algunas cuentas}(5); the fourth one, by the second equality in Proposition~\ref{acciones sobre S^j(h)} with $j=1$ (which is~\eqref{caso j=1:}); and the last one, by Lemma~\ref{basta menos que regular}(1). Finally, 
$$ 
S\bigl({}_lh\bigr)= S(h_{(1)}\cdot l_{h_{(2)}}) = S(h_{(1)}\cdot (l\dpu S(h_{(2)})) = S(h)\cdot l=S(h)^{S(l)},
$$
where the first equality holds by Remark~\ref{notacion bar s1 y bar s2}; the second one, by Lemma~\ref{basta menos que regular}(1); the third one, by the second equality in Proposition~\ref{acciones sobre S^j(h)} with $j=1$ (which is~\eqref{caso j=1.}); and the last one, by Lemma~\ref{algunas cuentas}(5).
\end{proof}

\begin{corollary}\label{dualidad} Let $\mathcal{H}$  be a Hopf $q$-brace and let $(H,s)$ be the weak braiding operator associated with~$\mathcal{H}$. Then $s\xcirc (S^2\ot S^2)=(S^2\ot S^2)\xcirc s$.
\end{corollary}

\begin{proof} Proposition~\ref{compatibilidad de s con S}, applied to $(H^{\op\cop},\tilde{s}_{\tau}^{-1})$, shows that $\tilde{s}_{\tau}^{-1}\xcirc (S\ot S) = (S\ot S)\xcirc s$. The result follows easily from this and Proposition~\ref{compatibilidad de s con S}.
\end{proof}

\begin{proposition}\label{Hopf q-brace is very strongly regular} Each Hopf $q$-brace $\mathcal{H}= (H,\cdot,\dpu)$, with bijective antipode, is very strongly regular.
\end{proposition}

\begin{proof} By Pro\-position~\ref{Las Hopf q-brazas son fuertemente regulares} and Remark~\ref{comentarios'}, both $\mathcal{H}$ and $\wh{\mathcal{H}}\coloneqq (H^{\op\cop},\ddiam,\diam)$ are strongly regular. To finish the proof we must check~iden\-tities~\eqref{very strongly 1}--\eqref{very strongly 4}. We prove the first two equalities and leave the remaining ones to the reader. Since $S$ in bijective, in order to check~\eqref{very strongly 1}, it suffices to note that 
\begin{align*}
(S^{2i}\ot \ide)\bigl(y_{(1)}\dpu_{-i} x_{(2)}\ot x_{(1)}\cdot_i y_{(2)}\bigr) &= (S^{2i}\ot \ide)\bigl(y_{(1)}\dpu S^{-2i}(x_{(2)})\ot x_{(1)}\cdot S^{2i}(y_{(2)})\bigr)\\
&=S^{2i}(y_{(1)})\dpu x_{(2)}\ot x_{(1)}\cdot S^{2i}(y_{(2)})\\
&=S^{2i}(y_{(2)})\dpu x_{(1)}\ot x_{(2)}\cdot S^{2i}(y_{(1)})\\
&= (S^{2i}\ot \ide)\bigl(y_{(2)}\dpu S^{-2i}(x_{(1)})\ot x_{(2)}\cdot S^{2i}(y_{(1)})\bigr)\\
& = (S^{2i}\ot \ide)\bigl(y_{(2)}\dpu_{-i} x_{(1)}\ot x_{(2)}\cdot_i y_{(1)}\bigr)
\end{align*}
where the first and last equalities hold by Proposition~\ref{Las Hopf q-brazas son fuertemente regulares}; the second and fourth, by Corollary~\ref{dualidad}; and the third one, by~\eqref{intercambio . :}. Similarly, in order to check~\eqref{very strongly 2}, it suffices to note that 
\begin{align*}
(S^{2i}\ot \ide)\bigl({}^{{\underline{x_{(2)}}}_{-i}}y_{(2)}\ot {x_{(1)}}^{{\underline{y_{(1)}}}_i}\bigr) & = (S^{2i}\ot \ide)\bigl({}^{S^{-2i}(x_{(2)})}y_{(2)}\ot {x_{(1)}}^{S^{2i}(y_{(1)})}\bigr)\\
& = {}^{x_{(2)}}S^{2i}(y_{(2)})\ot {x_{(1)}}^{S^{2i}(y_{(1)})}\\
& = {}^{x_{(1)}}S^{2i}(y_{(1)})\ot {x_{(2)}}^{S^{2i}(y_{(2)})}\\
& = (S^{2i}\ot \ide)\bigl({}^{S^{-2i}(x_{(1)})}y_{(1)}\ot {x_{(2)}}^{S^{2i}(y_{(2)})}\bigr)\\
&=(S^{2i}\ot \ide)\bigl({}^{{\underline{x_{(1)}}}_{-i}}y_{(1)}\ot {x_{(2)}}^{{\underline{y_{(2)}}}_i}\bigr),
\end{align*}
where the first and last equalities hold by Remark~\ref{comentarios'}; the second and fourth, by Corollary~\ref{dualidad}; and the third one, by Remark~\ref{some facts}(3). 
\end{proof}

Let $H$ be a Hopf algebra endowed with binary operations~$\cdot$ and~$\dpu$. We write
\begin{equation}\label{def de times y doubletimes}
h\times k\coloneqq (k\cdot h_{(1)})h_{(2)}\qquad\text{and}\qquad h\mdoubletimes k\coloneqq (k\dpu h_{(2)})h_{(1)}.
\end{equation}
Note that
\begin{equation}\label{unit de times y doubletimes}
1\times k = k\cdot 1,\quad 1\mdoubletimes k = k\dpu 1,\quad h\times 1= (1\cdot h_{(1)})h_{(2)}\quad\text{and}\quad h\mdoubletimes 1=(1\dpu h_{(2)})h_{(1)}.
\end{equation}

\begin{proposition}\label{caracterizacion de Hopf q-brace} Let $\mathcal{H}= (H,\cdot,\dpu)$ be a $q$-magma coalgebra, where $H$ is a Hopf algebra and let $\times$ and~$\mdoubletimes$ be as in~\eqref{def de times y doubletimes}. If $\mathcal{H}$ is a Hopf $q$-brace, then $\times$ and $\mdoubletimes$ are associative with unit~$1$, and the following equalities are satisfied for all $h,k,l\in H$:
\begin{alignat}{3}
&h\cdot kl = (h\cdot l)\cdot k,&&\qquad  (k \mdoubletimes l)\cdot h= (k\cdot h_{(2)}) \mdoubletimes (l\cdot h_{(1)}),&&\qquad h\cdot (k\times l) = h \cdot (l\mdoubletimes k),\label{primera fila}\\[3pt]
&h\dpu kl = (h\dpu l)\dpu k,&&\qquad  (k \times l)\dpu h= (k\dpu h_{(1)}) \times (l\dpu h_{(2)}),&&\qquad h\dpu (k\times l) = h \dpu (l\mdoubletimes k). \label{segunda fila}
\end{alignat}
Conversely, if $\times$ and $\mdoubletimes$ are associative with left unit~$1$, the first and third conditions in~\eqref{primera fila} and the three conditions in~\eqref{segunda fila} are satisfied, then $\mathcal{H}$ is a Hopf $q$-brace.
\end{proposition}

\begin{proof} By the definition of Hopf $q$-brace, the first two equalities in~\eqref{unit de times y doubletimes} and Lemma~\ref{basta menos que regular}(1), we can assume that $\mathcal{H}$ is a regular $q$-magma coalgebra, $1$ is a left unit of $\times$ and $\mdoubletimes$, the equalities in the first column of identities~\eqref{primera fila} and~\eqref{segunda fila} are satisfied, and both $(H,\cdot)$ and $(H,\dpu)$ are right $H^{\op}$-modules. Moreover, by the equalities in the first column of~\eqref{primera fila} and~\eqref{segunda fila}, and the very definitions of $\times$ and $\mdoubletimes$, we have
\begin{align*}
&h\cdot (k\times l) = h\cdot (l\cdot k_{(1)})k_{(2)} = (h\cdot k_{(2)})\cdot (l\cdot k_{(1)}),\\
& h\cdot (l\mdoubletimes k) = h\cdot (k\dpu l_{(2)})l_{(1)} = (h\cdot l_{(1)})\cdot (k\dpu l_{(2)}),\\
&h\dpu (k\times l) = h\dpu (l\cdot k_{(1)})k_{(2)} = (h\dpu k_{(2)})\dpu (l\cdot k_{(1)})
\shortintertext{and}
& h\dpu (l\mdoubletimes k)=  h\dpu (k\dpu l_{(2)})l_{(1)} = (h\dpu l_{(1)})\dpu (k\dpu l_{(2)}).
\end{align*}
Therefore, the equalities in the third column of~\eqref{primera fila} and~\eqref{segunda fila} are equivalent to items~1) and~3) of Definition~\ref{def: q cycle coalgebra}.  Assume now that these conditions hold. We have
\begin{align*}
& (h\times k)\times l= (l\cdot (h\times k)_{(1)})(h\times k)_{(2)} = (l\cdot (k_{(1)}\cdot h_{(2)})h_{(3)})(k_{(2)}\cdot h_{(1)})h_{(4)}
\shortintertext{and}
& h\times (k\times l)= ((k\times l)\cdot h_{(1)})h_{(2)}= (((l\cdot k_{(1)})k_{(2)})\cdot h_{(1)})h_{(2)}.
\end{align*}
Since, by the equality in the first column of~\eqref{primera fila} and item~1) of Definition~\ref{def: q cycle coalgebra},
$$
(l\cdot (k_{(1)}\cdot h_{(2)})h_{(3)})(k_{(2)}\cdot h_{(1)}) = ((l\cdot h_{(3)})\cdot (k_{(1)}\cdot h_{(2)}))(k_{(2)}\cdot h_{(1)}) = ((l\cdot k_{(1)})\cdot (h_{(2)}\dpu k_{(2)}))(k_{(3)}\cdot h_{(1)}),
$$
the operation $\times$ is associative if and only if
$$
((l\cdot k_{(1)})\cdot (h_{(2)}\dpu k_{(2)}))(k_{(3)}\cdot h_{(1)})= ((l\cdot k_{(1)})k_{(2)})\cdot h\quad\text{for all $h,k,l\in H$.}
$$
But, by identity~\eqref{intercambio . :} and the fact that $\mathcal{H}$ is a left regular $q$-magma coalgebra, this happens if and only if the first equality in~\eqref{condicion q-braza} holds. A similar argument proves that $\mdoubletimes$ is associative if and only if the second equality in~\eqref{condicion q-braza} holds. Assume now that these conditions are also satisfied. Then
$$
(k\dpu h_{(1)}) \times (l\dpu h_{(2)}) = ((l\dpu h_{(2)})\cdot (k\dpu h_{(1)})_{(1)})(k\dpu h_{(1)})_{(2)}= ((l\dpu h_{(3)})\cdot (k_{(1)}\dpu h_{(2)}))(k_{(2)}\dpu h_{(1)})
$$
and
$$
(k \times l)\dpu h= ((l\cdot k_{(1)})k_{(2)})\dpu h= ((l\cdot k_{(1)})\dpu (h_{(1)}\cdot k_{(3)}))(k_{(2)}\dpu h_{(2)}) = ((l\cdot k_{(1)})\dpu (h_{(2)}\cdot k_{(2)}))(k_{(3)}\dpu h_{(1)}),
$$
where in the last equality we have used~\eqref{intercambio . :}. So, the second equality in~\eqref{segunda fila} is satisfied if and only if the second identity in Definition~\ref{def: q cycle coalgebra} is fulfilled (use $(k_{(1)}\dpu h_{(2)})S(k_{(2)}\dpu h_{(1)})=\epsilon(k)\epsilon(h)1$). Finally, by the very definition of $\mdoubletimes$ and the first equality in~\eqref{condicion q-braza}
$$
(k \mdoubletimes l)\cdot h= ((l\dpu k_{(2)})k_{(1)})\cdot h= ((l\dpu k_{(3)})\cdot (h_{(1)}\dpu k_{(2)}))(k_{(1)}\cdot h_{(2)})
$$
and
$$
(k\cdot h_{(2)}) \mdoubletimes (l\cdot h_{(1)}) = ((l\cdot h_{(1)})\dpu (k\cdot h_{(2)})_{(2)})(k\cdot h_{(2)})_{(1)}= ((l\cdot h_{(1)})\dpu (k_{(2)}\cdot h_{(2)}))(k_{(1)}\cdot h_{(3)}),
$$
and consequently, the second identity in Definition~\ref{def: q cycle coalgebra} implies the second equality in~\eqref{primera fila}. Finally, by Proposition~\ref{acciones a derecha sobre 1} and the third and fourth equalities in~\eqref{unit de times y doubletimes}, if $\mathcal{H}$ is a Hopf $q$-brace, then $1$ is a unit of $\times$ and $\mdoubletimes$.
\end{proof}

\begin{remark} When $H$ is a group algebra, Proposition~\ref{caracterizacion de Hopf q-brace} yields~\cite{R2}*{Proposition 5}, where $+,\mdoblemas,0$~cor\-respond to $\times, \mdoubletimes,1$, respectively.
\end{remark}

\begin{corollary}\label{relacion .,:,S y S^{-1}} If $\mathcal{H}=(H,\cdot,\dpu)$ is a Hopf $q$-brace with bijective antipode, then
$$
k \hs\cdot\hs \bigl(S(h_{(1)})\hs\cdot\hs S^{-1}(h_{(2)})\hs\bigr)\hs=\hs k\cdot \bigl(S^{-1}(h_{(2)})\hs\dpu\hs S(h_{(1)})\hs\bigr)\quad\text{and}\quad  k\hs\dpu\hs \bigl(S^{-1}(h_{(2)})\hs\dpu\hs S(h_{(1)})\hs\bigr)\hs=\hs k\dpu \bigl(S(h_{(1)})\hs\cdot\hs S^{-1}(h_{(2)})\hs\bigr),
$$
for all $h,k\in H$.
\end{corollary}

\begin{proof} By Remark~\ref{q-braza opuesta}, it suffices to prove the first equality. Using that $H$ is a right $H^{\op}$-module via $\cdot$ and the third identity in~\eqref{primera fila}, we get
\begin{align*}
k\cdot \bigl(S(h_{(1)})\cdot S^{-1}(h_{(2)})\bigr) &= k\cdot \bigl(S(h_{(1)})\cdot S^{-1}(h_{(4)})\bigr)S^{-1}(h_{(3)})h_{(2)}\\
&= (k\cdot h_{(2)}) \cdot \bigl(S^{-1}(h_{(3)})\times S(h_{(1)})\bigr)\\
&= (k\cdot h_{(2)}) \cdot \bigl(S(h_{(1)})\mdoubletimes S^{-1}(h_{(3)})\bigr)\\
&= k \cdot \bigl(S^{-1}(h_{(4)})\dpu S(h_{(1)})\bigr)S(h_{(2)})h_{(3)}\\
&= k\cdot \bigl(S^{-1}(h_{(2)})\dpu S(h_{(1)}\bigr),
\end{align*}
as desired.
\end{proof}

\begin{notation}\label{convolution delta ,times y convolution delta, doubletimes} Given a $k$-vector space $H$, endowed with a coassociative and counitary comultiplication~$\Delta$ and an associative and unitary multiplication $\circ$, we let $\End_{\Delta,\circ}(H)$ denote $\End_k(H)$, endowed with the~con\-volution product associated with $\Delta$ and $\circ$.
\end{notation}

The following lemma is inspired by equality~(6) of~\cite{FS}.

\begin{lemma}\label{version de una formula} If $(H,s)$ is a weak braiding operator, then $k^{S\bigl(S(h_{(1)})^{h_{(2)}}\bigr)} = k^{{}^{S(h_{(1)})}h_{(2)}}$, for all $k,h\in H$.
\end{lemma}

\begin{proof} Since $s$ is a solution of the braid equation and $(H,s_2,s_1)$ is a matched pair (see Theorem~\ref{pares apareados vs q-brazas}), we have
$$
s_2\xcirc (H\ot\mu)\xcirc s_{23} = s_2\xcirc (s_2\ot H)\xcirc s_{23} = (\epsilon\ot\epsilon\ot H)\xcirc s_{23}\xcirc s_{12}\xcirc s_{23} = (\epsilon\ot\epsilon\ot H)\xcirc s_{12}\xcirc s_{23}\xcirc s_{12} = s_2\xcirc (s_2\ot H) = s_2\xcirc (H\ot\mu).
$$
In other words 
$$
k^{{}^{h_{(1)}}l_{(1)}{h_{(2)}}^{l_{(2)}}} = k^{hl}\quad\text{for all $h,k,l\in H$}, 
$$
and so $k^{{}^hl} = k^{h_{(1)}l_{(1)}S({h_{(2)}}^{l_{(2)}})}$. Hence, $k^{S(h^l)} = k^{S(l_{(1)})S(h_{(1)})({}^{h_{(2)}}{l_{(2)}})}$, and consequently,
$$
k^{S(S(h_{(1)})^{h_{(2)}})} = k^{S(h_{(3)})S^2(h_{(2)})({}^{S(h_{(1)})}h_{(4)})} = k^{{}^{S(h_{(1)})}h_{(2)}}, 
$$
as desired.
\end{proof}

\begin{proposition}\label{H es "grupo" via times y mdoubletimes} If $\mathcal{H}=(H,\cdot,\dpu)$ is a Hopf $q$-brace, then $\ide_H$ is invertible in $\End_{\Delta^{\!\cop},\times}(H)$ and its inverse is the map $T_{\times}(h)\coloneqq S(h_{(1)})^{h_{(2)}}$. Moreover, for all $h,k,l\in H$,
\begin{equation}\label{distributividad times cdot y mdoubletimes dpu}
(k\times l)\cdot h= (k\cdot h_{(1)})\times (l\cdot h_{(2)})\quad\text{and}\quad (k\mdoubletimes l)\dpu h= (k\dpu h_{(2)})\mdoubletimes (l\dpu h_{(1)}).
\end{equation}
Finally, if the antipode $S$, of $H$, is bijective, then $\ide_H$ is invertible in $\End_{\Delta,\mdoubletimes}(H)$ and its inverse is the map $T_{\mdoubletimes}(h)\coloneqq S^{-1}(h_{(2)})_{h_{(1)}} = S^{-1}(h_{(2)})\dpu S(h_{(1)})$.
\end{proposition}

\begin{proof} By~\eqref{def de times y doubletimes}, we have
$$
h_{(2)}\times T_{\times}(h_{(1)})  = \bigl(S(h_{(1)})^{h_{(2)}}\cdot h_{(3)}\bigr)h_{(4)}= S(h_{(1)}) h_{(2)}=\epsilon(h)1,
$$
for all $h\in H$, which proves that $T_{\times}$ is right inverse of $\ide_H$ in $\End_{\Delta^{\!\cop},\times}(H)$. On the other hand, we have 
$$
T_{\times}(h_{(2)})\times h_{(1)} = \bigl(h_{(1)}\cdot \bigl(S(h_{(2)})^{h_{(3)}}\bigr)_{(1)}\bigr)\bigl(S(h_{(2)})^{h_{(3)}}\bigr)_{(2)} = {h_{(1)}}^{S\bigl(S(h_{(3)})^{h_{(4)}}\bigr)}S(h_{(2)})^{h_{(5)}}, 
$$
where the second equality holds by Lemma~\ref{algunas cuentas}(5); and we have
$$
\epsilon(h)1 = \bigl(h_{(1)}S(h_{(2)})\bigr)^{h_{(3)}} = \bigl({h_{(1)}}^{{}^{S(h_{(2)})_{(1)}}(h_{(3)})_{(1)}}\bigr){S(h_{(2)})_{(2)}}^{(h_{(3)})_{(2)}} =  \bigl({h_{(1)}}^{{}^{S(h_{(3)})}h_{(4)}}\bigr)S(h_{(2)})^{h_{(5)}},
$$
where the first equality holds by Proposition~\ref{acciones a derecha sobre 1} and Lemma~\ref{algunas cuentas}(3); and the second one, by Theorem~\ref{pares apareados vs q-brazas} and item~4) above Theorem~\ref{pares apareados y estructuras de bialgebra en el producto}. By Lemma~\ref{version de una formula}, this~im\-plies that $T_{\times}$ is left inverse of $\ide_H$ in $\End_{\Delta^{\!\cop},\times}(H)$. In order to check the first equality in~\eqref{distributividad times cdot y mdoubletimes dpu}, we note that, by the first equality in~\eqref{condicion q-braza} and equality~\eqref{intercambio . :},
$$
(k\times l)\cdot h = (l\cdot k_{(1)})k_{(2)}\cdot h = ((l\cdot k_{(1)})\cdot (h_{(1)}\dpu k_{(3)}))(k_{(2)}\cdot h_{(2)})= ((l\cdot k_{(1)})\cdot (h_{(2)} \dpu k_{(2)}))(k_{(3)}\cdot h_{(1)}).
$$
Thus, by item~1) of Definition~\ref{def: q cycle coalgebra}, we have
$$
(k\times l)\cdot h = ((l\cdot h_{(3)})\cdot (k_{(1)}\cdot h_{(2)}))(k_{(2)}\cdot h_{(1)}) =(k\cdot h_{(1)})\times (l\cdot h_{(2)}),
$$
as desired. The second equality in~\eqref{distributividad times cdot y mdoubletimes dpu} follows by a similar argument, and the last assertion holds~by~Lem\-ma~\ref{basta menos que regular} and Re\-mark~\ref{q-braza opuesta}.
\end{proof}

\begin{remark} In the set-theoretic context, Proposition~\ref{H es "grupo" via times y mdoubletimes} together with Corollary~\ref{las q-brazas son no degeneradas}, yield~\cite{R2}*{Proposition~6}.
\end{remark}

\begin{definition}\label{regular "magma" coalgebra} A pair $(C,\cdot)$, consisting of a coalgebra $C$ and a coalgebra morphism $c\ot d\mapsto c\cdot d$, from $C\ot C^{\cop}$ to $C$, is {\em a regular magma coalgebra} if there exists a map $c\ot d\mapsto c^d$, from $C^2$ to $C$, such that the identities in~\eqref{igualdades para cdot} are satisfied.
\end{definition}

\begin{lemma}\label{auxiliar} Let $(H,\cdot)$ be a regular magma coalgebra endowed with a distinguished group-like element~$1$ and maps $\times \colon H^2\to H$ and $T_{\times}\colon H\to H$, such that $\times$ is an associative operation with identity~$1$. Assume that, for all $h,k,l\in H$, the following equalities hold:
	
\begin{enumerate}[itemsep=0.7ex, topsep=1.0ex, label=\emph{\arabic*)}]
		
\item $(k\times l)\cdot h= (k\cdot h_{(1)})\times (l\cdot h_{(2)})$,
		
\item $h\cdot (l_{(2)}\times k^{l_{(1)}})= (h\cdot l)\cdot k$,
		
\item $h_{(2)}\times T_{\times}(h_{(1)})= \epsilon(h)1$.
		
\end{enumerate}
Then, $h\cdot 1=h^1=h$ and $1\cdot h=1^h=\epsilon(h)1$, for all $h\in H$.
\end{lemma}

\begin{proof} Since the map $h\mapsto h\cdot 1$ is bijective (with inverse $h\mapsto h^1$) and $h\cdot 1=(h\times 1)\cdot 1=(h\cdot 1)\times (1\cdot 1)$ for all $h\in H$, we have $1\cdot 1=1$, which implies that $1^1=1$. Thus, $h\cdot 1= h\cdot 1^1= h\cdot (1\times 1^1)=(h\cdot 1)\cdot 1$, and so, $h\cdot 1=h$, for all $h\in H$ (which implies that $h^1=h$, for all $h\in H$). On the other hand,
$$
1\cdot h= (1\times 1)\cdot h= (1\cdot h_{(1)})\times (1\cdot h_{(2)})\quad\text{for all $h\in H$,}
$$
and, consequently,
$$
\epsilon(h)1=\epsilon(1\cdot h)1= (1\cdot h_{(1)}) \times T_{\times}(1\cdot h_{(2)}) = (1\cdot h_{(1)}) \times (1\cdot h_{(2)}) \times T_{\times}(1\cdot h_{(3)})=1\cdot h.
$$
Therefore, we also have $\epsilon(h)1=(1\cdot h_{(1)})^{h_{(2)}}=1^h$, which finishes the proof.
\end{proof}

\begin{lemma}\label{H es un algebra de Hopf} Let $(H,\cdot)$ be a regular magma coalgebra endowed with a distinguished group-like element~$1$ and maps $\times \colon H^2\to H$ and $T_{\times}\colon H\to H$, such that:
	
\begin{enumerate}[itemsep=0.7ex, topsep=1.0ex, label=\emph{\arabic*)}]
		
\item the map $h\ot k\mapsto hk\coloneqq k_{(2)}\times h^{k_{(1)}}$ is a coalgebra morphism from $H^2$ to $H$,
		
\item $\times$ is an associative operation with identity~$1$,
		
\item the map $h\mapsto S(h)\coloneqq T_{\times}(h_{(1)})\cdot h_{(2)}$ is a coalgebra antimorphism of $H$.
		
\end{enumerate}
If, for all $h,k,l\in H$ the following equalities hold:
	
\begin{enumerate}[itemsep=0.7ex, topsep=1.0ex, label=\emph{\arabic*)}, resume]
		
\item $(k\times l)\cdot h= (k\cdot h_{(1)})\times (l\cdot h_{(2)})$,
		
\item $h\cdot kl= (h\cdot l)\cdot k$,
		
\item $T_{\times}(h_{(2)})\times h_{(1)}=  h_{(2)}\times T_{\times}(h_{(1)})= \epsilon(h)1$,
				
\end{enumerate}
then $H$ becomes a Hopf algebra via $\mu(h\ot k)\coloneqq hk$. The unit of $H$ is~$1$ and~the~an\-tipode of $H$ is $S$.
\end{lemma}

\begin{proof} By items~1) and~5), for all $h,k,l\in H$, we have
$$
\bigl(h^{k_{(1)}l_{(1)}}\cdot l_{(2)}\bigr)\cdot k_{(2)}=h^{k_{(1)}l_{(1)}}\cdot k_{(2)}l_{(2)} =h^{(kl)_{(1)}}\cdot (kl)_{(2)} =h\epsilon(l)\epsilon(k),
$$
and so
\begin{equation}\label{eqaux1}
h^{kl}= \bigl(\bigl(\bigl(h^{k_{(1)}l_{(1)}}\cdot l_{(2)}\bigr)\cdot k_{(2)}\bigr)^{k_{(3)}}\bigr)^{l_{(3)}}=(h^k)^l.
\end{equation}
A similar computation, using item~4), yields
\begin{equation}\label{eqaux2}
(k\times l)^h= k^{h_{(2)}}\times l^{h_{(1)}} \quad\text{for all $h,k,l\in H$.}
\end{equation}
Using these facts and item~1), we obtain that
\begin{align*}
&(hk)l=l_{(2)}\times (hk)^{l_{(1)}}= l_{(2)}\times (k_{(2)}\times h^{k_{(1)}})^{l_{(1)}}= l_{(3)}\times \bigl({k_{(2)}}^{l_{(2)}}\times (h^{k_{(1)}})^{l_{(1)}}\bigr)
\shortintertext{and}
&h(kl) = (kl)_{(2)}\times h^{(kl)_{(1)}} = k_{(2)}l_{(2)}\times h^{k_{(1)}l_{(1)}}= \bigl(l_{(3)}\times {k_{(2)}}^{l_{(2)}}\bigr)\times (h^{k_{(1)}})^{l_{(1)}}.
\end{align*}
Thus, from item~2) we conclude that $\mu$ is associative. Using now that $1$ is group-like, item~2) and~Lem\-ma~\ref{auxiliar}, we obtain that
$$
h1=1\times h^1=h^1=h\quad\text{and}\quad 1h=h_{(2)}\times 1^{h_{(1)}}=h\times 1=h,
$$
which proves that $1$ is the unit of $\mu$. Since $1$ is group-like and $\mu$ is a coalgebra morphism, we conclude~that~$H$ is a bialgebras. Moreover, by item~6), we have
\begin{equation}\label{convolucion S*ide}
S(h_{(1)})h_{(2)}=h_{(3)}\times S(h_{(1)})^{h_{(2)}}= h_{(4)}\times \bigl(T_{\times}(h_{(1)})\cdot h_{(2)}\bigr)^{h_{(3)}}= h_{(2)}\times T_{\times}(h_{(1)})=\epsilon(h)1.
\end{equation}
It remains to prove that $h_{(1)}S(h_{(2)})=\epsilon(h)1$, for all $h\in H$. Since, by the previous lemma,  $1\cdot h = \epsilon(h)1$, for this it suffices to check $(h_{(1)}S(h_{(2)}))^{h_{(3)}} = \epsilon(h)1$. But,
$$
(h_{(1)}S(h_{(2)}))^{h_{(3)}} = \bigl(S(h_{(2)})\times {h_{(1)}}^{S(h_{(3)})}\bigr)^{h_{(4)}} = S(h_{(2)})^{h_{(3)}}\times h_{(1)} = T_{\times}(h_{(2)})\times h_{(1)} = \epsilon(h)1,
$$
where the first equality holds by item~3); the second one, by~\eqref{eqaux1}, \eqref{eqaux2} and~\eqref{convolucion S*ide}; the third one, by the definition of $S$ in item~3); and the last one, by item~6).
\end{proof}

\begin{theorem}\label{otra caracterizacion de Hopf q-brace} Let $\mathcal{H}=(H,\cdot,\dpu)$ be a left regular $q$-magma coalgebra endowed with a distinguished group-like element~$1$ and maps $\times,\mdoubletimes \colon H^2\to H$ and $T_{\times}\colon H\to H$. Set $hk \coloneqq k_{(2)}\times h^{k_{(1)}}$ and $S(h)\coloneqq T_{\times}(h_{(1)})\cdot h_{(2)}$. Assume that

\begin{enumerate}[itemsep=0.7ex, topsep=1.0ex, label=\emph{\arabic*)}]
		
\item The map $h\ot k\mapsto hk$ is a coalgebra morphism from $H^2$ to $H$,
		
\item $\times$ and $\mdoubletimes$ are associative operations with identity~$1$,
		
\item $S$ is a coalgebra antimorphism of $H$.
		
\end{enumerate}
If for all $h,k,l\in H$ the following equalities hold:

\begin{enumerate}[itemsep=0.7ex, topsep=1.0ex, label=\emph{\arabic*)}, resume]
		
\item $(k\times l)\cdot h= (k\cdot h_{(1)})\times (l\cdot h_{(2)})$,
		
\item $(k\times l)\dpu h= (k\dpu h_{(1)})\times (l\dpu h_{(2)})$,
		
\item $h\cdot kl = (h\cdot l)\cdot k$ and $h\dpu kl = (h\dpu l)\dpu k$,
		
\item $T_{\times}(h_{(2)})\times h_{(1)} = h_{(2)}\times T_{\times}(h_{(1)})=\epsilon(h)1$,
		
\item $h\cdot (k\times l) = h \cdot (l\mdoubletimes k)$ and $h\dpu (k\times l) = h \dpu (l\mdoubletimes k)$,
		
\item $h\mdoubletimes k= h_{(2)} \times (k\dpu h_{(3)})^{h_{(1)}}$,
\end{enumerate}
then $H$ is a Hopf algebra via $\mu(h\ot k)\coloneqq hk$ (with unit~$1$ and antipode $S$) and~$\mathcal{H}$ is a Hopf $q$-brace.~Converse\-ly, if $\mathcal{H}$ is a Hopf $q$-brace, then $\mathcal{H}$ is a regular~$q$-mag\-ma coalgebra with distinguished group-like~ele\-ment~$1$, the maps~$\times$, $\mdoubletimes$ and~$T_{\times}$, defined in~\eqref{def de times y doubletimes}~and~Propo\-si\-tion~\ref{H es "grupo" via times y mdoubletimes}, satisfy equalities 
\begin{equation}\label{ecua3}
k_{(2)}\times h^{k_{(1)}}=hk\quad\text{and}\quad T_{\times}(h_{(1)})\cdot h_{(2)}=S(h), 
\end{equation}
conditions~1)--9) are fulfilled, and 
$$
(k\mdoubletimes l)\cdot h= (k\cdot h_{(2)})\mdoubletimes (l\cdot h_{(1)})\quad\text{and}\quad (k\mdoubletimes l)\dpu h= (k\dpu h_{(2)})\mdoubletimes (l\dpu h_{(1)}),\qquad\text{for al $h,k,l\in H$.}
$$
\end{theorem}

\begin{proof} $\Rightarrow$)\enspace By Lemma~\ref{H es un algebra de Hopf}, we know that $H$ is a Hopf algebra via the multiplication map $\mu(h\ot k)\coloneqq hk$, with the unit and antipode as in the statement. Moreover, we have
$$
(k\cdot h_{(1)})h_{(2)}= h_{(3)}\times (k\cdot h_{(1)})^{h_{(2)}}= h\times k\quad\text{and}\quad (k\dpu h_{(2)})h_{(1)}= h_{(2)}\times (k\dpu h_{(3)})^{h_{(1)}} = h\mdoubletimes k.
$$
Hence, we can apply Proposition~\ref{caracterizacion de Hopf q-brace} to obtain that $\mathcal{H}$ is a Hopf $q$-brace.
	
\smallskip
	
\noindent $\Leftarrow$)\enspace Since $\mathcal{H}$ is a Hopf $q$-brace, from Remark~\ref{hopf q brace es regular}, it follows that $\mathcal{H}$ is a regular $q$-cycle coalgebra and~$1$ is a group-like element. Moreover items~1) and~3) are trivial; items~2), 5), 6), 8) and the fact that $(k\mdoubletimes l)\cdot h= (k\cdot h_{(2)})\mdoubletimes (l\cdot h_{(1)})$, for al $h,k,l\in H$, hold by Proposition~\ref{caracterizacion de Hopf q-brace}; items~4), 7) and the fact that $(k\mdoubletimes l)\dpu h= (k\dpu h_{(2)})\mdoubletimes (l\dpu h_{(1)})$, for al $h,k,l\in H$, hold by Proposition~\ref{H es "grupo" via times y mdoubletimes}; and item~9) holds since, by equalities~\eqref{def de times y doubletimes},
\begin{equation*}
h_{(2)} \times (k\dpu h_{(3)})^{h_{(1)}} = ((k\dpu h_{(4)})^{h_{(1)}} \cdot h_{(2)} )h_{(3)} = (k\dpu h_{(2)}) h_{(1)} = h\mdoubletimes k.
\end{equation*}
Furthermore,
\begin{equation*}
k_{(2)}\times h^{k_{(1)}} = (h^{k_{(1)}}\cdot k_{(2)})k_{(3)} = hk\quad\text{and}\quad T_{\times}(h_{(1)})\cdot h_{(2)} = S(h_{(1)})^{h_{(2)}}\cdot h_{(3)} = S(h),
\end{equation*}
as desired.
\end{proof}

\begin{remark} From the first identity in~\eqref{ecua3} it follows that $k^h=T_{\times}(h_{(2)})\times kh_{(1)}$, for all $h,k\in H$.
\end{remark}

\begin{remark}\label{para GV 1} Let $\mathcal{H}=(H,\cdot,\dpu)$ be a Hopf $q$-brace. By the first equality in~\eqref{def de times y doubletimes}, we know that
$$
(k_{(1)}\times h)S(k_{(2)}) =h\cdot k\quad\text{for all $h,k\in H$.}
$$
Hence, the map $h\ot k \mapsto (k_{(1)}\times h)S(k_{(2)})$ is a coalgebra morphism from $H\ot H^{\cop}$ to $H$. Assume now that $S$ is bijective. The previous argument applied to $\mathcal{H}^{\op}=(H^{\cop},\dpu,\cdot)$ shows that 
$$
(k_{(2)}\mdoubletimes h)S^{-1}(k_{(1)}) =h\dpu k\quad\text{for all $h,k\in H$,}
$$
and the map $h\ot k \mapsto (k_{(2)}\mdoubletimes h)S^{-1}(k_{(1)})$ is a coalgebra morphism from $H\ot H^{\cop}$ to $H$.
\end{remark}

\begin{proposition}\label{para GV 2} If $\mathcal{H}=(H,\cdot,\dpu)$ is a Hopf $q$-brace, then 
\begin{equation}\label{formula para times}
(k\times l)h=k h_{(3)}\times T_{\times}(h_{(2)})\times lh_{(1)}\quad\text{for all $h,k,l\in H$.}
\end{equation}
Moreover, if $S$ is bijective, then
\begin{equation}\label{formula para mdoubletimes}
(k\mdoubletimes l)h=k h_{(1)}\mdoubletimes T_{\mdoubletimes }(h_{(2)})\mdoubletimes lh_{(3)}\quad\text{for all $h,k,l\in H$.}
\end{equation}
\end{proposition}

\begin{proof} We have
\begin{align*}
(k\times l)h &= (l^{S(k_{(1)})})k_{(2)}h &&\text{by the first equality in~\eqref{def de times y doubletimes}}\\
&= (l^{h_{(1)}S(h_{(2)})S(k_{(1)})})k_{(2)}h_{(3)}\\
&= ((l^{h_{(1)}})^{S(k_{(1)}h_{(2)})})k_{(2)}h_{(3)}\\
&= kh_{(2)}\times l^{h_{(1)}} &&\text{by the first equality in~\eqref{def de times y doubletimes}}\\
&= kh_{(3)}\times T_{\times}(h_{(2)})\times lh_{(1)} &&\text{by the first equality in~\eqref{ecua3}},
\end{align*}
which proves equality~\eqref{formula para times}. Applying this to $\mathcal{H}^{\op}=(H^{\cop},\dpu,\cdot)$, we obtain~\eqref{formula para mdoubletimes}.
\end{proof}

\begin{definition}\label{def de circ^n, etc} For a Hopf $q$-brace $\mathcal{H}=(H,\cdot,\dpu)$, we recursively define binary operations~$\bullet^n$ on $H$, by
\begin{equation}\label{def de circ n y + n}
h\bullet^n k\coloneqq \begin{cases} h k &\text{if $n=0$,}\\ {}^{h_{(1)}}k_{(1)}\bullet^{n-1} {h_{(2)}}^{k_{(2)}} &\text{if $n>0$.} \end{cases}
\end{equation}
Then, we define~$\times^n$ and~$\mdoubletimes^n$, by $h\times^n k\coloneqq (k\cdot h_{(1)})\bullet^n h_{(2)}$ and $h\mdoubletimes^n k \coloneqq (k\dpu h_{(2)})\bullet^n h_{(1)}$. Note that
\begin{align*}
&h\times^0 k=h\times k,\quad h\mdoubletimes^0 k=h\mdoubletimes k\qquad\text{and}\\
&k\times^1 h = (h\cdot k_{(1)})\bullet^1 k_{(2)} =(k\dpu h_{(2)})h_{(1)}= h\mdoubletimes k,
\end{align*}
where the second equality holds by Remark~\ref{relacion entre s, G y dpu}.
\end{definition}

\begin{remark} In the set-theoretic context, the operations $\bullet^n$, $\times^n$ and $\mdoubletimes^n$ correspond to the operations $\circ_n$, $+_n$ and $\mdoblemas_n$, introduced above Theorem~3 of~\cite{R2}.
\end{remark}

\begin{theorem}\label{Hopf q-brace Cal{H}^n} Let $\mathcal{H}=(H,\cdot,\dpu)$ be a Hopf $q$-brace with bijective antipode, let $s$ be the invertible non-degenerate set-theoretic type solution of the braid equation associated with $\mathcal{H}$, let $H^{\underline{n}}$ be the underlying coalgebra of~$H$, endowed with the multiplication map $\bullet^n$, and let $H^{\overline{n}}\coloneqq H^{\underline{n}\, \cop}$. The triples $\mathcal{H}^{\underline{n}}\coloneqq (H^{\underline{n}},\cdot, \dpu)$ and $\mathcal{H}^{\overline{n}}\coloneqq (H^{\overline{n}},\cdot_{\tilde{s}}, \dpu_{\tilde{s}})$ are Hopf $q$-braces with bijective antipodes.
\end{theorem}

\begin{proof} By Theorem~\ref{pares apareados vs q-brazas}, in order to check that $\mathcal{H}^{\underline{n}}$ and $\mathcal{H}^{\overline{n}}$ are Hopf $q$-braces with bijective antipode, we must prove that $H^{\underline{n}}$ and $H^{\overline{n}}$ are Hopf algebras (which automatically implies that their antipodes are bijective because they are inverses one of each other) and that $(H^{\underline{n}},s)$ and $(H^{\overline{n}},\tilde{s})$ are weak braiding operators. Set $\mu^n(h\ot k)\coloneqq h \bullet^n k$. Since~$s$ is a coalgebra automorphism of $H^2$ and $\mu^{n+1}=\mu^n\xcirc s$, for all~$n$, from the fact that $\mu^0\colon H^2\to H$ is a coalgebra morphism it follows that $\mu^n\colon H^2\to H$ is a coalgebra morphism, for all~$n$.

By Theorem~\ref{pares apareados vs q-brazas} we know that $\mu^0$ satisfies identities~\eqref{eq:bo1} and~\eqref{eq:bo2}. This implies that $\mu^n$ satisfies the same identities for all~$n$. In fact, set $s_{12}\coloneqq s\ot H$ and $s_{23}\coloneqq H\ot s$. Since $s$ is an invertible solution of the braid equation, we have
\begin{align*}
& s\xcirc (\mu^n\ot H) = (H \ot \mu^n)\xcirc s_{12}\xcirc s_{23} \Leftrightarrow s\xcirc (\mu^{n+1} \ot H) = (H \ot \mu^{n+1})\xcirc s_{12}\xcirc s_{23}
\shortintertext{and}
& s\xcirc (H \ot \mu^n) = (\mu^n\ot H)\xcirc s_{23}\xcirc s_{12} \Leftrightarrow s\xcirc (H \ot \mu^{n+1}) = (\mu^{n+1}\ot H)\xcirc s_{23}\xcirc s_{12},
\end{align*}
which implies that $\mu^n$ satisfies identities~\eqref{eq:bo1} and~\eqref{eq:bo2}, for all $n$.  Using this we obtain that
\begin{align*}
&\mu^{n+1}\xcirc (\mu^{n+1}\ot H) = \mu^n\xcirc s\xcirc (\mu^n\xcirc s\ot H)= \mu^n\xcirc s\xcirc (\mu^n \ot H)s_{12}=\mu^n\xcirc (H\ot \mu^n)\xcirc s_{12}\xcirc s_{23}\xcirc s_{12}
\shortintertext{and}
&\mu^{n+1}\xcirc (H \ot \mu^{n+1}) = \mu^n\xcirc s\xcirc (H \ot \mu^n\xcirc s) = \mu^n\xcirc s\xcirc (H \ot \mu^n)\xcirc s_{23}= \mu^n\xcirc (\mu^n\ot H) \xcirc s_{23}\xcirc s_{12}\xcirc s_{23},
\end{align*}
which proves that~$\mu^n$ is associative for all $n$, since $\mu^0$ is associative and $s$ is a solution of the braid equation. Moreover, by Theorem~\ref{pares apareados vs q-brazas} and condition~5) above Theorem~\ref{pares apareados y estructuras de bialgebra en el producto}, we have
$$
1\bullet^{n+1} h={}^1h_{(1)}\bullet^n 1^{h_{(2)}} = h\bullet^n 1\quad\text{and}\quad h\bullet^{n+1} 1={}^{h_{(1)}}1 \bullet^n {h_{(2)}}^1 = 1\bullet^n h\quad\text{for all $h\in H$,}
$$
and so, $1$ is the unit of $\mu_n$, for all $n$. Hence, the $H^{\underline{n}}$'s are bialgebras and identities~\eqref{eq:bo3} and~\eqref{eq:bo4} are fulfilled, for all the~$H^{\underline{n}}$'s. Consequently, the~$H^{\overline{n}}$'s are also bialgebras satisfying conditions~\eqref{eq:bo1}--\eqref{eq:bo4}. In order to finish the proof, we need to check that~$H^{\underline{n}}$ and $H^{\overline{n}}$ are Hopf algebras, for all $n\in \mathds{N}_0$. Since $H$ and $H^{\cop}$ are Hopf algebras, this is true for~$n=0$. Assume that it is true for a fixed~$n\ge 0$. We are going to check that~$H^{\underline{n+1}}$ is a Hopf algebra by proving that $(\mu^{n+1}\ot H)\xcirc (H\ot \Delta)$ is bijective (which happens if and only if $H^{\underline{n+1}}$ have antipode). Write $L^{n+1}\coloneqq (\mu^{n+1}\ot H)\xcirc (H\ot \Delta)\xcirc \overline{G}_{\mathcal{H}}$, where $\overline{G}_{\mathcal{H}}$ is the map introduced at the beginning of Subsection~\ref{Left regular coalgebras}. Since $\overline{G}_{\mathcal{H}}$ is invertible (see Definition~\ref{q-magma coalgebra no degenerada a izquierda}), in order to fulfill our task we only must show that $L^{n+1}$ is bijective. But, by Remark~\ref{relacion entre s, G y dpu} and the very definitions of~$\mu^{n+1}$, $\overline{G}_{\mathcal{H}}$ and~$\mdoubletimes^n$, we have
$$
L^{n+1}(h\ot k)=(h\cdot k_{(1)})\bullet^{n+1} k_{(2)}\ot k_{(3)} = (k_{(1)}\dpu h_{(2)})\bullet^n h_{(1)}\ot k_{(2)}=h\mdoubletimes^{n} k_{(1)}\ot k_{(2)},
$$
and so, by Proposition~\ref{H es "grupo" via times y mdoubletimes}, the map $L^{n+1}$ is invertible with inverse $h\ot k \mapsto h\mdoubletimes^n T_{\mdoubletimes^n} (k_{(1)})\ot k_{(2)}$. The fact that $H^{\overline{n+1}}$ is a Hopf algebra follows applying the same argument to $\mathcal{H}^{\overline{n}} = (H^{\overline{n}},\cdot_{\tilde{s}}, \dpu_{\tilde{s}})$.
\end{proof}

\begin{remark}\label{(Cal{H}^underline{n} cuando n es negativo} By Remark~\ref{q-braza opuesta}, if $\mathcal{H}$ is a Hopf $q$-brace with bijective antipode, then $\mathcal{H}^{\op}$ is also. Applying Theorem~\ref{Hopf q-brace Cal{H}^n} to $\mathcal{H}^{\op}$, we obtain other families ${\mathcal{H}^{\op\, \underline{n}}}$ ($n\in \mathds{N}_0$) and ${\mathcal{H}^{\op\, \overline{n}}}$ ($n\in \mathds{N}_0$), of Hopf $q$-braces.
\end{remark}

\section{Hopf Skew-braces}\label{seccion: Hopf Skew-braces}

In this section, motivated by the remark after Corollary~2 of~\cite{R2}*{Proposition~4}, we define Hopf skew-braces as a particular type of Hopf $q$-braces. Then, we prove that a Hopf $q$-brace is a Hopf skew-brace if and only if its associated weak braiding operator is a braiding operator. Moreover, we also prove that the notion of Hopf skew-brace with bijective antipode has the following equivalent avatars: a direct generalization of the concept of skew-brace as formulated originally by Guarnieri and Vendramin and a generalization of the concept of linear $q$-cycle set formulated by Rump. Furthermore, we introduce the concept of invertible $1$-cocycle and we prove that the category of Hopf skew-braces with bijective antipode is equivalent to the category of invertible $1$-cocycles. Finally, in Subsection~\ref{italianos}, we show that the category of Hopf skew braces is isomorphic to the category of Yetter-Drinfeld braces, recently introduced in \cite{FS}.

\begin{definition}\label{def: Hopf skew-brace} A Hopf $q$-brace $\mathcal{H}=(H,\cdot,\dpu)$ is a {\em Hopf skew-brace} if
\begin{equation}\label{compatibilidad cdot dpu en skew brazas}
h\mdoubletimes k=k\times h\quad\text{for all $h,k\in H$,}
\end{equation}
that is, $(k\dpu h_{(2)})h_{(1)}=(h\cdot k_{(1)})k_{(2)}$.
\end{definition}

\begin{examples} 1)\enspace Hopf skew-braces whose underlying Hopf algebra structure are group algebras can be naturally identified with skew-braces (\cite{GV}). This is clearer if we think in the GV-Hopf skew-brace avatar of Hopf skew-braces (Definition~\ref{vendramin skew braza}).

\smallskip
	
\noindent 2)\enspace In Example~\ref{estructurasobre la algebras de Taft} we point out that the Sweedler algebra has a Hopf skew-brace structure.
	
\smallskip
	
\noindent 3)\enspace  The dihedral group $D_{2m}$ is generated by elements $x$ and $y$ subject to the relations $x^{2m}=y^2=yxyx=1$. It is well known that the underlying set of $D_{2m}$ is $\{1,x,\dots,x^{2m-1},y,xy,\dots,x^{2m-1}y\}$ and that its center is $\{1,x^m\}$. Let $k$ be a field of characteristic different of~$2$ and let $H$ be the Hopf $k$-algebra dual of $k[D_{2m}]$. Thus, $H$ is the $k$-vector space with basis $\{d_a:a\in D_{2m}\}$, endowed with the product and the coproduct given by
$$
d_ad_b = \delta_a^b d_a\qquad\text{and}\qquad \Delta(d_a) = \sum_{b\in D_{2m}} d_{ab}\ot d_{b^{-1}},
$$
where $\delta_a^b$ is the Kronecker delta. The unit, counit and antipode of $H$ are $1_H= \sum_{a\in D_{2m}}d_a$, $\epsilon(d_{a}) =\delta_a^1$ and~$S(d_a) = d_{a^{-1}}$. There are exactly four skew-braces structures on $H$ such that $d_a\cdot d_b = d_a\dpu d_b = 0$ for all $b\notin \{1,x^m\}$. In all of them $\dpu$ coincides with $\cdot$ and $d_a\cdot d_1 + d_a\cdot d_{x^m} = d_a$ for all $a\in D_{2m}$. They are:
\begin{align*}
& d_{x^iy^j}\cdot d_1 = d_{x^iy^j}. && \text{case~1}\\
& d_{x^iy^j}\cdot d_1 = \begin{cases} \frac{1}{2} d_{x^iy^j} + \frac{1}{2} d_{x^{i+m}y^j} & \text{if $j =1$,} \\ d_{x^iy^j} & \text{if $j=0$.} \end{cases} && \text{case~2}\\
& d_{x^iy^j}\cdot d_1 = \begin{cases} \frac{1}{2} d_{x^iy^j} + \frac{1}{2} d_{x^{i+m}y^j} & \text{if $i$ is odd,} \\ d_{x^iy^j} & \text{if $i$ is even.} \end{cases} && \text{case~3}\\
& d_{x^iy^j}\cdot d_1 = \begin{cases} \frac{1}{2} d_{x^iy^j} + \frac{1}{2} d_{x^{i+m}y^j} & \text{if $i+j$ is odd,} \\ d_{x^iy^j} & \text{if $i+j$ is even.} \end{cases} && \text{case~4}
\end{align*}
\end{examples}

\begin{remark}\label{h.h = h:h si S^2=ide} For Hopf skew-braces with bijective antipode there is an improvement of Corollary~\ref{relacion .,:,S y S^{-1}}. In fact, in this case the monoids $\End_{\Delta^{\!\cop},\times}(H)$ and $\End_{\Delta,\mdoubletimes}(H)$, introduced in
Proposition~\ref{H es "grupo" via times y mdoubletimes}, are opposite monoids, and so $T_{\times}=T_{\mdoubletimes}$. In other words,
$$
S(h_{(1)})\cdot S^{-1}(h_{(2)})=S^{-1}(h_{(2)})\dpu S(h_{(1)})\quad\text{for all $h\in H$.}
$$
\end{remark}

Recall that having a Hopf $q$-brace $\mathcal{H} = (H,\cdot,\dpu)$ is equivalent to having a weak braiding operator~$(H,s)$.

\begin{proposition}\label{motivacion para Hopf skew-brace} A weak braiding operator $(H,s)$ is a braiding operator if and only if its associated Hopf $q$-brace is a Hopf skew-brace.
\end{proposition}

\begin{proof} By Remark~\ref{relacion entre s, G y dpu}, we know that $(H,s)$ is a braiding operator if and only if
$$
(k\dpu h_{(2)})h_{(1)} = \bigl(h\cdot k_{(1)}\bigr)k_{(2)}\quad\text{for all $h,k\in H$.}
$$
In other words if and only if condition~\eqref{compatibilidad cdot dpu en skew brazas} is satisfied.
\end{proof}

\begin{remark}\label{equivalencia cat braid oper...} By Proposition~\ref{motivacion para Hopf skew-brace}, the categories of braiding operators and Hopf skew-braces are isomorphic.   
\end{remark}

\begin{proposition}\label{skew-braces con S biyectiva} Let $H$ be a Hopf algebra and let $\mathcal{H} = (H,\cdot,\dpu)$ be a left regular $q$-magma coal\-gebra.~Then~$\mathcal{H}$ is a Hopf skew-brace if and only if $(H,\cdot)$ and $(H,\dpu)$ are right $H^{\op}$-modules and condition~\eqref{compatibilidad cdot dpu en skew brazas} is fulfilled.
\end{proposition}

\begin{proof} By definition, if $\mathcal{H}$ is a Hopf skew-brace, then $(H,\cdot)$ and $(H,\dpu)$ are right $H^{\op}$-modules and condition~\eqref{compatibilidad cdot dpu en skew brazas} is satisfied. Conversely, assume that these facts hold, and let $s\colon H^2\to H^2$ be the left non-degenerate coalgebra endomorphisms of $H^2$, associated with $\mathcal{H}$ according to Remarks~\ref{notacion s1 y s2} and~\ref{correspondencia endomorfismos de X^2 no degenerados a izquierday q magma coalgebras regulares a izquierda}. Since~$\mathcal{H}$ is left regular, from~Remark~\ref{relacion entre s, G y dpu} and condition~\eqref{compatibilidad cdot dpu en skew brazas}, it follows that condition~\eqref{eq:bo5} is satisfied. Hence, by Theorem~\ref{Condicion suficiente para braided operator} and Proposition~\ref{motivacion para Hopf skew-brace}, in order to finish the proof it suffices to show that $H$ is a left $H$-mo\-dule via $h\ot l\mapsto {}^hl$ and a right $H$-module via $h\ot l\mapsto h^l$. By Lemma~\ref{algunas cuentas}(1), the second fact holds. We~next prove the first one. Let $h,k,l\in H$ arbitrary. By Remark~\ref{notacion s1 y s2} and the fact that $(H,\dpu)$ is a right $H^{\op}$-module, we have
\begin{align*}
&{}^{hk}l = l_{(2)}\dpu (hk)^{l_{(1)}}\\
\shortintertext{and}
& {}^h(^kl) = {}^h\bigl(l_{(2)}\dpu k^{l_{(1)}}\bigr) = \bigl(l_{(4)}\dpu {k_{(1)}}^{l_{(1)}}\bigr)\dpu h^{l_{(3)}\dpu {k_{(2)}}^{l_{(2)}}} = l_{(4)}\dpu \bigl(h^{l_{(3)}\dpu {k_{(2)}}^{l_{(2)}}}\bigr){k_{(1)}}^{l_{(1)}} = l_{(3)}\dpu h^{^{k_{(2)}} l_{(2)}}{k_{(1)}}^{l_{(1)}}.
\end{align*}
Hence, by equality~\eqref{eq para nenes} and Lemma~\ref{lema para 3 implica 2}, we are reduced to check that the first identity in~\eqref{condicion q-braza} holds. But, by condition~\eqref{compatibilidad cdot dpu en skew brazas}, the fact that $(H,\dpu)$ is a right $H^{\op}$-module and Remark~\ref{d y p son morfismos de coalgears}, we have
$$
(hk\cdot l_{(1)})l_{(2)} = \bigl((l\dpu k_{(2)})\dpu h_{(2)}\bigr)h_{(1)}k_{(1)} = (h\cdot (l_{(1)}\dpu k_{(3)}))(l_{(2)}\dpu k_{(2)})k_{(1)} = (h\cdot (l_{(1)}\dpu k_{(2)}))(k_{(1)}\cdot l_{(2)})l_{(3)},
$$ 
from which, the first identity in~\eqref{condicion q-braza} follows immediately.
\end{proof}

\begin{remark} In the set-theoretic context, Proposition~\ref{skew-braces con S biyectiva} corresponds to Corollary 4 of~\cite{R2}*{Proposition~6}.
\end{remark}

\begin{proposition}\label{2 de 3} Let $H$ be a Hopf algebra with bijective antipode and let $p,d\colon H\ot H^{\cop}\to H$ be two maps. Write $h\cdot k \coloneqq p(h\ot k)$ and $h\dpu k\coloneqq d(h\ot k)$. Assume that condition~\eqref{compatibilidad cdot dpu en skew brazas} is fulfilled. Then, we have:
	
\begin{enumerate}[itemsep=0.7ex, topsep=1.0ex, label=\emph{\arabic*)}]
		
\item If $p$ is a coalgebra map, then $d$ is a coalgebra map if and only if condition~\eqref{intercambio . :} is satisfied.
		
\item If $d$ is a coalgebra map, then $p$ is a coalgebra map if and only if condition~\eqref{intercambio . :} is satisfied.
		
\end{enumerate}
	
\end{proposition}

\begin{proof} Assume that $p$ is a coalgebra map. By identity~\eqref{compatibilidad cdot dpu en skew brazas}, we have $h\dpu k = (k_{(2)}\cdot h_{(1)})h_{(2)}S^{-1}(k_{(1)})$. Thus, 
$$
\epsilon(h\dpu k) = \epsilon(h)\epsilon(k)\quad\text{and}\quad (h\dpu k)_{(1)}\ot (h\dpu k)_{(2)} = h_{(1)}\dpu k_{(2)}\ot h_{(2)}\dpu k_{(1)}
$$ 
if and only if
$$
(k_{(3)}\cdot h_{(2)})h_{(3)}S^{-1}(k_{(2)}) \ot (k_{(4)}\cdot h_{(1)})h_{(4)}S^{-1}(k_{(1)})\! =\! (k_{(4)}\cdot h_{(1)})h_{(2)}S^{-1}(k_{(3)}) \ot (k_{(2)}\cdot h_{(3)})h_{(4)}S^{-1}(k_{(1)}).
$$
But the last equality is fulfilled if and only if
$$
(k_{(2)}\cdot h_{(2)})h_{(3)}S^{-1}(k_{(1)})\ot k_{(3)}\cdot h_{(1)} = (k_{(3)}\cdot h_{(1)})h_{(2)}S^{-1}(k_{(2)})\ot k_{(1)}\cdot h_{(3)},
$$
or, in other words, $h_{(2)}\dpu k_{(1)} \ot k_{(2)}\cdot h_{(1)} = h_{(1)}\dpu k_{(2)}\ot k_{(1)}\cdot h_{(2)}$. This ends the proof of item~1). The proof of item~2) is similar.
\end{proof}

\begin{remark}\label{es inecesario S biyec} For item~2), it is not necessary for $S$ to be bijective.
\end{remark}

\begin{proposition}\label{caracterizacion de skew brazas} Let $\mathcal{H}=(H,\cdot,\dpu)$ be a Hopf $q$-brace and let $H\bowtie H$ be the bicrossed product associated with $\mathcal{H}$ according to Theorems~\ref{pares apareados y estructuras de bialgebra en el producto} and~\ref{pares apareados vs q-brazas}. Let $F\colon H\to H\bowtie H$ be the map given by $F(h)\coloneqq S(h_{(1)})\ot h_{(2)}$. If $S$ is injective, then $\mathcal{H}$ is a Hopf skew-brace if and only if $F(H)$ is a subalgebra of $H\bowtie H$.
\end{proposition}

\begin{proof} By Proposition~\ref{acciones sobre S^j(h)}, we have
$$
S(k_{(1)})\dpu (h\cdot k_{(2)}) = S\bigl(k_{(1)}\dpu \bigl((h\cdot k_{(3)})\cdot S(k_{(2)})\bigr)\bigr)  = S\bigl(k_{(1)}\dpu \bigl(h\cdot S(k_{(2)})k_{(3)}\bigr)\bigr) = S(k\dpu h),
$$
for all $h,k\in H$. Using this,  the very definition of the multiplication in $H\bowtie H$, the last assertion~in~Re\-mark~\ref{Caracterizacion de no degenerado a izquierda}, Lemma~\ref{basta menos que regular} and identity~\eqref{intercambio . :}, we obtain that
\begin{align*}
(S(h_{(1)})\ot h_{(2)})(S(k_{(1)})\ot k_{(2)}) &= S(h_{(1)})\bigl({}^{h_{(2)}}S(k_{(2)})\bigr)\ot \bigl({h_{(3)}}^{S(k_{(1)})}\bigr) k_{(3)}\\
&= S(h_{(1)})\bigl(S(k_{(2)})\dpu {h_{(2)}}^{S(k_{(3)})}\bigr)\ot \bigl({h_{(3)}}^{S(k_{(1)})}\bigr)k_{(4)}\\
&= S(h_{(1)})\bigl(S(k_{(2)})\dpu (h_{(2)}\cdot k_{(3)})\bigr)\ot (h_{(3)}\cdot k_{(1)})k_{(4)}\\
&= S(h_{(1)})S(k_{(2)}\dpu h_{(2)})\ot (h_{(3)}\cdot k_{(1)})k_{(3)}\\
&= S((k_{(1)}\dpu h_{(3)})h_{(1)})\ot (h_{(2)}\cdot k_{(2)})k_{(3)}.
\end{align*}
If identity~\eqref{compatibilidad cdot dpu en skew brazas} is fulfilled, then
$$
F(h)F(k) = S((k_{(1)}\dpu h_{(4)})h_{(1)})\ot (k_{(2)}\dpu h_{(3)})h_{(2)} = F((k\dpu h_{(2)})h_{(1)}) = F(h\mdoubletimes k),
$$
and consequently $F(H)$ is a subalgebra of $H\bowtie H$. Conversely, if $F(H)$ is a subalgebra of $H\bowtie H$, then, for each $h,k\in H$ there exists $l\in H$ such that $S((k_{(1)}\dpu h_{(3)})h_{(1)})\ot (h_{(2)}\cdot k_{(2)})k_{(3)} = S(l_{(1)})\ot l_{(2)}$. So, on one hand, $(h\cdot k_{(1)})k_{(2)} = l$; while, on the other hand, $S((k\dpu h_{(2)})h_{(1)}) = S(l)$, which implies that $(k\dpu h_{(2)})h_{(1)}=l$, since $S$ is injective. So $h\mdoubletimes k = (k\dpu h_{(2)})h_{(1)} = (h\cdot k_{(1)})k_{(2)} = k\times h$.
\end{proof}

\begin{remark} In the set-theoretic context, Proposition~\ref{caracterizacion de skew brazas} corresponds to Corollary 3 of~\cite{R2}*{Proposition~6}.
\end{remark}

\begin{definition}\label{def: linear q-cycle coalgebra} A {\em linear $q$-cycle coalgebra} is a tuple $\mathcal{H}\coloneqq (H,\times,T_{\times},\cdot,1)$, consisting of a coalgebra~$H$ with a distinguished group-like element~$1$, a map $h\ot k\mapsto h\times k$ from $H^2$ to $H$, a coalgebra morphism $h\ot k\mapsto h\cdot k$ from $H\ot H^{\cop}$ to $H$, and a map $T_{\times}\colon H\to H$, such that:
	
\begin{enumerate}[itemsep=0.7ex, topsep=1.0ex, label={\arabic*)}]
		
\item $(H,\cdot)$ is a regular magma coalgebra,
		
\item the map $h\ot k\mapsto hk\coloneqq k_{(2)}\times h^{k_{(1)}}$ is a coalgebra morphism from $H^2$ to $H$,
		
\item $\times$ is an associative operation with identity~$1$,
		
\item the map $h\mapsto S(h)\coloneqq T_{\times}(h_{(1)})\cdot h_{(2)}$ is a coalgebra antiautomorphism of $H$,
		
\item $(k\times l)\cdot h= (k\cdot h_{(1)})\times (l\cdot h_{(2)})$ for all $h,k,l\in H$,
		
\item $h\cdot (k\times l)= (h\cdot k_{(2)})\cdot (l\cdot k_{(1)})$ for all $h,k,l\in H$,
		
\item $h_{(2)}\times T_{\times}(h_{(1)})= T_{\times}(h_{(2)})\times h_{(1)}=\epsilon(h)1$ for all $h\in H$,
				
\item the map $h\ot k\mapsto h\dpu k\coloneqq (k_{(2)}\cdot h_{(1)})h_{(2)}S^{-1}(k_{(1)})$, is a coalgebra morphism, from $H\ot H^{\cop}$ to $H$.
		
\end{enumerate}
	
Let $\mathcal{H}\coloneqq (H,\times,T_{\times},\cdot, 1)$ and $\mathcal{L}\coloneqq (L,\times,T_{\times},\cdot, 1)$ be linear $q$-cycle coalgebras. A {\em morphism} $f\colon \mathcal{H}\to \mathcal{L}$ is a coalgebra morphism $f\colon H\to L$, such that $f(1)=1$ and such that for all $h,k\in H$,
$$
f(h\times k)=f(h)\times f(k),\quad f(T_{\times}(h))=T_{\times}(f(h))\quad\text{and}\quad f(h\cdot k)=f(h)\cdot f(k).
$$
\end{definition}

\begin{remark}\label{1 cdot k = epison(k)1} Items~3), 5) and~7) imply that $1\cdot h=\epsilon(h)1$ for all $h\in H$.
\end{remark}

\begin{remark}\label{la def de linear q-cycle coalgebra  en el caso coconmutativo} Assume that $H$ is cocommutative. Then item~8) of Definition~\ref{def: linear q-cycle coalgebra} is trivially satisfied and item~2) is fulfilled if and only if $\times$ is a coalgebra morphism (use Proposition~\ref{UTIL}). Consequently, in this case $(H,\Delta, \times)$ is a Hopf algebra with antipode $T_{\times}$. Even more, $\mathcal{H}= (H,\times,T_{\times},\cdot, 1)$ is a linear $q$-cycle coalgebra if and only if $(H,\Delta,\times)$ is a Hopf algebra with identity $1$ and antipode $T_{\times}$, and items~1), 5) and~6) are fulfilled.
\end{remark}

\begin{remark} When $H$ is a group algebra $k[G]$, then a linear $q$-cycle coalgebra structure on $H$ induces by restriction a linear $q$-cycle set structure on $G$ in the sense of~\cite{R2}*{Definition 6}. Conversely, each linear $q$-cycle set structure on $G$ yields a unique linear $q$-cycle coalgebra structure on $k[G]$.
\end{remark}

\begin{lemma}\label{punto y exponencial son inversas} Let $H$ be a Hopf algebra with bijective an\-tipode $S$ and let $h\ot k\mapsto h\times k$ and $h\mapsto T_{\times}(h)$ be morphisms, from $H^2$ to $H$ and from $H$ to $H$, respectively. If the binary operation $\times$ is associative with unity~$1$ and $h_{(2)}\times T_{\times}(h_{(1)}) = T_{\times}(h_{(2)}) \times h_{(1)} =\epsilon(h)1$, for all $h\in H$, then the maps $k^h\coloneqq T_{\times}(h_{(2)})\times kh_{(1)}$ and $k\cdot h\coloneqq (h_{(1)}\times k)S(h_{(2)})$, satisfy the equalities ${k^{h_{(1)}}}\cdot h_{(2)} = (k\cdot h_{(1)})^{h_{(2)}} = \epsilon(h)k$, for all $h,k\in H$.
\end{lemma}

\begin{proof} In fact
\begin{align*}
& k^{h_{(1)}}\cdot h_{(2)} = (T_{\times}(h_{(2)})\times kh_{(1)})\cdot h_{(2)} = (h_{(3)}\times T_{\times}(h_{(2)})\times kh_{(1)})S(h_{(4)}) =\epsilon(h)k 
\shortintertext{and}
& (k\cdot h_{(1)})^{h_{(2)}} = T_{\times}(h_{(3)})\times (k\cdot h_{(1)})h_{(2)} = T_{\times}(h_{(4)})\times (h_{(1)}\times k)S(h_{(2)})h_{(3)} = \epsilon(h)k, 
\end{align*}
as desired.
\end{proof}

The following lemma generalizes~\cite{GV}*{Proposition~1.9}.

\begin{lemma}\label{lema 1 de GV-skew braces} Let $H$ be a Hopf algebra with bijective an\-tipode $S$ and let $h\ot k\mapsto h\times k$ and $h\mapsto T_{\times}(h)$ be morphisms, from $H^2$ to $H$ and from $H$ to $H$, respectively. Assume that $\times$ is associative with unity $1$ and set $k^h\coloneqq T_{\times}(h_{(2)})\times kh_{(1)}$. The following assertions are equivalent:

\begin{enumerate}

\item $(k\times l)h=k h_{(3)}\times T_{\times}(h_{(2)})\times lh_{(1)}$, for all $h,k,l\in H$.

\item $h_{(2)}\times T_{\times}(h_{(1)}) = T_{\times}(h_{(2)})\times h_{(1)}=\epsilon(h)1$ and $k^{hl} = (k^h)^l$, for all $h,k,l\in H$.

\item $h_{(2)}\times T_{\times}(h_{(1)}) = T_{\times}(h_{(2)})\times h_{(1)}=\epsilon(h)1$ and $(k\times l)^h = k^{h_{(2)}}\times l^{h_{(1)}}$, for all $h,k,l\in H$.

\end{enumerate}

\end{lemma}

\begin{proof} 1) $\Rightarrow$ 2)\enspace The first condition in item~2) follows applying item~1) to 
$$
\epsilon(h)1=(S^{-1}(h_{(2)})\times 1)h_{(1)} = (1\times S(h_{(1)}))h_{(2)}. 
$$
Since, by definition,
$$
k^{hl} = T_{\times}(h_{(2)}l_{(2)})\times kh_{(1)}l_{(1)}\quad\text{and}\quad (k^h)^l = T_{\times}(l_{(2)})\times k^hl_{(1)} = T_{\times}(l_{(2)})\times (T_{\times}(h_{(2)})\times kh_{(1)}) l_{(1)}.
$$
the second condition in item~2) follows from the fact that, applying item~1) twice, we obtain
$$
l_{(3)}\times T_{\times}(h_{(2)}l_{(2)})\times k h_{(1)}l_{(1)} = (S^{-1}(h_{(2)})\times k)h_{(1)}l = (T_{\times}(h_{(2)})\times kh_{(1)})l.
$$

\smallskip

\noindent 2) $\Rightarrow$ 3)\enspace Let $\cdot$ be as in the previous lemma. We have
\begin{align*}
(k\times l)^h & = ((l\cdot k_{(1)})k_{(2)})^h &&\text{by definition}\\
& = T_{\times}(h_{(2)}) \times (l\cdot k_{(1)})k_{(2)}h_{(1)}&&\text{by definition}\\
& = T_{\times}(h_{(4)})\times k_{(4)}h_{(3)}\times T_{\times}(k_{(3)}h_{(2)})\times (l\cdot k_{(1)})k_{(2)}h_{(1)}\\
& = {k_{(3)}}^{h_{(2)}}\times (l\cdot k_{(1)})^{k_{(2)}h_{(1)}}\\
& = {k_{(3)}}^{h_{(2)}}\times ((l\cdot k_{(1)})^{k_{(2)}})^{h_{(1)}} &&\text{by hypothesis}\\
& = k^{h_{(2)}}\times l^{h_{(1)}} &&\text{by Lemma~\ref{punto y exponencial son inversas}}
\end{align*}
as desired.

\smallskip

\noindent 3) $\Rightarrow$ 1)\enspace This follows immediately from the fact that 
$$
k^{h_{(2)}}\times l^{h_{(1)}} = T_{\times}(h_{(4)})\times kh_{(3)}\times T_{\times}(h_{(2)})\times lh_{(1)}\quad\text{and}\quad (k\times l)^h =  T_{\times}(h_{(2)})\times (k\times l)h_{(1)},
$$ 
for all $h,k,l\in H$.
\end{proof}

\begin{definition}\label{vendramin skew braza} A {\em GV-Hopf skew-brace} is a tuple $\mathcal{H}\coloneqq (H,\times,T_{\times})$, consisting of a Hopf algebra $H$ with bijective an\-tipode $S$ and maps $h\ot k\mapsto h\times k$ and $h\mapsto T_{\times}(h)$, from $H^2$ to $H$ and from $H$ to $H$, respectively, such~that:
	
\begin{enumerate}[itemsep=0.7ex, topsep=1.0ex, label={\arabic*)}]
		
\item $\times$ is associative with unity $1$,
		
\item $(k\times l)h=k h_{(3)}\times T_{\times}(h_{(2)})\times lh_{(1)}$ for all $h,k,l\in H$,
						
\item $h\ot k \mapsto (k_{(1)}\times h)S(k_{(2)})$ is a coalgebra morphism from $H\ot H^{\cop}$ to $H$,
		
\item $h\ot k \mapsto (h\times k_{(2)})S^{-1}(k_{(1)})$ is a coalgebra morphism from $H\ot H^{\cop}$ to $H$.
		
\end{enumerate}
Let $\mathcal{K}=(K,\times,T_{\times})$ be another GV-Hopf skew-brace. A {\em morphism} $f\colon \mathcal{H}\to \mathcal{K}$ is a morphism $f\colon H\to K$,~of Hopf~alge\-bras, such that $f(h\times h')=f(h)\times f(h')$ for all $h,h'\in H$.
\end{definition}

\begin{remark}\label{morfismos de Vendramin ... respetan T_x} By Lemma~\ref{lema 1 de GV-skew braces}, the map $T_{\times}$ is the convolution inverse of $\ide$ in $\End_{\Delta^{\cop},\times}(H)$. Hence each GV-Hopf skew-brace morphism $f\colon \mathcal{H}\to \mathcal{K}$, satisfies $f\xcirc T_{\times}=T_{\times} \xcirc f$.
\end{remark}

\begin{remark}\label{simetria de GV-skew braces} If $(H,\times,T_{\times})$ is a GV-Hopf skew-brace, then $(H^{\cop},\mdoubletimes,T_{\mdoubletimes})$, also is (here $\mdoubletimes$ is the opposite multiplication of $\times$ and $T_{\mdoubletimes} = T_{\times}$). Moreover the definition $k^h\coloneqq T_{\times}(h_{(2)})\times kh_{(1)}$, given in Lemma~\ref{punto y exponencial son inversas} for $(H,\times,T_{\times})$, in $(H^{\cop},\mdoubletimes,T_{\mdoubletimes})$ reads $k_h \coloneqq T_{\mdoubletimes}(h_{(1)})\mdoubletimes kh_{(2)} = kh_{(2)}\times T_{\times}(h_{(1)})$. Furthermore, by~Lem\-ma~\ref{lema 1 de GV-skew braces}, we have
$$
k_{hl} = (k_h)_l\quad\text{and}\quad (k\times l)_h = (l\mdoubletimes k)_h = l_{h_{(1)}}\mdoubletimes k_{h_{(2)}} = k_{h_{(2)}}\times l_{h_{(1)}} \qquad\text{for all $h,k,l\in H$.}  
$$
\end{remark}

\begin{remark} Item~1) in Definition~\ref{vendramin skew braza} can be replaced by
	
\begin{enumerate}[itemsep=0.7ex, topsep=1.0ex, label={\arabic*')}]

\item $\times$ is associative with unity and $T_{\times}$ is the convolution inverse of $\ide$ in $\End_{\Delta^{\cop},\times}(H)$.
	
\end{enumerate}	
In fact, by item~2), we have $k\times 1 =(k\times 1)1= k\times T_{\times}(1)\times 1=k$, which proves that the unity of $\times$ is $1$.
\end{remark}

\begin{remark}\label{def original de skew braza en el casos coconmutativo} When $H$ is cocommutative, then items~3) and~4) are equivalent and they are fulfilled if and only if $\times\colon H^2\to H$ is a coalgebra morphism. So, by Remark~\ref{morfismos de Vendramin ... respetan T_x}, in the cocommutative case, our definition reduces to the definition of Hopf brace structure given in~\cite{AGV}*{Definition 1.1}. For the set-theoretic case see~\cite{AGV}*{Example~1.4}.
\end{remark}

\begin{theorem}\label{equivalencia de nociones de skew-brace, etc} The three notions of Hopf skew-brace with bijective antipode, linear $q$-cycle coalgebra and GV-Hopf skew-brace are equivalent. More precisely:

\begin{enumerate}[itemsep=0.7ex, topsep=1.0ex, label=\emph{\arabic*)}]

\item If $(H,\cdot,\dpu)$ is a Hopf skew-brace with bijective antipode, then $(H,\times, T_{\times},\cdot, 1)$ is a linear $q$-cycle~coalgebra, where $1$ is the unit of $H$, $h\times k\coloneqq (k\cdot h_{(1)})h_{(2)}$ and $T_{\times}(h)\coloneqq S(h_{(1)})^{h_{(2)}} = S(h_{(1)})\cdot S^{-1}(h_{(2)})$.

\item If $(H,\times,T_{\times},\cdot, 1)$ is a linear $q$-cycle coalgebra, then the coalgebra $H$, endowed with the mul\-tiplication $hk\coloneqq k_{(2)}\times h^{k_{(1)}}$, is a Hopf algebra with bijective antipode $S(h)\coloneqq T_{\times}(h_{(1)})\cdot h_{(2)}$. Moreover, the triple $(H,\times,T_{\times})$ is a GV-Hopf skew-brace. 

\item If $(H,\times,T_{\times})$ is a GV-Hopf skew-brace, then $(H,\cdot,\dpu)$ is a Hopf skew-brace with bijective antipode, where $h\cdot k\coloneqq (k_{(1)}\times h)S(k_{(2)})$ and $h\dpu k\coloneqq (h\times k_{(2)})S^{-1}(k_{(1)})$.

\end{enumerate}

\end{theorem}

\begin{proof} 1)\enspace Item~1) of Definition~\ref{def: linear q-cycle coalgebra} follows from the definition of Hopf skew brace, items~2), 3), 4), 5) and~7), from items~1), 2), 3), 4) and~7) of Theorem~\ref{otra caracterizacion de Hopf q-brace}, respectively. Item~6) follows from item~6) of the same theorem, noting that $h\cdot (k\times l)= h\cdot((l\cdot k_{(1)})k_{(2)})=(h\cdot k_{(2)})\cdot (l\cdot k_{(1)})$. Finally, item~8) holds, because by~\eqref{compatibilidad cdot dpu en skew brazas}, in each Hopf skew brace with bijective antipode the coalgebra morphism $\dpu$ satisfies
$$
h:k=(h\dpu k_{(3)})k_{(2)} S^{-1}(k_{(1)})=(k_{(2)}\cdot h_{(1)})h_{(2)}S^{-1}(k_{(1)}).
$$

\smallskip

\noindent 2)\enspace By items~1) and~6) of Definition~\ref{def: linear q-cycle coalgebra}, we have
\begin{equation}\label{es modulo}
h\cdot kl = h\cdot (l_{(2)}\times k^{l_{(1)}}) = (h\cdot l_{(3)})\cdot (k^{l_{(1)}}\cdot l_{(2)}) = (h\cdot l)\cdot k.
\end{equation}
So, the hypotheses of Lemma~\ref{H es un algebra de Hopf} are satisfied. Hence $H$, endowed with the multiplication map \hbox{$h\ot k\mapsto hk$}, is a Hopf algebra with unit~$1$ and bijective antipode $S$. Moreover, item~1) of Definition~\ref{vendramin skew braza} holds, since it is item~3) of Defini\-tion~\ref{def: linear q-cycle coalgebra}; while item~4) of Definition~\ref{vendramin skew braza} follows from item~8) of Defini\-tion~\ref{def: linear q-cycle coalgebra}, because
\begin{equation}\label{ecua}
(k\cdot h_{(1)})h_{(2)} = h_{(3)}\times (k\cdot h_{(1)})^{h_{(2)}} = h\times k.
\end{equation}
A direct computation using~\eqref{ecua}, proves that $(k_{(1)}\times h)S(k_{(2)})=h\cdot k$, for all $h,k\in H$. Hence item~3) of Definition~\ref{vendramin skew braza} follows from item~1) of Defini\-tion~\ref{def: linear q-cycle coalgebra}. It remains to check item~2) of Definition~\ref{vendramin skew braza}. By item~7) of Definition~\ref{def: linear q-cycle coalgebra}, we can apply Lemma~\ref{punto y exponencial son inversas}, which shows that  $k^h=T_{\times}(h_{(2)})\times kh_{(1)}$ (see~Re\-mark~\ref{unicidad de exponencial}). Moreover, by item~5) of Definition~\ref{def: linear q-cycle coalgebra},
$$
(k\times l)^h=((k^{h_{(2)}}\cdot h_{(3)})\times (l^{h_{(1)}}\cdot h_{(4)}))^{h_{(5)}} =((k^{h_{(2)}}\times l^{h_{(1)}})\cdot h_{(3)})^{h_{(4)}}= k^{h_{(2)}}\times l^{h_{(1)}},
$$
for all $h,k,l\in H$. Hence, item~2) of Definition~\ref{vendramin skew braza} follows immediately from Lemma~\ref{lema 1 de GV-skew braces}.

\smallskip

\noindent 3)\enspace Identity~\eqref{def: Hopf skew-brace} holds by the definitions of $\cdot$ and~$\dpu$. Moreover, by items~3) and~4) of Definition~\ref{vendramin skew braza}, Proposition~\ref{2 de 3} and Remark~\ref{d y p son morfismos de coalgears}, we know that $(H,\cdot,\dpu)$ is a $q$-magma coalgebra. Furthermore, by items~1) and~2) of Definition~\ref{vendramin skew braza} and Lemma~\ref{lema 1 de GV-skew braces}, we know that $h_{(2)}\times T_{\times}(h_{(1)}) = T_{\times}(h_{(2)})\times h_{(1)}=\epsilon(h)1$, for all $h\in H$. Hence, we can apply Lemma~\ref{punto y exponencial son inversas}, which shows that $(H,\cdot,\dpu)$ is left regular (see Remark~\ref{notacion s1 y s2}). Consequently, by Proposition~\ref{skew-braces con S biyectiva}, in order to complete the proof, we only need to check that $H$ is a right $H^{\op}$-module via $\cdot$ and $\dpu$. For $\cdot$ this follows from item~1) of Lemma~\ref{algunas cuentas}, since $H$ is a right $H$-module via $k\ot h\mapsto k^h$ (by Lemma~\ref{lema 1 de GV-skew braces}). For $\dpu$ this follows from the same argument, applied to the GV-Hopf skew-brace $(H^{\cop},\mdoubletimes,T_{\mdoubletimes})$, taking into account that $h\dpu k = (k_{(2)}\mdoubletimes h)S^{-1}(k_{(1)})$.
\end{proof}

\begin{remark} In the set-theoretic context, Theorem~\ref{equivalencia de nociones de skew-brace, etc} yields Corollary~2 of~\cite{R2}*{Proposition 4}.
\end{remark}

\subsection[Hopf skew-braces and invertible \texorpdfstring{$1$}{1}-cocycle]{Hopf skew-braces and invertible \texorpdfstring{$\mathbf{1}$}{1}-cocycle}

\begin{proposition}\label{def equiv de Vendramin Hopf skew-brace'} A triple $\mathcal{H}\coloneqq (H,\times, T_{\times})$, consisting of a Hopf algebra $H$ with bijective antipode $S$ and maps $\times \colon H^2 \to H$ and $T_{\times}\colon H\to H$, is a GV-Hopf skew-brace if and only if the following conditions hold:
	
\begin{enumerate}[itemsep=0.7ex, topsep=1.0ex, label=\emph{\arabic*)}]
		
\item $\times$ is associative with identity~$1$,
		
\item $T_{\times}$ is the inverse of $\ide$ in $\End_{\Delta^{\cop},\times}(H)$,
		
\item the vector space $H$ is a right $H$-module coalgebra both via $h\ot k\mapsto h^k\coloneqq T_{\times}(k_{(2)})\times hk_{(1)}$ and via $h\ot k\mapsto h_k\coloneqq hk_{(2)}\times T_{\times}(k_{(1)})$,
		
\item $(k\times l)_h=k_{h_{(2)}}\times l_{h_{(1)}}$, for all $h,k,l\in H$.
	
\end{enumerate}
\end{proposition}

\begin{proof} Assume first that $\mathcal{H}$ is a GV-Hopf skew-brace. Item~1) holds by item~1) of Definition~\ref{vendramin skew braza}, while items~2), 3) and~4) follow from Lemma~\ref{lema 1 de GV-skew braces} and Remark~\ref{simetria de GV-skew braces}. We now assume that items~1)--4) are satisfied and prove the converse. Item~1) of~Defini\-tion~\ref{vendramin skew braza} is trivial. Item~2) of~Defini\-tion~\ref{vendramin skew braza} follows from the equivalence between items~1) and 2) of Lemma~\ref{lema 1 de GV-skew braces}. Consequently,
$$
(k_{(1)}\times h)S(k_{(2)}) = k_{(1)}S(k_{(2)})\times T_{\times}(S(k_{(3)})) \times hS(k_{(4)}) = T_{\times}(S(k)_{(2)}) \times hS(k)_{(1)} = h^{S(k)}.
$$
Since $S\colon H^{\cop}\to H$ is a coalgebra map, this combined with item~3) shows that item~3) of Definition~\ref{vendramin skew braza} is fulfilled. Finally, we have
$$
(h\times k_{(2)})S^{-1}(k_{(1)}) = hS^{-1}(k_{(1)})\times T_{\times}(S^{-1}(k_{(2)})) \times k_{(4)}S^{-1}(k_{(3)}) = hS^{-1}(k)_{(2)}\times T_{\times}(S^{-1}(k)_{(1)}) = h_{S^{-1}(k)} ,
$$
which, combined with item~3), shows that item~4) of Definition~\ref{vendramin skew braza} is also fulfilled, because $S^{-1}\colon H^{\cop}\to H$ is a coalgebra map.
\end{proof}

Let $H$ be a Hopf algebra with bijective antipode and let $(L,\times', T_{\times'}, \mathbin{\hookleftarrow})$ be a coalgebra $L$ endowed with a binary operation $\times'\colon L^2\to L$, a map $T_{\times'}\colon L\to L$ and a right action $\mathbin{\hookleftarrow} \colon L\ot H\to L$ such that $\times'$ is associative with unit $1_L$, $T_{\times'}$ is the inverse of $\ide$ in~$\End_{\Delta^{\cop},\times'}(L)$, $L$ is a right $H$-module coalgebra via $\mathbin{\hookleftarrow}$ and $(k\times' l) \mathbin{\hookleftarrow} h= (k\mathbin{\hookleftarrow} h_{(2)})\times' (l\mathbin{\hookleftarrow} h_{(1)})$, for all $k,l\in L$ and $h\in H$.

\begin{definition}\label{1 cociclo biyectivo version skew brazas} A {\em bijective $1$-cocycle} of $H$ with values in $(L,\times', T_{\times'}, \mathbin{\hookleftarrow})$ is a coalgebra isomorphism $\pi \colon H\to L$, such that:

\begin{enumerate}[itemsep=0.7ex, topsep=1.0ex, label={\arabic*)}]
	
\item $L$ is a right $H$-module coalgebra via $l\ot h\mapsto T_{\times'}(\pi(h_{(3)})) \times' (l \mathbin{\hookleftarrow} h_{(2)}) \times' \pi(h_{(1)})$,
	
\item $\pi(hk)=(\pi(h)\mathbin{\hookleftarrow} k_{(2)})\times' \pi(k_{(1)})$, for all $h,k\in H$ (cocycle condition).
	
\end{enumerate}
Let $\xi \colon K\to J$ be another bijective $1$-cocycle. A {\em morphism} from $\pi$ to $\xi$ is a pair $(f,g)$, where $f\colon H\to K$ is a Hopf algebra morphism and $g\colon L\to J$ is a coalgebra morphism, such that $\xi\xcirc f = g\xcirc \pi$, $g(1_L)=1_J$, $g(l\times' k)=g(l)\times' g(k)$ and $g(l \mathbin{\hookleftarrow} h) = g(l) \mathbin{\hookleftarrow} f(h)$, for all $h\in H$ and $l,k\in L$.
\end{definition}

\begin{remark} In the set-theoretic context, Definition~\ref{1 cociclo biyectivo version skew brazas} reduces to the definition above~\cite{GV}*{Proposition~1.11} (Note that they use left actions).
\end{remark}

\begin{remark}\label{compatibilidad con T_x} If $(f,g)\colon\pi\to\xi$ is a morphism of bijective $1$-cocycles, then $g\xcirc T_{\times'}=T_{\times'}\xcirc g$.
\end{remark}

\begin{remark}\label{pi de uno ache es uno ele} Note that $\pi(1_H)=1_L$. In fact, the cocycle condition implies that $\pi(1_H)\times' \pi(1_H) = \pi(1_H)$, which combined with the fact that $\pi$ is a coalgebra map and $T_{\times'}$ is the inverse of $\ide$ in~$\End_{\Delta^{\cop},\times'}(L)$ yields
$$
\pi(1_H)=\pi(1_H)\times' \pi(1_H)\times' T_{\times'}(\pi(1_H))=  \pi(1_H) \times' T_{\times'}(\pi(1_H))=\epsilon(\pi(1_H))1_L=1_L,
$$
as desired.
\end{remark}

\begin{proposition}\label{skew brazas vs 1 cocyclos biyectivos} If $\mathcal{H}=(H,\times,T_{\times})$ is a GV-Hopf skew-brace, then $\ide_H$ is a bijective $1$-co\-cycle of $H$ with values in $(H,\times,T_{\times},\mathbin{\hookleftarrow})$, where $h\mathbin{\hookleftarrow} k\coloneqq h_k$. Conversely, if $\pi\colon H\to L$ is an bijective $1$-cocycle of $H$ with values in $(L,\times',T_{\times'},\mathbin{\hookleftarrow})$, then $\mathcal{H}\coloneqq (H,\times, T_{\times})$ is a GV-Hopf skew-brace, where $h\times k\coloneqq \pi^{-1}\bigl(\pi(h)\times' \pi(k)\bigr)$ and $T_{\times}(h)\coloneqq \pi^{-1}\bigl(T_{\times'}(\pi(h))\bigr)$.
\end{proposition}

\begin{proof} Let $\mathcal{H}=(H,\times,T_{\times})$ be a GV-Hopf skew-brace and let $(H,\times,T_{\times},\mathbin{\hookleftarrow})$ be as in the statement. By Proposition~\ref{def equiv de Vendramin Hopf skew-brace'}, the conditions above Definition~\ref{1 cociclo biyectivo version skew brazas} are satisfied. Item~2) of Definition~\ref{1 cociclo biyectivo version skew brazas} follows from the equality $h_k = hk_{(2)}\times T_{\times}(k_{(1)})$ in Remark~\ref{simetria de GV-skew braces}. By this and the definition of $l^h$ given in Lemma~\ref{punto y exponencial son inversas}, we obtain that
$$
T_{\times'}(\pi(h_{(3)})) \times' (l \mathbin{\hookleftarrow} h_{(2)}) \times' \pi(h_{(1)}) = T_{\times}(h_{(3)}) \times l_{h_{(2)}} \times h_{(1)} = T_{\times}(h_{(2)}) \times l h_{(1)} = l^h,
$$
which yields item~1) (with $\pi = \ide_H$, $\times'=\times$  and $l\mathbin{\hookleftarrow} h = l_h$), since $l\ot h\mapsto l^h$ is a coalgebra morphism.

We next take a bijective $1$-cocycle $\pi\colon H\to L$, of $H$ with values in $(L,\times', T_{\times'}, \mathbin{\hookleftarrow})$, and prove the converse. For this it suffices to show that conditions~1)--4) of Proposition~\ref{def equiv de Vendramin Hopf skew-brace'} are satisfied. Clearly  the operation~$\times$ is associative, $1_H$ is its unit (by Remark~\ref{pi de uno ache es uno ele}) and $T_{\times}$ is the convolution inverse of $\ide_H$ in $\End_{\Delta^{\cop},\times}(H)$. So, conditions~1) and~2) are satisfied. We now prove that conditions~3) and~4) are also. Since, by the cocycle condition,
$$
h_k\coloneqq hk_{(2)}\times T_{\times}(k_{(1)}) = \pi^{-1}\bigl((\pi(h)\mathbin{\hookleftarrow} k_{(3)})\times' \pi(k_{(2)})\bigr)\times T_{\times}(k_{(1)}) = \pi^{-1}(\pi(h)\mathbin{\hookleftarrow} k),
$$
$H$ is a right $H$-module coalgebra via $h\ot k\mapsto h_k$; while, by the last condition above Definition~\ref{1 cociclo biyectivo version skew brazas}, item~4) of Proposition~\ref{def equiv de Vendramin Hopf skew-brace'} is fulfilled. Finally, by the cocycle condition and item~2) of Definition~\ref{1 cociclo biyectivo version skew brazas},
$$
h^k\coloneqq T_{\times}(k_{(2)})\times hk_{(1)} = \pi^{-1}\bigl(T_{\times'}(\pi (k_{(2)})) \times' \pi (hk_{(1)})\bigr) = \pi^{-1}\bigl(T_{\times'} (\pi(k_{(3)})) \times' (\pi(h) \mathbin{\hookleftarrow} k_{(2)}) \times' \pi(k_{(1)})\bigr),
$$
and so, by item~1) of Definition~\ref{1 cociclo biyectivo version skew brazas}, $H$ is a right $H$-module coalgebra via $h\ot k\mapsto h^k$, which completes the proof of condition~3).
\end{proof}

\begin{remark} In the set-theoretic context, Proposition~\ref{skew brazas vs 1 cocyclos biyectivos} yields~\cite{GV}*{Proposition~1.11}.
\end{remark}

\begin{remark}\label{equivalencia entre skew brazas y 1 cocyclos biyectivos} The correspondences in Proposition~\ref{skew brazas vs 1 cocyclos biyectivos} yield an equivalence between the categories of~GV-Hopf skew-braces and bijective $1$-cocycles.
\end{remark}

\subsection{Comparison between Hopf-skew braces and Yetter-Drinfeld braces}\label{italianos}

Let $H$ be a Hopf algebra and let $s_1,s_2\colon H\ot H\to H$ be maps. For each $h,l\in H$, set ${}^hl\coloneqq s_1(h\ot l)$ and $h^l\coloneqq s_2(h\ot l)$. Recall that $(H,s_1,s_2)$ is a matched pair if conditions~1)--6) above Theorem~\ref{pares apareados y estructuras de bialgebra en el producto} are fulfilled. Assume that $(H,s_1,s_2)$ is a matched pair and let $s\colon H^2\to H^2$ be the map defined by
$$
s(h\ot l)\coloneqq {}^{h_{(1)}}l_{(1)}\ot {h_{(2)}}^{l_{(2)}}. 
$$
As we saw at the beginning of the proof of Theorem~\ref{Condicion suficiente para braided operator}, 
$$
(\ide\ot\epsilon)\xcirc s = s_1,\quad (\epsilon\ot \ide)\xcirc s = s_2\quad\text{and}\quad \text{$s$ is a coalgebra morphism.} 
$$
We say that a matched pair $(H,s_1,s_2)$ is {\em commutative} if $s$ is a left non-degenerate solution of the braid equation and condition~\eqref{smuesmu} is satisfied. We let $\mathcal{MP}1$ denote the category of commutative matched pairs.

\smallskip

Let $H$ and $s_1,s_2\colon H\ot H\to H$ be as above. Recall from \cite{FS}*{Definition~2.1}, that the triple $(H,s_1,s_2)$ is a {\em matched pair of actions} if condition~\eqref{smuesmu} and conditions 1)--5) above Theorem~\ref{pares apareados y estructuras de bialgebra en el producto} are satisfied. The definition of {\em morphisms of matched pairs of actions} is the same as that of matched pair morphisms. We let $\mathcal{MP}2$ denote the category of matched pairs of actions. Clearly $\mathcal{MP}1$ is a full subcategory of $\mathcal{MP}2$, but by Theorem~\ref{Condicion suficiente para braided operator}, actually $\mathcal{MP}1 = \mathcal{MP}2$.

\smallskip 

Let $H$ be a Hopf algebra. A {\em (left-left) Yetter–Drinfeld module on $H$} is the datum of a left $H$-module~$X$, via an action $h\ot x\mapsto {}^hx$, which is also a left $H$-comodule, via a coaction $\rho\colon X\to H\ot X$, such that, for all $h\in H$ and $x\in X$, 
$$
\rho({}^hx) = h_{(1)}x_{(-1)}S(h_{(3)})\ot {}^{h_{(2)}}x_{(2)}\quad\text{where $x_{(-1)}\ot x_{(0)}\coloneqq \rho(x)$}.
$$
A morphism of Yetter–Drinfeld modules is a morphism of both left $H$-modules and left $H$-comodules. It is well known that the category ${}^H_H \mathcal{YD}$, of Yetter–Drinfeld modules on $H$, is a braided tensor category.

\smallskip

The following Definition is taken from~\cite{FS}*{Definition~3.16}.

\begin{definition}\label{Yetter–Drinfeld braces} A {\em Yetter–Drinfeld brace} is a triple $(H,\times,T)$, consisting of a Hopf algebra $H$, an associative internal operation $\times$, of $H$, and a map $T\colon H\to H$, such that:

\begin{itemize}

\item[-] The underlying coalgebra of $H$, endowed with the multiplication $\times$, the same unit of $H$, the left action ${}^hl\coloneqq T(h_{(1)})\times h_{(2)}l$ and the coaction $\rho(h)\coloneqq h_{(1)}S(h_{(3)})\ot h_{(2)}$, is a Hopf algebra~in~${}^H_H \mathcal{YD}$, with antipode $T$,

\item[-] if we define $h^l\coloneqq S({}^{h_{(1)}} l_{(1)}) h_{(2)}l_{(2)}$, then $^{h_{(1)}}l_{(1)}\ot {h_{(2)}}^{l_{(2)}} = {}^{h_{(2)}}l_{(2)}\ot {h_{(1)}}^{l_{(1)}}$, for all $h,l\in H$,

\item[-] $h(k\times l) = h_{(1)}k\times T(h_{(2)})\times h_{(3)}l$, for all $h,k,l\in H$.

\end{itemize}
A morphism from a Yetter–Drinfeld brace $(H,\times,T)$ to a Yetter–Drinfeld brace $(K,\times,T)$, is a Hopf algebra morphism $f\colon H\to K$, such that $f(h\times l) = f(h)\times f(l)$ and $T\xcirc f = f\xcirc T$. 
\end{definition}

\begin{theorem}\label{principal} The categories of braiding operators, Hopf-skew braces and Yetter-Drinfeld braces are~iso\-morphic.
\end{theorem}

\begin{proof} From Remark~\ref{cat iso}, it follows easily that the category $\mathcal{MP}1$, of commutative matched pairs, and the category of braiding operators are isomorphic. On the other hand, in \cite{FS}*{Theorem~3.25} it was proved that the categories of Yetter-Drinfeld braces and $\mathcal{MP}2$, matched pair of actions, are also isomorphic. Since $\mathcal{MP}1 = \mathcal{MP}2$, we conclude that the categories of braiding operators and Yetter-Drinfeld braces are~iso\-morphic. By Remark~\ref{skew-braces con S biyectiva}, this finishes the proof.
\end{proof} 

\begin{remark} By Theorems~\ref{equivalencia de nociones de skew-brace, etc} and~\ref{principal}, the categories of GV-Hopf skew-braces and Yetter-Drinfeld braces with bijective antipode are isomorphic.
\end{remark}|

\section[Ideals of Hopf \texorpdfstring{$q$}{q}-braces]{Ideals of Hopf \texorpdfstring{$\mathbf{q}$}{q}-braces}\label{seccion: Ideals of Hopf q-braces}

In this section we introduce the notions of an ideal and of a Hopf sub $q$-brace of a Hopf $q$-brace with bijective antipode, and we begin the study of these notions.

\begin{definition}\label{ideal de un Hopf q-brace} An {\em ideal} of a Hopf $q$-brace $\mathcal{H} = (H,\cdot,\dpu)$ is a Hopf ideal $I$ of $H$ such that $h\cdot k$, $k\cdot h$, $h\dpu k$ and $k\dpu h$ belong to $I$, for all $h\in H$ and $k\in I$.
\end{definition}

\begin{proposition}\label{caso skew braza} Let $\mathcal{H} = (H,\cdot,\dpu)$ be a Hopf skew-brace. A Hopf ideal $I$ of $H$ is an ideal of $\mathcal{H}$ if and only if $h\dpu k$ and $k\dpu h$ belong to $I$, for all $h\in H$ and $k\in I$. Moreover, if the antipode of $H$ is bijective, then this happens if and only if $h\cdot k$ and $k\cdot h$ belong to $I$, for all $h\in H$ and $k\in I$.
\end{proposition}

\begin{proof} Let $h\in H$, $k\in I$ and assume that $h\dpu k$ and $k\dpu h$ belong to $I$. Since, by identity~\eqref{compatibilidad cdot dpu en skew brazas},
$$
h\cdot k= (k_{(1)}\dpu h_{(2)})h_{(1)}S(k_{(2)})\quad\text{and}\quad k\cdot h= (h_{(1)}\dpu k_{(2)})k_{(1)}S(h_{(2)}),
$$
we obtain that $h\cdot k$ and $k\cdot h$ belong to $I$. The proof of the last assertion is similar.
\end{proof}

\begin{remark}\label{ideal=congruencia} Let $\mathcal{H} = (H,\cdot,\dpu)$ be a Hopf $q$-brace, and let $I$ be a Hopf ideal of $H$. The relation $h\simeq k$ if $h-k\in I$, is a congruence with respect to $\cdot$ and $\dpu$ if and only if $I$ is an ideal of $\mathcal{H}$.
\end{remark}

\begin{remark}\label{Hopf q-algebra cociente} By Remark~\ref{ideal=congruencia}, each ideal $I$ of a Hopf $q$-brace $\mathcal{H} = (H,\cdot,\dpu)$ gives rise to a {\em quotient Hopf $q$-brace} $\mathcal{H}/I\coloneqq (H/I,\cdot,\dpu)$ and a canonical surjective morphism $p\colon \mathcal{H}\to \mathcal{H}/I$, with the evident universal property. In the sequel we let $[h]$ denote the class of $h\in H$ in $\mathcal{H}/I$.
\end{remark}

\begin{example}\label{q-conmutadores} Let $\mathcal{H} = (H,\cdot,\dpu)$ be a Hopf $q$-brace with bijective antipode and let $F,G\colon H\ot H^{\cop}\to H$ be the maps defined by $F(l\ot h) \coloneqq l\times h = (h\cdot l_{(1)})l_{(2)}$ and $G(l\ot h)\coloneqq h\mdoubletimes l = (l\dpu h_{(2)})h_{(1)}$, respectively. Let $*$ be the usual con\-volution product in $\Hom_k(H\ot H^{\cop},H)$. Note that  $G$ is convolution invertible with convolution inverse $G^*(l\ot h)=S^{-1}(h_{(2)})S(l\dpu  h_{(1)})$. We define the $q$-commutator $[h,l]_q$, of~$h,l\in H$,  by
$$
[h,l]_q\coloneqq F*G^*(l\ot h)= (h_{(3)}\cdot l_{(1)}) l_{(2)}S^{-1}(h_{(2)})S(l_{(3)}\dpu h_{(1)})
$$
(note that $[h,l]_q=\epsilon(hl)1$, for all $h,l\in H$ if and only if $\mathcal{H}$ is a skew-brace). By definition, the {\em $q$-commutator
ideal $[\mathcal{H},\mathcal{H}]_q$ of $\mathcal{H}$} is the ideal of $\mathcal{H}$ generated by the set $\{[h,l]_q-\epsilon(hl)1:h,l\in H\}$. Then, the quotient Hopf $q$-brace $\mathcal{H}/[\mathcal{H},\mathcal{H}]_q$ is a skew-brace and the canonical map $p\colon \mathcal{H} \to \mathcal{H}/[\mathcal{H},\mathcal{H}]_q$ is universal in the~fol\-lowing sense: given a Hopf $q$-brace morphism  $f\colon \mathcal{H} \to \mathcal{K}$, with  $\mathcal{K}$ a skew-brace with bijective antipode, there exists a unique Hopf skew-brace morphism $\bar{f}\colon \mathcal{H}/[\mathcal{H},\mathcal{H}]_q\to \mathcal{K}$ such that $f=\bar{f}\xcirc p$.
\end{example}

Next we adapt the notion of ideal of a $q$-brace, introduced in \cite{R2}*{Definition~8}, to the context of Hopf $q$-braces.

\begin{definition}\label{normal subHopf q-brace} Let $\mathcal{H}=(H,\cdot,\dpu)$ be a Hopf $q$-brace with bijective antipode. A {\em Hopf sub $q$-brace} of $\mathcal{H}$ is a normal Hopf subalgebra $A$ of $H$ such that $a\cdot h$, $a\dpu h$, $S^{-1}(h_{(2)}) (h_{(1)}\dpu a)$ and $S(h_{(1)}) (h_{(2)}\cdot a)$ belong to $A$, for all $a\in A$ and $h\in H$.
\end{definition}

\begin{proposition}\label{A^+H para subHopf q braces} Let $\mathcal{H}=(H,\cdot,\dpu)$ be a Hopf $q$-brace with bijective antipode. If $A$ is a Hopf sub $q$-brace of $H$, then $A^{+}H$ is an ideal of $\mathcal{H}$, where $A^{+}\coloneqq A\cap \ker{\epsilon}$.
\end{proposition}

\begin{proof} By \cite{Mo}*{Lemma~3.4.2} we know that $A^{+}H$ is a Hopf ideal of $H$ and $HA^{+}=A^{+}H$. Let $a\in A^{+}$ and $h,k\in H$. By the first identity in~\eqref{condicion q-braza}, the fact that $\epsilon$ is compatible with $\cdot$, and the hypotheses,
$$
ah\cdot k= (a\cdot (k_{(1)}\dpu h_{(2)}))(h_{(1)}\cdot k_{(2)})\in A^{+}H,
$$
and a similar computation proves that also $ah\dpu k\in A^{+}H$. Since $k\cdot ah=(k\cdot h)\cdot a$, in order to prove that $k\cdot ah\in A^{+}H$, it suffices to check that $k\cdot a\in A^{+}H$, for all $k\in H$ and $a\in A^{+}$.  But this follows immediately from the fact that
$$
k\cdot a=k_{(1)}S(k_{(2)})(k_{(3)}\cdot a)\in HA^{+}=A^{+}H.
$$
Finally, a similar computation proves that $k\dpu ah \in A^{+}H$.
\end{proof}

\begin{proposition}\label{condicion suficiente para normal subHopf q-brace} Let $\mathcal{H}=(H,\cdot,\dpu)$ be a Hopf skew-brace with bijective antipode. A normal Hopf subalgebra $A$ of $H$ is a Hopf sub $q$-brace of $\mathcal{H}$ if and only if $a\cdot h\in A$ and $a\dpu h\in A$, for all $a\in A$ and $h\in H$.
\end{proposition}

\begin{proof} $\Rightarrow$)\enspace This is clear.
	
\smallskip

\noindent $\Leftarrow$)\enspace By identity~\eqref{compatibilidad cdot dpu en skew brazas}, we have
\begin{align*}
&S(h_{(1)}) (h_{(2)}\cdot a)= S(h_{(1)}) (h_{(2)}\cdot a_{(1)})a_{(2)}S(a_{(3)})=  S(h_{(1)}) (a_{(1)}\dpu h_{(3)})h_{(2)}S(a_{(2)})
\shortintertext{and}
& S^{-1}(h_{(2)}) (h_{(1)}\dpu a) = S^{-1}(h_{(2)}) (h_{(1)}\dpu a_{(3)})a_{(2)}S^{-1}(a_{(1)})= S^{-1}(h_{(3)}) (a_{(2)}\cdot h_{(1)})h_{(2)}S^{-1}(a_{(1)}).
\end{align*}
Hence $S(h_{(1)}) (h_{(2)}\cdot a),S^{-1}(h_{(2)}) (h_{(1)}\dpu a)\in A$, because $A$ is a normal Hopf subalgebras.
\end{proof}

\begin{proposition}\label{condicionsuficiente para ser cerrado bajo times} Let $\mathcal{H}=(H,\cdot,\dpu)$ be a Hopf $q$-brace with bijective antipode. If $A$ is a Hopf subalgebra~of~$H$ and $k\cdot h\in A$, for all $k,h\in A$, then $A$ is closed under $\times$ and~$T_{\times}$.
\end{proposition}

\begin{proof} This is true since, by definition, $h\times k=(k\cdot h_{(1)})h_{(2)}$ and $T_{\times}(h)=S(h_{(1)})\cdot S^{-1}(h_{(2)})$.
\end{proof}

\begin{proposition}\label{caracterizacion de sub skew braza normal cuando la antipoda es biyectiva} Let $\mathcal{H}=(H,\cdot,\dpu)$ be a Hopf skew-brace with bijective antipode and let~$A$ be a normal Hopf subalgebra of $H$. Assume that $a\cdot h_{(1)}\ot h_{(2)}= a\cdot h_{(2)}\ot h_{(1)}$, for all $a\in A$ and $h\in H$. Then $A$ is a Hopf sub $q$-brace of $\mathcal{H}$ if and only if the following conditions are fulfilled:
	
\begin{enumerate}[itemsep=0.7ex, topsep=1.0ex, label={\emph{\alph*)}}]
		
\item $a\cdot h\in A$, for all $a\in A$ and $h\in H$,
		
\item $h_{(2)}\times a \times T_{\times}(h_{(1)})\in A$, for all $a\in A$ and $h\in H$.
\end{enumerate}
	
\end{proposition}

\begin{proof} By identity~\eqref{compatibilidad cdot dpu en skew brazas}, the fact that $H$ is an $H^{\op}$-module via $\cdot$ and Remark~\ref{d y p son morfismos de coalgears}, we have
\begin{equation}\label{ecua1}
\begin{aligned}
(a\cdot h_{(2)})\dpu S(h_{(1)}) &= \bigl((a\cdot h_{(3)})\dpu S(h_{(1)})_{(2)}\bigr)S(h_{(1)})_{(1)}h_{(2)} \\
&= \bigl(S(h_{(1)}) \cdot (a_{(1)}\cdot h_{(4)})\bigr)(a_{(2)}\cdot h_{(3)})h_{(2)} \\
& = \bigl(S(h_{(1)})\cdot (h_{(3)}S^{-1}(h_{(2)})) \cdot (a_{(1)}\cdot h_{(6)})\bigr)(a_{(2)}\cdot h_{(5)})h_{(4)}\\
&= \bigl((S(h_{(1)})\cdot S^{-1}(h_{(2)})) \cdot (a_{(1)}\cdot h_{(6)})h_{(3)}\bigr)(a_{(2)}\cdot h_{(5)})h_{(4)} \\
&= \bigl(T_{\times}(h_{(1)}) \cdot \bigl((a\cdot h_{(3)})h_{(2)}\bigr)_{(1)}\bigr)\bigl((a\cdot h_{(3)})h_{(2)}\bigr)_{(2)} \\
&= \bigl(T_{\times}(h_{(1)}) \cdot (h_{(2)}\times a)_{(1)}\bigr)(h_{(2)}\times a)_{(2)}\\
&=h_{(2)}\times a \times T_{\times}(h_{(1)}),
\end{aligned}
\end{equation}
By~\eqref{ecua1} it is clear that if $A$ is a Hopf sub $q$-brace of $\mathcal{H}$, then $A$ is a normal Hopf subalgebra of $H$ and items~a) and~b) hold. Conversely, if $A$ is a normal Hopf subalgebra of $H$ and items~a) and~b) are satisfied, then again by~\eqref{ecua1}
$$
a:S(h)=((a\cdot S(h_{(3)}))\cdot h_{(2)})\dpu S(h_{(1)})=h_{(2)}\times (a\cdot S(h_{(3)})) \times T_{\times}(h_{(1)})\in A,
$$
and so, by Proposition~\ref{condicion suficiente para normal subHopf q-brace} and the fact that $S$ is bijective, $A$ is a
Hopf sub $q$-brace of $\mathcal{H}$.
\end{proof}

\subsection{The socle}\label{The socle}
Throughout this subsection we assume that $\mathcal{H}=(H,\cdot,\dpu)$ is a Hopf $q$-brace with bijective antipode.

\begin{definition}\label{def de zocalo} The {\em socle} of $\mathcal{H}$ is the set
$$
\Soc(\mathcal{H})\coloneqq \{h\in H: h_{(1)}\ot k\cdot h_{(2)}\ot h_{(3)}=h_{(1)}\ot k\dpu h_{(2)}\ot h_{(3)}=h_{(1)}\ot k\ot h_{(2)}\ \forall k\in H\}.
$$
\end{definition}

\begin{remark} In the set-theoretic context, Definition~\ref{def de zocalo} corresponds to \cite{R2}*{formula~(31)}
\end{remark}

\begin{remark}\label{acciones del zocalo en las puntas} Let $\wr$ denote any of the operations $\cdot$ or $\dpu$. Note that if $h\in \Soc(\mathcal{H})$, then
\begin{equation}\label{consecuencias}
k\wr h_{(1)}\ot h_{(2)} = k\ot h,\quad h_{(1)}\ot k\wr h_{(2)}= h\ot k \quad\text{and}\quad k\wr h = k\epsilon(h),
\end{equation}
for all $k\in H$. Note also that if $k\wr h_{(1)}\ot h_{(2)}= k\wr h_{(2)}\ot h_{(1)}$, for all $h,k\in H$ (which happens, for instance, if $H$ is cocommutative or if $\cdot$ and $\dpu$ are the trivial actions), then the first two conditions in~\eqref{consecuencias} are equivalent. Moreover, under this hypothesis, if the first condition is true for $\wr=\cdot$ and $\wr=\dpu$, then $h\in\Soc(\mathcal{H})$.
\end{remark}

\begin{proposition}\label{el zocalo es un subalgebra de Hopf} The socle of $\mathcal{H}$ is a Hopf subalgebra of $H$. Moreover 
\begin{equation}\label{cond para sub q-braza}
S(m_{(1)})(m_{(2)}\cdot h),S^{-1}(m_{(2)})(m_{(1)}\dpu h),h\cdot m,h\dpu m\in \Soc(\mathcal{H}),\quad\text{for all $h\in \Soc(\mathcal{H})$ and $m\in H$.}
\end{equation}
\end{proposition}

\begin{proof} It is clear that $1\in \Soc(\mathcal{H})$ and $\Soc(\mathcal{H})$ is a linear sub\-space of $H$. Let $\wr$ denote any of the operations $\cdot$ or $\dpu$. Let $h,l\in \Soc(\mathcal{H})$ and $k\in H$. Since
$$
h_{(1)}l_{(1)}\ot k\wr h_{(2)}l_{(2)}\ot h_{(3)}l_{(3)}=h_{(1)}l_{(1)}\ot (k\wr l_{(2)})\wr h_{(2)}\ot h_{(3)}l_{(3)}=h_{(1)}l_{(1)}\ot k\ot h_{(2)}l_{(2)},
$$
the socle of $\mathcal{H}$ is closed under products. We claim that $\Delta(\Soc(\mathcal{H}))\subseteq \Soc(\mathcal{H})\ot \Soc(\mathcal{H})$. Take $h\in \Soc(\mathcal{H})$ and write $\Delta(h)=\sum_{i=1}^r U^i\ot V^i$ with $r$ minimal. For all $k\in H$, we have
$$
\sum_{i=1}^r U^i_{(1)} \ot k\wr U^i_{(2)} \ot U^i_{(3)}\ot V^i = h_{(1)}\ot k\wr h_{(2)}\ot \Delta(h_{(3)}) =  h_{(1)}\ot k \ot \Delta(h_{(2)})= \sum_{i=1}^r U^i_{(1)} \ot k\ot U^i_{(2)}\ot V^i.
$$
Since the $V^i$'s are linearly independent, this implies that $U^i_{(1)} \ot k\wr U^i_{(2)}\ot U^i_{(3)}= U^i_{(1)} \ot k \ot U^i_{(2)}$ for all~$i$. Thus, $U^i\in\Soc(\mathcal{H})$, for all $i$, and similarly for $V^i$. This finishes the proof of the claim. Let $h\in \Soc(\mathcal{H})$ and $k\in H$. From the equality $h_{(1)}\ot k\ot h_{(2)} = h_{(1)}\ot k\wr h_{(2)} \ot h_{(3)}$, we obtain that
$$
h_{(1)}\ot k\wr S(h_{(2)})\ot h_{(3)} = h_{(1)}\ot (k\wr h_{(3)})\wr S(h_{(2)}) \ot h_{(4)} = h_{(1)}\ot k\wr S(h_{(2)})h_{(3)} \ot h_{(4)} = h_{(1)}\ot k\ot h_{(2)},
$$
and so, $S(\Soc(\mathcal{H}))\subseteq \Soc(\mathcal{H})$. Let $h\in \Soc(\mathcal{H})$ and $m\in H$. By the third equality in~\eqref{consecuencias}, we have
$$
S(m_{(1)})(m_{(2)}\cdot h)=S(m_{(1)})m_{(2)}\epsilon(h)=1\epsilon(mh)= S^{-1}(m_{(2)})m_{(1)}\epsilon(h) =  S^{-1}(m_{(2)})(m_{(1)}\dpu h),
$$
and so $S(m_{(1)})(m_{(2)}\cdot h),S^{-1}(m_{(2)})(m_{(1)}\dpu h)\in \Soc(\mathcal{H})$. We next prove that $h\cdot m,h\dpu m\in \Soc(\mathcal{H})$. For this it suffices to check that, for all $k\in H$,
\begin{align*}
& h_{(1)}\cdot m_{(3)}\ot k\cdot (h_{(2)}\cdot m_{(2)})\ot h_{(3)}\cdot m_{(1)}= h_{(1)}\cdot m_{(2)}\ot k\ot h_{(2)}\cdot m_{(1)},\\
& h_{(1)}\cdot m_{(3)}\ot k\dpu (h_{(2)}\cdot m_{(2)})\ot h_{(3)}\cdot m_{(1)}= h_{(1)}\cdot m_{(2)}\ot k\ot h_{(2)}\cdot m_{(1)},\\
& h_{(1)}\dpu m_{(3)}\ot k\cdot (h_{(2)}\dpu m_{(2)})\ot h_{(3)}\dpu m_{(1)} = h_{(1)}\dpu m_{(2)}\ot k\ot h_{(2)}\dpu m_{(1)}
\shortintertext{and}	
& h_{(1)}\dpu m_{(3)}\ot k\dpu (h_{(2)}\dpu m_{(2)})\ot h_{(3)}\dpu m_{(1)} = h_{(1)}\dpu m_{(2)}\ot k\ot h_{(2)}\dpu m_{(1)}.
\end{align*}
Since $\Soc(\mathcal{H})$ is a subcoalgebra of $H$, this will follow if we show that, for all $m,k\in H$ and $h\in \Soc(\mathcal{H})$,
$$
k \cdot (h\cdot m) = \epsilon(hm)k,\quad k \dpu (h\cdot m) = \epsilon(hm)k,\quad k \dpu (h\dpu m) = \epsilon(hm)k,\quad k \cdot (h\dpu m) = \epsilon(hm)k.
$$
We prove the first identity and leave the others, which are similar, to the reader. Using the first identity in Definition~\ref{def: q cycle coalgebra} and the fact that $H$ is an $H^{\op}$-module via $\cdot$, we obtain that
$$
k\cdot (h\cdot m) = \bigl((k\cdot S(m_{(3)}))\cdot m_{(2)}\bigr)\cdot \bigl(h\cdot m_{(1)}\bigr) = \bigl((k\cdot S(m_{(2)}))\cdot h_{(1)}\bigr)\cdot \bigl(m_{(1)}\dpu h_{(2)}\bigr).
$$
Consequently, since $\Soc(\mathcal{H})$ is a subcoalgebra and $H$ is an $H^{\op}$-module via $\cdot$, we have
$$
k\cdot (h\cdot m) = \bigl((k\cdot S(m_{(2)}))\cdot h_{(1)}\bigr)\cdot \bigl(m_{(1)}\dpu h_{(2)}\bigr) = (k\cdot S(m_{(2)}))\cdot m_{(1)} \epsilon(h)= k \epsilon(hm),
$$
as desired. 
\end{proof}

\begin{proposition}\label{cuando el zocalo es un Hopf sub q-brace} If $k\cdot h_{(1)}\ot h_{(2)}= k\cdot h_{(2)}\ot h_{(1)}$ and $k\dpu h_{(1)}\ot h_{(2)}= k\dpu h_{(2)}\ot h_{(1)}$ for all~$h,k\in H$, then, the socle of $\mathcal{H}$ is a Hopf sub $q$-brace of $\mathcal{H}$.
\end{proposition}

\begin{proof} By the previous proposition we know that $\Soc(\mathcal{H})$ is a Hopf subalgebra of $H$ and that condition~\eqref{cond para sub q-braza} is satisfied. It remains to prove that $\Soc(\mathcal{H})$ is a normal Hopf subalgebra of $H$, which means that $S(h_{(1)})ah_{(2)}, h_{(1)}aS(h_{(2)}\in \Soc(\mathcal{H})$, for all $a\in\Soc(\mathcal{H})$ and $h\in H$. Let $\wr$ denote any of the operations $\cdot$ or $\dpu$ and let $a\in \Soc(\mathcal{H})$ and $h,k\in H$ arbitrary. Then
\begin{align*}
k\wr (S(h_{(1)})a_{(1)}h_{(2)})_{(1)}\ot (S(h_{(1)})a_{(1)}h_{(2)})_{(2)} &= k\wr S(h_{(2)})a_{(1)}h_{(3)}\ot S(h_{(1)})a_{(1)}h_{(4)}\\ 
& = ((k\wr h_{(3)})\wr a_{(1)})\wr S(h_{(2)})\ot S(h_{(1)})a_{(1)}h_{(4)}\\ 
& = (k\wr h_{(3)})\wr S(h_{(2)}\ot S(h_{(1)})a h_{(4)}\\ 
& = k\ot S(h_{(1)})a_{(1)}h_{(2)}
\end{align*}
and similarly
$$
(h_{(1)}a_{(1)}S(h_{(2)}))_{(1)}\ot k\wr (h_{(1)}a_{(1)}S(h_{(2)}))_{(2)} = h_{(1)}a_{(1)}S(h_{(4)})\ot k\wr h_{(2)}a_{(2)}S(h_{(3)}) = h_{(1)} a_{(1)}S(h_{(2)})\ot k.
$$
Hence, by Remark~\ref{acciones del zocalo en las puntas}, $S(h_{(1)})ah_{(2)}, h_{(1)}aS(h_{(2)}\in \Soc(\mathcal{H})$, as desired.
\end{proof}

\begin{proposition}\label{left y right socle} If $h\in \Soc(\mathcal{H})$, then $h\times k = h\mdoubletimes k$ for all $k\in H$.
\end{proposition}

\begin{proof} If $h\in \Soc(\mathcal{H})$, then, by Remark~\ref{acciones del zocalo en las puntas} and Proposition~\ref{el zocalo es un subalgebra de Hopf},
$$
h\times k= (k\cdot h_{(1)})h_{(2)}=kh=(k\dpu h_{(2)})h_{(1)} = h\mdoubletimes k,
$$
for all $k\in H$, as desired.
\end{proof}

\begin{remark} Proposition~\ref{left y right socle} implies Corollary 1 of~\cite{R2}*{Proposition 8}.
\end{remark}

\begin{proposition}\label{H coconmutativo implica que H/Soc H es una skew -braza} Assume that $k\cdot h_{(1)}\ot h_{(2)}= k\cdot h_{(2)}\ot h_{(1)}$ and $k\dpu h_{(1)}\ot h_{(2)}= k\dpu h_{(2)}\ot h_{(1)}$ for all~$h,k\in H$. Then $\mathcal{H}/\Soc(\mathcal{H})^{+}\mathcal{H}$ is a Hopf skew-brace.
\end{proposition}

\begin{proof} Let $F$, $G$ and $G^*$ be as in Example~\ref{q-conmutadores}. By the universal property mentioned in that example, in order to prove the result it suffices to check that the convolution product $F\ast G^*$ takes its values in $\Soc(\mathcal{H})$. By Remark~\ref{acciones del zocalo en las puntas} and our hypothesis, it suffices to check that for all $k,l,h\in H$, 
$$
k\cdot F \ast G^*(l\ot h)_{(1)} \ot  F\ast G^*(l\ot h)_{(2)}=k\dpu F \ast G^*(l\ot h)_{(1)} \ot  F\ast G^*(l\ot h)_{(2)} = k\ot F\ast G^*(l\ot h).
$$
But, by the fact that $H$ is an $H^{\op}$-module via~$\cdot$, the hypothesis and the first condition in Definition~\ref{def: q cycle coalgebra},
\begin{align*}
k\cdot F&\ast G^*(l\ot h)_{(1)} \ot  F\ast G^*(l\ot h)_{(2)}\\
& = k\cdot  (h_{(5)}\cdot l_{(2)})l_{(3)}S^{-1}(h_{(4)})S(l_{(6)}\dpu h_{(1)}) \ot  (h_{(6)}\cdot l_{(1)})l_{(4)}S^{-1}(h_{(3)})S(l_{(5)}\dpu h_{(2)})\\
& = ((k\cdot S^{-1}(h_{(4)})S(l_{(6)}\dpu h_{(1)}))\cdot l_{(3)}) \cdot (h_{(5)}\cdot l_{(2)}) \ot  (h_{(6)}\cdot l_{(1)})l_{(4)}S^{-1}(h_{(3)})S(l_{(5)}\dpu h_{(2)})\\
& = ((k\cdot S^{-1}(h_{(4)})S(l_{(6)}\dpu h_{(2)}))\cdot l_{(4)}) \cdot (h_{(5)}\cdot l_{(2)}) \ot  (h_{(6)}\cdot l_{(1)})l_{(3)}S^{-1}(h_{(3)})S(l_{(5)}\dpu h_{(1)})\\
& = ((k\cdot S^{-1}(h_{(4)})S(l_{(6)}\dpu h_{(3)}))\cdot l_{(4)}) \cdot (h_{(5)}\cdot l_{(3)}) \ot  (h_{(6)}\cdot l_{(1)})l_{(2)}S^{-1}(h_{(2)})S(l_{(5)}\dpu h_{(1)})\\
& = ((k\cdot S^{-1}(h_{(4)})S(l_{(6)}\dpu h_{(3)}))\cdot l_{(5)}) \cdot (h_{(5)}\cdot l_{(3)}) \ot  (h_{(6)}\cdot l_{(1)})l_{(2)}S^{-1}(h_{(2)})S(l_{(4)}\dpu h_{(1)})\\
& = ((k\cdot S^{-1}(h_{(4)})S(l_{(6)}\dpu h_{(3)}))\cdot l_{(5)}) \cdot (h_{(5)}\cdot l_{(4)}) \ot  (h_{(6)}\cdot l_{(1)})l_{(2)}S^{-1}(h_{(2)})S(l_{(3)}\dpu h_{(1)})\\
& = ((k\cdot S^{-1}(h_{(4)})S(l_{(5)}\dpu h_{(3)}))\cdot h_{(5)}) \cdot (l_{(4)}\dpu h_{(6)}) \ot  (h_{(7)}\cdot l_{(1)})l_{(2)}S^{-1}(h_{(2)})S(l_{(3)}\dpu h_{(1)})\\
&= k\cdot  (l_{(4)}\dpu h_{(6)})h_{(5)}S^{-1}(h_{(4)})S(l_{(5)}\dpu h_{(3)}) \ot  (h_{(7)}\cdot l_{(1)})l_{(2)}S^{-1}(h_{(2)})S(l_{(3)}\dpu h_{(1)})\\
&= k \ot  (h_{(3)}\cdot l_{(1)})l_{(2)}S^{-1}(h_{(2)})S(l_{(3)}\dpu h_{(1)})\\
&= k\ot F\ast G^*(l\ot h),
\end{align*}
where in the equalities~3)--6) we use the hypothesis. The fact that
$$
k\dpu F\ast G^*(l\ot h)_{(1)} \ot  F\ast G^*(l\ot h)_{(2)}= k\ot F\ast G^*(l\ot h)
$$
follows by a similar computation.
\end{proof}

\begin{remark} Proposition~\ref{el zocalo es un subalgebra de Hopf}, together with Proposition~\ref{H coconmutativo implica que H/Soc H es una skew -braza},  yields~\cite{R2}*{Proposition 8}.
\end{remark}

\section[The free Hopf \texorpdfstring{$q$}{q}-brace over a very strongly regular \texorpdfstring{$q$}{q}-cycle coalgebra]{The free Hopf \texorpdfstring{$\mathbf{q}$}{q}-brace over a very strongly regular \texorpdfstring{$\mathbf{q}$}{q}-cycle coalgebra}

In this section we construct the universal Hopf $q$-brace with bijective antipode and the universal Hopf skew-brace with bijective antipode of a very strongly regular $q$-cycle coalgebra.

\smallskip

Let $\mathcal{X}$ be a very strongly regular $q$-cycle coalgebra, let $Y\coloneqq X^{(\mathds{Z})}$ and let $i_j\colon X\to Y$ be the $j$-th canonical injection. We set $X_j\coloneqq i_j(X)$, and we write $x_j \coloneqq i_j(x)$ for each $x\in X$. Let $S\colon Y\to Y$ be the bijective map defined by $S(x_j)\coloneqq x_{j+1}$. We consider $Y$ endowed with the unique coalgebra structure such that $i_0\colon X\to Y$ is a coalgebra morphism and $S$ is an coalgebra antiautomorphism. Note that $i_j$ is a coalgebra morphism if $j$ is even and a coalgebra antimorphism if $j$ is odd. In this section we will freely use the notations introduced in Definition~\ref{strongly regular} and above Definition~\ref{muy fuertemente regular}. Motivated by Propositions~\ref{Las Hopf q-brazas son fuertemente regulares} and~\ref{compatibilidad de s con S}, we introduce bilinear operations $x\ot y\mapsto x\cdot y$ and $x\ot y\mapsto x\dpu y$ on $Y$ as follows: first we define $\cdot$ and $\dpu$ on $X_0$ by translation of structure through $i_0$. Next, for each $x,y\in X_0$ and all $m,n\in \mathds{Z}$, we define
\begin{align*}
&S^m(x)\cdot S^n(y)\coloneqq \begin{cases} S^m(x\cdot_i y) &\text{if $m$ is even and $n-m=2i$ with $i\in \mathds{Z}$,}\\ S^m(x^{\underline{y}_i}) &\text{if $m$ is even and $n-m=2i-1$ with $i\in \mathds{Z}$,}\\ S^m(x\ddiam_{\!i} y) &\text{if $m$ is odd and $n-m=2i$ with $i\in \mathds{Z}$,}\\ S^m({}_{\underline{y}_i}x) &\text{if $m$ is odd and $n-m=2i-1$ with $i\in \mathds{Z}$}
\end{cases}
\shortintertext{and}
&S^m(x)\dpu S^n(y)\coloneqq \begin{cases} S^m(x\dpu_i y) &\text{if $m$ is even and $n-m=2i$ with $i\in \mathds{Z}$,}\\ S^m(x_{\underline{y}_i}) & \text{if $m$ is even and $n-m=2i+1$ with $i\in	\mathds{Z}$,}\\ S^m(x\diam_i y) &\text{if $m$ is odd and $n-m=2i$ with $i\in \mathds{Z}$,}\\ S^m({}^{\underline{y}_i}x) &\text{if $m$ is odd and $n-m=2i+1$ with $i\in \mathds{Z}$.}\end{cases}
\end{align*}

\begin{remark}\label{. y : sobre Y son morfismos de coalgebra} A direct computation using Lemma~\ref{los pi, di, gpi y gdi son morfismos de coalgebra} proves that the maps $\cdot$ and $\dpu$ defined above are~coal\-gebra morphisms from $Y\ot Y^{\cop}$ to $Y$.
\end{remark}

\begin{remark}\label{la funcion s_2 en la suma directa} We introduce bilinear operations $x\ot y\mapsto x^y$ and $x\ot y\mapsto x_y$ on $Y$ as follows: first we define $x^y$ and $x_y$ on $X_0$ by translation of structure through $i_0$; next, for each $x,y\in X_0$ and all $m,n\in \mathds{Z}$, we define
\begin{align*}
&S^m(x)^{S^n(y)}\coloneqq \begin{cases}S^m(x^{\underline{y}_i}) &\text{if $m$ is even and $n-m=2i$ with $i\in \mathds{Z}$,}\\ S^m(x\cdot_{i-1} y) &\text{if $m$ is even and $n-m=2i-1$ with $i\in \mathds{Z}$,}\\ S^m({}_{\underline{y}_i}x) &\text{if $m$ is odd and $n-m=2i$ with $i\in \mathds{Z}$,}\\ S^m(x\ddiam_{\!i-1} y) &\text{if $m$ is odd and $n-m=2i-1$ with $i\in \mathds{Z}$}
\end{cases}
\shortintertext{and}
&S^m(x)_{S^n(y)}\coloneqq \begin{cases} S^m(x_{\underline{y}_i}) &\text{if $m$ is even and $n-m=2i$ with $i\in \mathds{Z}$,}\\ S^m(x\dpu_{i+1} y) &\text{if $m$ is even and $n-m=2i+1$ with $i\in \mathds{Z}$,}\\ S^m({}^{\underline{y}_i}x) &\text{if $m$ is odd and $n-m=2i$ with $i\in \mathds{Z}$,}\\ S^m(x\diam_{i+1} y) &\text{if $m$ is odd and $n-m=2i+1$ with $i\in \mathds{Z}$.}\end{cases}
\end{align*}
Equalities~\eqref{def de gp_i simplificado},~\eqref{def de gd_i simplificado},~\eqref{def de p_i simplificado} and~\eqref{def de d_i simplificado} applied to $\mathcal{X}$ and $\widehat{\mathcal{X}}$ (see the discussion above Definition~\ref{muy fuertemente regular}) show that $(Y,\cdot)$ satisfies   iden\-tities~\eqref{igualdades para cdot} and $(Y,\dpu)$ satisfies identities~\eqref{igualdades para dpu}.
\end{remark}

\begin{remark}\label{x,y-> x^y y x,y->x_y sobre Y son morfismos de coalgebra} By Proposition~\ref{UTIL} the maps $x\ot y\mapsto x^y$ and $x\ot y\mapsto x_y$, introduced in Remark~\ref{la funcion s_2 en la suma directa}, are coalgebra morphisms from $Y^2$ to $Y$.
\end{remark}

Motivated by Proposition~\ref{q magma coalgebra de tilde s^-1} and Remarks~\ref{notacion s1 y s2} and~\ref{notacion bar s1 y bar s2}, we define bilinear operations $x\ot y\mapsto {}^xy$ and $x\ot y\mapsto {}_xy$ on $Y$, by ${}^xy\coloneqq y_{(2)}\dpu x^{y_{(1)}}$ and ${}_xy\coloneqq y_{(1)}\cdot x_{y_{(2)}}$, respectively.

\begin{proposition}\label{ley de intercambio en Y} The maps $\cdot$ and $\dpu$ defined above Remark~\ref{. y : sobre Y son morfismos de coalgebra} satisfy identity~\eqref{intercambio . :} 
\end{proposition}
	
\begin{proof} With must prove that 
$$
S^m(x)_{(1)}\dpu S^n(y)_{(2)}\ot S^n(y)_{(1)}\cdot S^m(x)_{(2)} = S^m(x)_{(2)}\dpu S^n(y)_{(1)}\ot S^n(y)_{(2)}\cdot S^m(x)_{(1)}
$$
for all $x,y\in X$ and $m,n\in \mathds{Z}$. We consider separately four cases according to the parity of $m$ and $n$. Suppose for example that $m$ and $n$ are even and write $n-m = 2i$. We have
\begin{align*}
S^m(x)_{(1)}\dpu S^n(y)_{(2)}\ot S^n(y)_{(1)}\cdot S^m(x)_{(2)} & = (S^m\ot S^n)\bigl(x_{(1)}\dpu_i y_{(2)} \ot y_{(1)}\cdot_{-i} x_{(2)}\bigr)\\
& = (S^m\ot S^n)\bigl(x_{(2)}\dpu_i y_{(1)} \ot y_{(2)}\cdot_{-i} x_{(1)}\bigr)\\
&= S^m(x)_{(2)}\dpu S^n(y)_{(1)}\ot S^m(y)_{(2)}\cdot S^n(x)_{(1)}
\end{align*}
where the second equality holds by~\eqref{very strongly 1}. The other cases are similar. 
\end{proof}	

By Proposition~\ref{ley de intercambio en Y} and Remarks~\ref{d y p son morfismos de coalgears} and~\ref{. y : sobre Y son morfismos de coalgebra}, the tuple $\mathcal{Y}\coloneqq (Y,\cdot,\dpu)$ is a $q$-magma coalgebra. Moreover, by Remark~\ref{la funcion s_2 en la suma directa} we know that $\mathcal{Y}$ is regular. Thus, by Remark~\ref{notacion bar s1 y bar s2} and Proposition~\ref{s es morfismo de coalgebras}, the maps 
$$
s\colon Y^2\to Y^2\quad\text{and}\quad \bar{s}\colon (Y^{\cop})^2\to (Y^{\cop})^2, 
$$
defined by $\cramped{s(x\ot y)\coloneqq {}^{x_{(1)}}y_{(1)}\ot {x_{(2)}}^{y_{(2)}}}$ and $\cramped{\bar{s}(x\ot y)\coloneqq  {}_{x_{(1)}}y_{(1)}\ot {x_{(2)}}_{y_{(2)}}}$, respectively, are left non-degen\-erate coalgebra endomorphism.

\smallskip

Consider the tensor algebra $T(Y)$, endowed with the unique bialgebras structure such that the canonical map $i\colon Y\to T(Y)$ becomes a coalgebra morphism.

\begin{lemma}\label{extension de . y : a T(Y)} There exist unique extensions of the binary operations $x\ot y\mapsto x\cdot y$ and $x\ot y\mapsto x\dpu y$ on $Y$ to coalgebra morphisms from $T(Y)\ot T(Y)^{\cop}$ to $T(Y)$ such that $(T(Y),\cdot)$ and $(T(Y),\dpu)$ are $T(Y)^{\op}$-modules and identities~\eqref{condicion q-braza} are fulfilled. Moreover, $1\cdot h = 1\dpu h = \epsilon(h)1$ for all $h\in T(Y)$.
\end{lemma}

\begin{proof} Recall that $T(Y)=k\oplus Y \oplus Y^2\oplus \cdots$. We say that an element $h$ of $T(Y)$ has length $n$ if $h\in Y^n$. We extend the definitions of $\cdot$ and $\dpu$ to $T(Y)\ot Y$ by setting $1\cdot y=1\dpu y=\epsilon(y)1$ and recursively defining
$$
(hx)\cdot y \coloneqq (h\cdot (y_{(1)}\dpu x_{(2)}))(x_{(1)}\cdot y_{(2)})\quad\text{and}\quad (hx)\dpu y \coloneqq (h\dpu (y_{(1)}\cdot x_{(2)}))(x_{(1)}\dpu y_{(2)}),
$$
for all $h\in Y^n$ and $x,y\in Y$; and then, we extend $\cdot$ and $\dpu$ to $T(Y)^2$ by setting
$$
h\cdot 1=h\dpu 1=h,\quad h\cdot kx\coloneqq (h\cdot x)\cdot k\quad\text{and}\quad h\dpu kx\coloneqq (h\dpu x)\dpu k,
$$
for all $h\in T(Y)$, $k\in Y^n$ and $x\in Y$. Since these definitions are forced by the conditions required in the statement, they are unique. We leave to the reader to check that $T(Y)$ is a $T(Y)^{\op}$-module via $\cdot$ and $\dpu$, that these maps are compatible with the counits and that $1\cdot h = 1\dpu h = \epsilon(h)1$. Next we prove that
\begin{equation}
(h\cdot k)_{(1)}\ot (h\cdot k)_{(2)}=h_{(1)}\cdot k_{(2)}\ot h_{(2)}\cdot k_{(1)} \quad\text{and}\quad  (h\dpu k)_{(1)}\ot (h\dpu k)_{(2)}=h_{(1)}\dpu k_{(2)}\ot h_{(2)}\dpu k_{(1)}\label{primera ec}
\end{equation}
for all $h,k\in T(Y)$. Note that for $h,k\in Y$, these equalities are true by Remark~\ref{. y : sobre Y son morfismos de coalgebra}. Since both identities can be proven in a similar way, we only prove the first one. Assume that it is true for $h\in Y^n$ and $k\in Y$. Then, by the inductive hypothesis and the fact that~\eqref{intercambio . :} holds on $Y$, we have:
\begin{align*}
(hx\cdot y)_{(1)}\ot (hx\cdot y)_{(2)}& = (h\cdot (y_{(1)}\dpu x_{(2)}))_{(1)}(x_{(1)}\cdot y_{(2)})_{(1)} \ot (h\cdot (y_{(1)}\dpu x_{(2)}))_{(2)}(x_{(1)}\cdot y_{(2)})_{(2)}\\
& = (h_{(1)}\cdot (y_{(2)}\dpu x_{(3)}))(x_{(1)}\cdot y_{(4)}) \ot (h_{(2)}\cdot (y_{(1)}\dpu x_{(4)}))(x_{(2)}\cdot y_{(3)})\\
& = (h_{(1)}\cdot (y_{(3)}\dpu x_{(2)}))(x_{(1)}\cdot y_{(4)}) \ot (h_{(2)}\cdot (y_{(1)}\dpu x_{(4)}))(x_{(3)}\cdot y_{(2)})\\
& = h_{(1)}x_{(1)}\cdot y_{(2)}\ot h_{(2)}x_{(2)}\cdot y_{(1)},
\end{align*}
for all $x,y\in Y$. This proves the equality for all $h\in T(Y)$ and $k\in Y$. Assume now that this identity is true for all $h\in T(Y)$ and $k\in Y^n$. Then, we have
$$
(h\cdot kx)_{(1)}\ot (h\cdot kx)_{(2)} = (h_{(1)}\cdot x_{(2)})\cdot k_{(2)}\ot (h_{(2)}\cdot x_{(1)})\cdot k_{(1)}= h_{(1)}\cdot (kx)_{(2)}\ot h_{(2)}\cdot (kx)_{(1)},
$$
for all $x\in Y$, which finishes the proof of the first equality in~\eqref{primera ec}. We next check the first identity in~\eqref{condicion q-braza}, and leave the second one, which is similar, to the reader. Clearly, this is fulfilled when $h=1$, $k=1$ or $l=1$. Assume that it is true for all $h\in T(Y)$, $k\in Y^n$ and $l\in Y$. Then,
\begin{align*}
h(kx)\cdot l  &= ((hk)\cdot (l_{(1)}\dpu x_{(2)}))(x_{(1)}\cdot l_{(2)})\\
&=  \bigl(h\cdot ((l_{(1)}\dpu x_{(3)})\dpu k_{(2)})\bigr)\bigl(k_{(1)}\cdot (l_{(2)}\dpu x_{(2)})\bigr)\bigl(x_{(1)}\cdot l_{(3)}\bigr)\\
&=  \bigl(h\cdot (l_{(1)}\dpu(kx)_{(2)})\bigr)\bigl((kx)_{(1)}\cdot l_{(2)}\bigr).
\end{align*}
for all $x\in Y$. So, the first identity in~\eqref{condicion q-braza} holds for all $h,k\in T(Y)$ and $l\in Y$. Assume now that it holds for all $h\in T(Y)$, $k\in Y$ and $l\in Y^n$. Then, for all $x\in Y$,
\begin{align*}
hk\cdot xl & = \bigl((h\cdot (l_{(1)}\dpu k_{(2)}))(k_{(1)}\cdot l_{(2)})\bigr)\cdot x\\
&=\bigl((h\cdot (l_{(1)}\dpu k_{(3)}))\cdot (x_{(1)}\dpu (k_{(2)}\cdot l_{(2)}))\bigr) \bigl((k_{(1)}\cdot l_{(3)})\cdot x_{(2)} \bigr)\\
&= \bigl(h\cdot  (x_{(1)}\dpu (k_{(2)}\cdot l_{(2)}))(l_{(1)}\dpu k_{(3)})\bigr)\bigl(k_{(1)}\cdot x_{(2)} l_{(3)}\bigr)\\
&= \bigl(h\cdot ((x_{(1)}l_{(1)})\dpu k_{(2)})\bigr)\bigl(k_{(1)}\cdot x_{(2)} l_{(2)}\bigr),
\end{align*}
which proves the first identity in~\eqref{condicion q-braza} for all  $h,l\in T(Y)$ and $k\in Y$.  Finally, assume that this is true for all  $h,l\in T(Y)$ and $k\in Y^n$. Then, for all $x\in Y$, we have
\begin{align*}
h(kx)\cdot l & = \bigl(hk\cdot (l_{(1)}\dpu x_{(2)})\bigr)(x_{(1)}\cdot l_{(2)})\\ 
& = \bigl(h\cdot ((l_{(1)}\dpu x_{(3)})\dpu k_{(2)})\bigr)\bigl(k_{(1)}\cdot (l_{(2)}\dpu x_{(2)})\bigr)\bigl(x_{(1)}\cdot l_{(3)}\bigr)\\
& =\bigl(h\cdot (l_{(1)}\dpu(kx)_{(2)})\bigr)\bigl((kx)_{(1)}\cdot l_{(2)}\bigr),
\end{align*}
which finishes the proof of the first identity in~\eqref{condicion q-braza}.
\end{proof}

\begin{lemma}\label{(T(Y),.,:) es un qq-magma coalgebra regular } $T(\mathcal{Y})\coloneqq (T(Y),\cdot ,\dpu)$ is a regular $q$-magma coalgebra.
\end{lemma}

\begin{proof} First we prove that $T(\mathcal{Y})$ is a $q$-magma coalgebra. By Remark~\ref{d y p son morfismos de coalgears} and Lemma~\ref{extension de . y : a T(Y)}, for this it suffices to show that identity~\eqref{intercambio . :} holds for all $x,y\in T(Y)$. By hypothesis we know that this identity holds for all $x,y\in Y$. Moreover, a direct computation shows that it also holds when $x=1$ and $y\in T(Y)$, and when $y=1$ and $x\in T(Y)$. Assume that
\begin{equation}\label{h en T(Y), y en Y}
h_{(1)}\dpu y_{(2)}\ot y_{(1)}\cdot h_{(2)} = h_{(2)}\dpu y_{(1)}\ot y_{(2)}\cdot h_{(1)}
\end{equation}
for all $h\in Y^n$ and $y\in Y$. Then, by Lemma~\ref{extension de . y : a T(Y)} and the inductive hypothesis, for all $h\in Y^n$ and $x,y\in Y$, we have
\begin{align*}
h_{(1)}x_{(1)}\dpu y_{(2)}\ot y_{(1)}\cdot h_{(2)}x_{(2)} &=  (h_{(1)}\dpu (y_{(2)}\cdot x_{(2)})) (x_{(1)}\dpu y_{(3)})\ot (y_{(1)}\cdot x_{(3)})\cdot h_{(2)}\\
&=  (h_{(2)}\dpu (y_{(1)}\cdot x_{(3)})) (x_{(1)}\dpu y_{(3)})\ot (y_{(2)}\cdot x_{(2)})\cdot h_{(1)}\\
&=  (h_{(2)}\dpu (y_{(1)}\cdot x_{(3)})) (x_{(2)}\dpu y_{(2)})\ot (y_{(3)}\cdot x_{(1)})\cdot h_{(1)}\\
&=h_{(2)}x_{(2)}\dpu y_{(1)}\ot y_{(2)}\cdot h_{(1)}x_{(1)},
\end{align*}
which proves~\eqref{h en T(Y), y en Y} for all $h\in T(Y)$ and $y\in Y$. Finally, assume  that $h_{(1)}\dpu k_{(2)}\ot k_{(1)}\cdot h_{(2)}= h_{(2)}\dpu k_{(1)}\ot k_{(2)}\cdot h_{(1)}$ for all $h\in T(Y)$ and $k\in Y^n$. Then, for all $h\in T(Y)$, $k\in Y^n$ and $y\in Y$
\begin{align*}
h_{(1)}\dpu k_{(2)}y_{(2)}\ot k_{(1)}y_{(1)}\cdot h_{(2)} &=(h_{(1)}\dpu y_{(3)})\dpu k_{(2)}\ot (k_{(1)}\cdot (h_{(2)}\dpu y_{(2)}))(y_{(1)}\cdot h_{(3)})\\
&=(h_{(2)}\dpu y_{(2)})\dpu k_{(1)}\ot (k_{(2)}\cdot (h_{(1)}\dpu y_{(3)}))(y_{(1)}\cdot h_{(3)})\\
&=(h_{(3)}\dpu y_{(1)})\dpu k_{(1)}\ot (k_{(2)}\cdot (h_{(1)}\dpu y_{(3)}))(y_{(2)}\cdot h_{(2)})\\
&=h_{(2)}\dpu k_{(1)}y_{(1)}\ot k_{(2)}y_{(2)}\cdot h_{(1)},
\end{align*}
for all $h\in T(Y)$, $k\in Y^n$ and $y\in Y$, which proves~\eqref{intercambio . :} on $T(\mathcal{Y})$. For the purpose of checking that $T(\mathcal{Y})$ is regular, we first extend the maps introduced in Remark~\ref{la funcion s_2 en la suma directa} to~$T(Y)\ot Y$, by setting $1^y=1_y=\epsilon(y)1$ and recursively defining
$$
(hx)^y \coloneqq \bigl(h^{{}^{x_{(1)}}y_{(1)}}\bigr)\bigl({x_{(2)}}^{y_{(2)}}\bigr)\quad\text{and}\quad (hx)_y \coloneqq \bigl(h_{{}_{x_{(1)}}y_{(1)}}\bigr)\bigl({x_{(2)}}_{y_{(2)}}\bigr),
$$
for all $h\in Y^n$ and $x,y\in Y$; and then, we extend these maps to $T(Y)^2$ by setting $h^1=h_1=h$, $h^{kx}\coloneqq (h^k)^x$ and $h_{kx}\coloneqq (h_k)_x$ for all $h\in T(Y)$, $k\in Y^n$ and $x\in Y$.
	
Next we prove that $T(\mathcal{Y})$ is left regular, and leave the proof that it is right regular, which is similar, to the reader. By Remark~\ref{la funcion s_2 en la suma directa} we know that $x^{y_{(1)}}\cdot y_{(2)}=\epsilon(y)x$ for all $x,y\in Y$. Moreover, $h^1\cdot 1=h=\epsilon(1)h$ and $1^{h_{(1)}}\cdot h_{(2)}= \epsilon(h)1$ for all $h\in T(Y)$. Assume that $h^{y_{(1)}}\cdot y_{(2)}=\epsilon(h)y$ for all $h\in Y^n$ and $y\in Y$. Then, by the first identity  in~\eqref{condicion q-braza} and the discussion above Proposition~\ref{ley de intercambio en Y}, for all $h\in Y^n$ and $x,y\in Y$, we have
$$
(hx)^{y_{(1)}}\cdot y_{(2)} = \bigl(\bigl(h^{{}^{x_{(2)}}y_{(2)}}\bigr) \bigl({x_{(1)}}^{y_{(1)}}\bigr)\bigr)\cdot y_{(3)} =\bigl(\bigl(h^{y_{(4)}\dpu {x_{(3)}}^{y_{(3)}}}\bigr)\cdot \bigl(y_{(5)}\dpu {x_{(2)}}^{y_{(2)}}\bigr)\bigr)\bigl({x_{(1)}}^{y_{(1)}}\cdot y_{(6)}\bigr)=\epsilon(y)hx,
$$
and so $h^{y_{(1)}}\cdot y_{(2)}=\epsilon(y)h$ for all $h\in T(Y)$ and $y\in Y$. Assume now that $h^{k_{(1)}}\cdot k_{(2)}=\epsilon(k)h$ for all $h\in T(Y)$ and $k\in Y^n$. Then, for all $h\in T(Y)$, $k\in Y^n$ and $y\in Y$, we have
$$
h^{k_{(1)}y_{(1)}}\cdot k_{(2)}y_{(2)} =\bigl(\bigl(h^{k_{(1)}}\bigr)^{y_{(1)}}\cdot y_{(2)}\bigr)\cdot k_{(2)}=\epsilon(y)h^{k_{(1)}}\cdot k_{(2)}=\epsilon(ky)h,
$$
and hence, $h^{k_{(1)}}\cdot k_{(2)}=\epsilon(k)h$ for all $h,k\in T(Y)$. It remains to check that $(h\cdot k_{(1)})^{k_{(2)}}=\epsilon(k)h$ for all $h,k\in T(Y)$. Again by Remark~\ref{la funcion s_2 en la suma directa} we know that $(x\cdot y_{(1)})^{y_{(2)}}$ for all $x,y\in Y$. Moreover, $(h\cdot 1)^1=h= \epsilon(1)h$ and $(1\cdot h_{(1)})^{h_{(2)}}= \epsilon(h)1$ for all $h\in T(Y)$. Assume that $(h\cdot {y_{(1)}})^{y_{(2)}}=\epsilon(y)h$ for all $h\in Y^n$ and $y\in Y$. Then, for all $h\in Y^n$ and $x,y\in Y$, we have
\begin{align*}
\bigl(hx\cdot y_{(1)}\bigr)^{y_{(2)}} &= \bigl((h\cdot (y_{(1)}\dpu x_{(2)}))(x_{(1)}\cdot y_{(2)})\bigr)^{y_{(3)}}\\
&= \bigl((h\cdot (y_{(1)}\dpu x_{(3)}))^{{}^{x_{(1)}\cdot y_{(3)}}y_{(4)}}\bigr)\bigl((x_{(2)}\cdot y_{(2)})^{y_{(5)}}\bigr)\\
&= \bigl((h\cdot (y_{(1)}\dpu x_{(3)}))^{y_{(3)}\dpu x_{(1)}}\bigr)\bigl((x_{(2)}\cdot y_{(2)})^{y_{(4)}}\bigr)\\
&= \bigl((h\cdot (y_{(1)}\dpu x_{(3)}))^{y_{(2)}\dpu x_{(2)}}\bigr)\bigl((x_{(1)}\cdot y_{(3)})^{y_{(4)}}\bigr)\\
&=\epsilon(y)hx,
\end{align*}
and hence $(h\cdot {y_{(1)}})^{y_{(2)}}=\epsilon(y)h$ for all $h\in T(Y)$ and $y\in Y$. Assume now that $(h\cdot {k_{(1)}})^{k_{(2)}}=\epsilon(k)h$ for all $h\in T(Y)$ and $k\in Y^n$. Then, for all $h\in T(Y)$, $k\in Y^n$ and $y\in Y$, we have
$$
(h\cdot k_{(1)}y_{(1)})^{k_{(2)}y_{(2)}} =(((h\cdot y_{(1)})\cdot k_{(1)})^{k_{(2)}})^{y_{(2)}}=\epsilon(k)(h\cdot y_{(1)})^{y_{(2)}}=\epsilon(ky)h,
$$
and so, $(h\cdot k_{(1)})^{k_{(2)}}=\epsilon(k)x$ for all $h,k\in T(Y)$, which finishes the proof that $T(\mathcal{Y})$ is regular.
\end{proof}

\begin{theorem}\label{q-Hopf algebra universal 1} Let $\mathcal{X}$ be a very strongly regular $q$-cycle coalgebra. There exists a Hopf $q$-brace $\mathcal{H}_{\mathcal{X}}$ with bijective antipode and a $q$-cycle coalgebra morphism $\iota_{\mathcal{X}}\colon \mathcal{X}\to \mathcal{H}_{\mathcal{X}}$, which is universal in the following sense: given a Hopf $q$-brace with bijective antipode $\mathcal{L}$ and a $q$-cycle coalgebra morphism $f\colon \mathcal{X}\to \mathcal{L}$ there exists a unique Hopf $q$-brace morphism $\bar{f}\colon \mathcal{H}_{\mathcal{X}}\to \mathcal{L}$ such that $f=\bar{f}\xcirc \iota_{\mathcal{X}}$.
\end{theorem}

\begin{proof} By Lemmas~\ref{extension de . y : a T(Y)} and~\ref{(T(Y),.,:) es un qq-magma coalgebra regular } we know that $T(\mathcal{Y})=(T(Y)),\cdot,\dpu)$ is a regular $q$-magma coalgebra, that $T(Y)$ is a bialgebras, that $(T(Y),\cdot)$ and $(T(Y),\dpu)$ are $T(Y)^{\op}$-right modules and that $\cdot $ and~$\dpu$ satisfy identities~\eqref{condicion q-braza}. Moreover, it is clear that the coalgebra~iso\-morphism $S\colon Y\to Y$ can be extended in a unique way to a bialgebras isomorphism (also called~$S$) from $T(Y)$ to $T(Y)^{\op\cop}$. Let $I$ be the minimal bi-ideal of $T(Y)$ such that:
	
\begin{itemize}[itemsep=0.7ex, topsep=1.0ex, label=-]
		
\item $h\cdot k$, $k\cdot h$, $h\dpu k$ and $k\dpu h$ belong to $I$, for all $h\in T(Y)$ and $k\in I$,
		
\item $S(h_{(1)})h_{(2)}- \epsilon(h) 1\in I$ and $h_{(1)}S(h_{(2)})- \epsilon(h) 1\in I$, for all $h\in T(Y)$,
		
\item $(h \cdot k_{(1)})\cdot (l\dpu k_{(2)})-(h\cdot l_{(2)})\cdot (k\cdot l_{(1)})\in I$, for all $h,k,l\in T(Y)$,
		
\item $(h \cdot k_{(1)})\dpu (l\cdot k_{(2)})-(h\dpu l_{(2)})\cdot (k\dpu l_{(1)})\in I$, for all $h,k,l\in T(Y)$,
		
\item $(h \dpu  k_{(1)})\dpu (l\dpu k_{(2)}) -(h\dpu l_{(2)})\dpu (k\cdot l_{(1)})\in I$, for all $h,k,l\in T(Y)$.
		
\end{itemize}
Let $H_{\mathcal{X}}\coloneqq T(Y)/I$. Then $H_{\mathcal{X}}$ is a Hopf algebra with bijective antipode and the operations $\cdot$ and $\dpu$ on $T(Y)$ induce operations $\cdot$ and $\dpu$ on $H_{\mathcal{X}}$, such that $\mathcal{H}_{\mathcal{X}}\coloneqq (H_{\mathcal{X}},\cdot,\dpu)$ is a Hopf $q$-brace with bijective antipode. Moreover, the map $\iota_{\mathcal{X}}\coloneqq \mathcal{X} \hookrightarrow \mathcal{Y} \hookrightarrow T(\mathcal{Y}) \twoheadrightarrow \mathcal{H}_{\mathcal{X}}$, where all the arrows are the canonical ones, is a $q$-cycle coalgebra morphism which is universal in the sense established in the statement. In fact, there is a unique extension of $f$ to a coalgebra morphism $f_1\colon Y\to L$ such that $f_1(S^i(x))=S^i_L(f(x))$ for all $x\in X$ and $i\in \mathds{Z}$. By the definitions of $\cdot$ and $\dpu$ on $Y$ and Proposition~\ref{indecente}, applied to $\mathcal{X}$ and $\widehat{\mathcal{X}}$ (see Definition~\ref{muy fuertemente regular}), the map $f_1$ is compatible with the operations~$\cdot$ and~$\dpu$. By the universal property of the tensor algebra, $f_1$ extends to a unique bialgebras map $f_2\colon T(Y)\to L$, which, by the fact that identities~\eqref{condicion q-braza} are valid on $T(Y)$ and $L$, is also compatible with the operations $\cdot$ and $\dpu$. Finally, since $L$ is a Hopf algebra and the identities in Definition~\ref{def: q cycle coalgebra} are satisfied in $L$, this map factorizes through $\mathcal{H}_{\mathcal{X}}$.
\end{proof}

\begin{remark}\label{la skew braza de Hopf universal} Let $\mathcal{X}$ be a very strongly regular $q$-cycle coalgebra. The quotient $\mathcal{H}_{\mathcal{X}}/[\mathcal{H}_{\mathcal{X}},\mathcal{H}_{\mathcal{X}}]_q$ is a~Hopf skew-brace with bijective antipode (Example~\ref{q-conmutadores}) and the map $\tilde{\iota}_{\mathcal{X}}\coloneqq p\xcirc \iota_{\mathcal{X}}$, where $p\colon \mathcal{H}_{\mathcal{X}}\to \mathcal{H}_{\mathcal{X}}/[\mathcal{H}_{\mathcal{X}},\mathcal{H}_{\mathcal{X}}]_q$ is the canonical surjection, is universal in the following sense: given a Hopf skew-brace with bijective antipode $\mathcal{L}$ and a $q$-cycle coalgebra morphism $f\colon \mathcal{X}\to \mathcal{L}$ there exists a unique Hopf skew-brace morphism $\bar{f}\colon \mathcal{H}_{\mathcal{X}}\to \mathcal{L}$, such that $f=\bar{f}\xcirc \tilde{\iota}_{\mathcal{X}}$.
\end{remark}

\begin{remark} Let $s\colon Y\times Y\to Y\times Y$ be a set theoretic bijective non-degenerate solution of the braid equation, let $ks\colon kY\ot kY\to kY\ot kY$ be the linearization of $s$, and let $\mathcal{X}$ be the non-degenerate $q$-cycle coalgebra associated with $ks$ according to Corollary~\ref{correspondencia entr q-cycle coalgebras y soluciones}. In \cite{LYZ}*{Theorem~4}, the authors construct a group $G(Y,s)$, a braiding operator $s^G$ on $G(Y,s)$ and a map $\iota\colon Y\to G(Y,s)$, such that $s^G\xcirc (\iota\times \iota) = (\iota\times \iota) \xcirc s$, which is universal. By Remark~\ref{equivalencia cat braid oper...} the linearization $ks^G$ of $s^G$ is the braiding operator associated with the skew-brace $\mathcal{H}_{\mathcal{X}}/[\mathcal{H}_{\mathcal{X}},\mathcal{H}_{\mathcal{X}}]_q$ introduced in Remark~\ref{la skew braza de Hopf universal}. Thus the underlying Hopf algebra of $\mathcal{H}_{\mathcal{X}}/[\mathcal{H}_{\mathcal{X}},\mathcal{H}_{\mathcal{X}}]_q$ is the group algebra of $G(Y,s)$.
\end{remark}

\end{document}